\documentclass{svproc}
\usepackage{url}

\usepackage{amsmath}
\usepackage{amssymb}
\usepackage{graphicx}
\usepackage{float}
\usepackage{lipsum}
\usepackage{mathtools}

\usepackage{mathrsfs}

\def\eps{\varepsilon}

\begin{document}
\newtheorem{proofyyy}{Proof of Lemma \ref{lemma:stokes}}

\mainmatter

\title{Stokes phenomenon arising in the confluence of the Gauss hypergeometric equation}
\titlerunning{Stokes phenomenon in the confluent  hypergeometric equation}
\author{Calum Horrobin and Marta Mazzocco}
\authorrunning{C. Horrobin and M. Mazzocco.} 

\institute{University of Birmingham,\\
\email{m.mazzocco@bham.ac.uk}}

\maketitle

\begin{abstract}
In this paper we study the Gauss and Kummer hypergeometric equations in-depth. In particular, we focus on the confluence of two regular singularities of the Gauss hypergeometric equation to produce
the Kummer hypergeometric equation with an irregular singularity at infinity. We show how to pass from solutions with power-like behaviour which are analytic in disks, to solutions with exponential behaviour which are analytic in sectors and have divergent asymptotics. We explicitly calculate the Stokes matrices of the confluent system in terms of the monodromy data, specifically the connection matrices, of the original system around the merging singularities. 
\end{abstract}

\section{Introduction}
This paper studies the Gauss hypergeometric differential equation,
\begin{eqnarray} x(1-x) \ \frac{d^{2}y}{dx^{2}} + (\gamma - (\alpha+\beta+1)x) \ \frac{dy}{dx} - \alpha \beta \ y = 0, \label{eq:gauss} \end{eqnarray}
where $x \in \mathbb{C}$, and the Kummer confluent hypergeometric differential equation,
\begin{eqnarray} z \ \frac{d^{2}\tilde{y}}{dz^{2}} + (\gamma - z ) \ \frac{d\tilde{y}}{dz} - \beta \ \tilde{y} = 0, \label{eq:kummer} \end{eqnarray}
where $z \in \mathbb{C}$. 

For brevity, in this paper, these equations are simply called Gauss equation and Kummer equation respectively.

The aim of the paper is to give rigour to the confluence of two regular singularities of the Gauss equation to produce
the Kummer equation with an irregular singularity at infinity. In particular, the monodromy data of the confluent equation (Kummer), including Stokes data, are produced as limits of the monodromy data of the original equation (Gauss) using explicit formulae. 

One of the main difficulties addressed in this paper is how to make sense of the confluence limits by understanding how to pass from the solutions of the original system to the solutions of the confluent system. This is a non-trivial question because it involves passing from a solution with power-like behaviour which converges in a disk to solutions with exponential behaviour which are analytic in a sector and asymptotic to a divergent series. 

The procedure of this paper is based on an existence theorem by Glutsyuk \cite{glutsyuk}. Essentially, this states that there exist certain diagonal matrices $K_{\eps}$ and $K_{-\eps}$ such that the limit,
\[\lim_{\eps \rightarrow 0} K_{-\eps}^{-1} \ C \ K_{\eps},\]
where $C$ is the connection matrix between the merging simple poles of the original system, exists. Moreover, this limit gives one of the Stokes matrices if $\eps\to 0$ is taken along a certain ray. However, this existence theorem does not prescribe how to calculate the diagonal matrices $K_{\eps}$ and $K_{-\eps}$. The main result of this paper is to calculate such diagonal matrices and thus produce both the Stokes matrices in terms of limits of the connection matrix of the original equation explicitly. In particular calculate how one Stokes matrix is produced as limit along a certain ray and the other one by the limit along the opposite ray.

Despite the fact that the analytic theory of the Gauss and Kummer equations has been developed more than a hundred years ago, the question of producing the Stokes data of the Kummer equation in terms of limits of monodromy data of the Gauss one has only been approached rather recently \cite{lambert,watanabe}. In particular, in \cite{watanabe}, the Mellin-Barnes integral representations of the solutions of Kummer equation are produced as limits of the ones for the Gauss equation, and then the Stokes data are deduced from the Mellin-Barnes integral representations (this last calculation is reported here in Appendix B for completeness). 
In  \cite{lambert}, the confluence problem is solved by observing that one of the Fuchsian singularities remains Fuchsian under the confluence, so that the corresponding local fundamental matrix of the Gauss equation admits an analytic limit under the confluence, thus allowing to compute explicitly the monodromy of the Kummer equation around $0$. The Stokes matrices are then determined by the fact that loops around $0$ are homotopic to loops around $\infty$ in the Riemann sphere with two punctures. 

The approach of the current paper does not require closed form expressions such as Mellin-Barnes integrals.  Indeed, in \cite{HM}, we use this procedure to calculate the Stokes matrices of the linear problem associated to the fifth Painlev\'e equation (and its higher order analogues ) in terms of limits of the connection matrix between $1$ and $\infty$ in the linear problem associated to the sixth Painlev\'e equation (and its higher order analogues) for which closed form fundamental matrices are unknown.
 
Another advantage of the procedure of the current paper is that it does not rely on the existence of an additional simple pole which survives the confluence limit, and therefore it can be applied to the confluence from the Bessel differential equation to the Airy one for example, or even more ambitiously, in the confluence from the fifth to the third Painlev\'e equation - this challenging work is postponed to subsequent publications.

This paper is organised as follows: 
In Sections \ref{sec:gg} and \ref{sec:kummer}, the authors remind some background on the Gauss and Kummer hypergeometric differential equations respectively. In Section \ref{sec:hgconf} the confluence procedure is explained, and the main result of this paper, Theorem \ref{main:top} is proved.
In appendices A and B, the classical derivation of the monodromy data for the Gauss and Kummer hypergeometric differential equations respectively are derived using Mellin-Barnes integrals.

\vskip 3mm
{\it This paper is inspired by some of the facets of Nalini's mathematical taste and style because to tackle a seemingly simple problem it requires an unexpected depth that opens a Pandora's box of beautiful mathematical problems. For this reason, we wish to dedicate this paper to her.} [Calum Horrobin and Marta Mazzocco]

\vskip 3mm
{\it I wish to thank Nalini for her friendship of more than twenty years. Throughout her career, Nalini has mentored, supported and sponsored a huge number of 
early career mathematicians, some formally as her PhD students and post docs, others informally, like myself and many others.} [Marta Mazzocco]

\vskip 2mm \noindent{\bf Acknowledgements.} We thank D. Guzzetti for many helpful conversations. This research was supported by the EPSRC Research Grant $EP/P021913/1$ and EPSRC DTA allocation to the Mathematical Sciences Department at Loughborough University.

\section{Gauss hypergeometric differential equation} \label{sec:gg} 

Throughout the paper we work in the non-resonance assumption: $\gamma$, $\gamma-\alpha-\beta$, $\alpha-\beta \not \in \mathbb{Z}$.

To define monodromy data, it is easier to deal with a system of first order ODEs by using the following trivial lemma:

\begin{lemma} \label{lemma:hglemma} Under the assumptions $\alpha \neq 0$, $\gamma \neq \beta \neq 1$ and $\alpha \neq \beta-1$, the matrix 
\begin{eqnarray}Y(x) = \left( \begin{array}{cc} y_{1}(x) & y_{2}(x) \\ \Psi \left( y_{1} , y_{1}' ; x \right) & \Psi \left( y_{2}, y_{2}' ; x \right) \end{array} \right) , \label{eq:Y} \end{eqnarray}
where
\begin{eqnarray} \Psi \left( y_{k}, y_{k}' ; x \right) = \frac{ \alpha \left(\beta - \gamma + (\alpha+1-\beta) x \right) y_{k}(x) + x(x-1)(\alpha+1-\beta) y_{k}'(x)}{\alpha (\beta-1)(\beta-\gamma)}, \label{eq:psi1} \end{eqnarray}
is a fundamental solution of the equation 
\begin{eqnarray}\frac{{\rm d} Y}{{\rm d} x} = \left( \frac{A_{0}}{x} + \frac{A_{1}}{x-1}\right)Y, \label{eq:hg1} \end{eqnarray}
\begin{eqnarray} &\ & A_{0} = \frac{1}{\alpha+1-\beta} \left( \begin{array}{cc} \alpha(\beta-\gamma) & \alpha(1-\beta)(\beta-\gamma) \\ \alpha+1-\gamma & (1-\beta)(\alpha+1-\gamma) \end{array} \right), \nonumber \\
&\ & A_{1} = \frac{1}{\alpha+1-\beta} \left( \begin{array}{cc} \alpha(\gamma-\alpha-1) & \alpha(\beta-1)(\beta-\gamma) \\ \gamma-\alpha-1 & (\beta-1)(\beta-\gamma) \end{array} \right), \nonumber \end{eqnarray}
if and only if $y_{1}(x)$ and $y_{2}(x)$ are linearly independent solutions of Gauss hypergeometric equation (\ref{eq:gauss}).
\end{lemma}

So, from now on, we stick to the system of first order ODEs \eqref{eq:hg1} .

We define the following disks with chosen branches, as illustrated in Figure \ref{fig:123} below:
\begin{align} \Omega_{0} &= \left\{ x : |x|<1 , \ -\pi \leq \text{arg}(x) < \pi \right\}, \nonumber \\
\Omega_{1} &= \left\{ x : |x-1|<1, \ - \pi \leq \text{arg}(1-x) < \pi \right\}, \nonumber \\
\Omega_{\infty} &= \left\{ x : |x| > 1 , \ -\pi \leq \text{arg}(-x) < \pi \right\}, \nonumber \end{align}

\begin{figure}\begin{center}
\includegraphics[scale=.8]{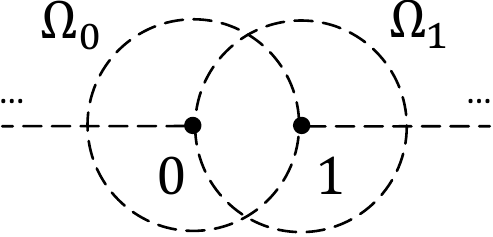}
\caption{\label{fig:123}Chosen disks with branch cuts. Note that  $\Omega_{\infty}$ is a disk in the complement of $\overline\Omega_0\cup \overline\Omega_1$.}
\end{center} \end{figure} 

It is well-known that the solutions of equation (\ref{eq:gauss}) are expressible in terms of Gauss hypergeometric $\ _{2}F_{1}$ series, in particular the following 
 three pairs of linearly independent local solutions $y_{1}^{(k)}(x)$ and $y_{2}^{(k)}(x)$ of (\ref{eq:gauss}) defined in the neighbourhoods $\Omega_{k}$ form a basis around around each singular point: 
\begin{align} &\begin{matrix*}[l] y_{1}^{(0)}(x) = x^{1-\gamma} \ _{2}F_{1} \left( \begin{array}{c} \alpha+1-\gamma , \ \beta+1-\gamma \\ 2-\gamma \end{array} ; x \right), \\ y_{2}^{(0)}(x) = \ _{2}F_{1} \left(\begin{array}{c} \alpha , \ \beta \\ \gamma \end{array} ; x \right), \end{matrix*} &&x \in \Omega_{0}, \label{eq:y0} \\
&\begin{matrix*}[l] y_{1}^{(1)}(x) = (1-x)^{\gamma-\alpha-\beta} \ _{2}F_{1} \left(\begin{array}{c} \gamma-\alpha , \ \gamma-\beta \\ \gamma+1-\alpha-\beta \end{array} ; 1-x \right), \\
y_{2}^{(1)}(x) = \ _{2}F_{1} \left(\begin{array}{c} \alpha , \ \beta \\ \alpha+\beta+1-\gamma \end{array} ; 1-x \right), \end{matrix*} &&x \in \Omega_{1}, \label{eq:y1} \\
&\begin{matrix*}[l] y_{1}^{(\infty)}(x) = (-x)^{-\alpha} \ _{2}F_{1} \left( \begin{array}{c} \alpha , \ \alpha+1-\gamma \\ \alpha+1-\beta \end{array} ; x^{-1} \right), \\
y_{2}^{(\infty)}(x) = (-x)^{-\beta} \ _{2}F_{1} \left( \begin{array}{c} \beta , \ \beta+1-\gamma \\ \beta+1-\alpha \end{array} ; x^{-1} \right), \end{matrix*} &&x \in \Omega_{\infty}. \label{eq:yinf} \end{align}

\begin{lemma} \label{lemma:cup} The following local fundamental solutions of the matrix hypergeometric equation (\ref{eq:hg1}) have the following form
\begin{align} Y^{(0)}(x) &= R_{0} G_{0}(x) x^{\Theta_{0}}, &&x \in \Omega_{0}, \label{eq:ak} \\
Y^{(1)}(x) &= R_{1} G_{1}(x) (1-x)^{\Theta_{1}}, &&x \in \Omega_{1}, \label{eq:ak2} \\ 
Y^{(\infty)}(x) &= R_{\infty} G_{\infty}(x) (-x)^{-\Theta_{\infty}}, &&x \in \Omega_{\infty}, \label{eq:infinity} \end{align}
where $R_{k}$ and $\Theta_{k}$ are the following matrices:
\begin{align} R_{0} = \left(\begin{array}{cc} 1 & 1 \\ \frac{\alpha+1-\gamma}{\alpha(\beta-\gamma)} & \frac{1}{\beta-1} \end{array} \right), \ R_{1} = \left(\begin{array}{cc} 1 & 1 \\ \frac{1}{\alpha} & \frac{\alpha+1-\gamma}{(\beta-1)(\beta-\gamma)} \end{array} \right), \ R_{\infty} = \left( \begin{array}{cc} 1 & 0 \\ 0 & \frac{(\beta-\alpha)(\alpha+1-\beta)}{\alpha(\beta-1)(\beta-\gamma)} \end{array} \right), \nonumber \end{align}
\begin{eqnarray} \Theta_{0} = \left(\begin{array}{cc} 1-\gamma & 0 \\ 0 & 0 \end{array} \right) , \ \Theta_{1} = \left(\begin{array}{cc} \gamma-\alpha-\beta & 0 \\ 0 & 0 \end{array} \right) , \ \Theta_{\infty} = \left(\begin{array}{cc} \alpha & 0 \\ 0 & \beta-1 \end{array} \right), \nonumber \end{eqnarray}
which satisfy $R_{k}^{-1}A_{k}R_{k} = \Theta_{k}$, and $G_{k}(x)$ are the following series: \\

$G_{0}(x) = \left( \begin{matrix*}[l] \ _{2}F_{1} \left( \begin{array}{c} \alpha+1-\gamma,\ \beta-\gamma \\ 1-\gamma \end{array} ;x\right) \text{\LARGE ,} \\ \frac{x(\alpha+1-\gamma)(1-\beta)}{(1-\gamma)(2-\gamma)} \ _{2}F_{1} \left(\begin{array}{c} \alpha+2-\gamma, \  \beta+1-\gamma \\ 3-\gamma \end{array} ; x \right) \text{\LARGE ,} \end{matrix*} \right.$ \\
\begin{flushright}$\left. \begin{matrix*}[r] \frac{x \alpha (\gamma-\beta)}{\gamma(\gamma-1)} \ \ _{2}F_{1} \left(\begin{array}{c} \alpha+1, \ \beta \\ \gamma+1 \end{array} ; x\right) \\ \ _{2}F_{1} \left(\begin{array}{c}\alpha , \ \beta-1 \\ \gamma - 1 \end{array} ; x \right) \end{matrix*} \right),$ \end{flushright}

$G_{1}(x) = \left( \begin{matrix*}[l] \ _{2}F_{1} \left(\begin{array}{c}\gamma-\alpha-1 , \ \gamma-\beta \\ \gamma-\alpha-\beta \end{array} ; 1-x \right) \text{\LARGE ,} \\ \frac{(1-x)(\beta-1)(\beta-\gamma)}{(\alpha+\beta-\gamma-1)(\alpha+\beta-\gamma)} \ _{2}F_{1}\left(\begin{array}{c} \gamma-\alpha , \ \gamma+1-\beta \\ \gamma+2-\alpha-\beta \end{array} ; 1-x \right) \text{\LARGE ,} \end{matrix*} \right.$ \\
\begin{flushright} $\left. \begin{matrix*}[r] \frac{(1-x)\alpha(\alpha+1-\gamma)}{(\alpha+\beta-\gamma)(\alpha+\beta+1-\gamma)} \ _{2}F_{1} \left( \begin{array}{c} \alpha+1 , \ \beta \\ \alpha+\beta+2-\gamma \end{array} ; 1-x \right) \\ \ _{2}F_{1} \left( \begin{array}{c} \alpha , \ \beta-1 \\ \alpha+\beta-\gamma \end{array} ; 1-x \right) \end{matrix*} \right),$ \end{flushright}

$G_{\infty}(x) = \left( \begin{matrix*}[l] \ _{2}F_{1} \left(\begin{array}{c} \alpha , \ \alpha+1-\gamma \\ \alpha+1-\beta \end{array} ; x^{-1} \right) \text{\LARGE ,} \\ \frac{\alpha(\beta-1)(\beta-\gamma)(\gamma-\alpha-1)}{(\alpha-\beta)(\alpha+1-\beta)^{2}(\alpha+2-\beta)}\frac{1}{x} \ _{2}F_{1} \left( \begin{array}{c} \alpha+1, \ \alpha+2-\gamma \\ \alpha+3-\beta \end{array} ; x^{-1} \right) \text{\LARGE ,} \end{matrix*} \right.$ \\
\begin{flushright} $\left. \begin{matrix*}[r] -\frac{1}{x} \ _{2}F_{1} \left( \begin{array}{c} \beta , \ \beta+1-\gamma \\ \beta+1-\alpha \end{array} ; x^{-1} \right) \\ \ _{2}F_{1} \left(\begin{array}{c} \beta-1 , \ \beta-\gamma \\ \beta-\alpha-1 \end{array} ; x^{-1} \right) \end{matrix*} \right).$ \end{flushright} \end{lemma}

\begin{proof} This result can be proved in two ways: either by reducing equation (\ref{eq:hg1}) to Birkhoff normal form near each singularity and computing the corresponding gauge transformations $R_0G_0(x)$,  $R_1G_1(x)$ and $G_\infty(x)$ recursively or  by direct substitution of the local solutions (\ref{eq:y0})-(\ref{eq:yinf}) into expression (\ref{eq:Y}) and using Gauss contiguous relations. \end{proof}

\begin{remark} \label{remark:lm} The matrices $R_{k}$, $k=0,1$ and $\infty$, in the above solutions (\ref{eq:ak}), (\ref{eq:ak2}) and (\ref{eq:infinity}) have been chosen to satisfy $R_{k}^{-1}A_{k}R_{k} = \Theta_{k}$, where $A_{\infty} := -A_{0}-A_{1}$. The matrices $G_0,G_1,G_\infty$ have leading term given by the identity.\end{remark}

We now define the monodromy data of Gauss hypergeometric equation (\ref{eq:gauss}) and recall how to express them in explicit form \cite{bateman, ww}. In appendix A we 
derive these classical formulae by following the approach of representing solutions using Mellin-Barnes integrals.

When defining local solutions, we have been specific about identifying which sheet of the Riemann surface of the logarithm we are restricting our local solutions to at each singular point. We may extend the definitions of our local fundamental solutions $Y^{(k)}(x)$ to other sheets $e^{2 m \pi i} \Omega_{k}$, $k=0,1,\infty$, by analytically continuing along a closed loop encircling the singularity $x=0,1,\infty$. This action simply means that our solution becomes multiplied by the corresponding exponent $e^{2 m \pi i \Theta_{k}}$, for $k=0,1$ and $\infty$, $m \in \mathbb{Z}$. Note that, for $k=0$ and $1$, the analytic continuation of $Y^{(k)}(x)$ around its singularity in the positive direction means $m>0$ in the previous sentence; while, for $k=\infty$, it means $m < 0$. The diagonal matrices $e^{2 \pi i \Theta_{k}}$ are called the local monodromy exponents of the singularities. \\

We proceed with the global analysis of solutions. Let $Y^{(0)}(x)$, $Y^{(1)}(x)$ and $Y^{(\infty)}(x)$ be the fundamental solutions of the hypergeometric equation as defined in the previous section. Denote by $\gamma_{j,k}\left[Y^{(j)}\right](x)$ the analytic continuation of $Y^{(j)}(x)$ along an orientable curve $\gamma_{j,k} : [0,1] \rightarrow \mathbb{C}$ with $\gamma_{j,k}(0) \in \Omega_{j}$ and $\gamma_{j ,k}(1) \in \Omega_{k}$, for $j,k=0,1,\infty$. We have the following connection formulae (see Appendix A for the detailed derivation of these): 
\begin{align} \gamma_{j,k}\left[Y^{(j)}\right](x) = Y^{(k)}(x) C^{kj}, \label{eq:cont} \end{align}
where: 
\begin{align} C^{0 \infty} &= \left(\begin{array}{cc} e^{i \pi (\gamma-1)} \frac{\Gamma(\alpha+1-\beta)\Gamma(\gamma-1)}{\Gamma(\alpha)\Gamma(\gamma-\beta)} & e^{i \pi(\gamma-1)} \frac{\Gamma(\beta+1-\alpha)\Gamma(\gamma-1)}{\Gamma(\beta)\Gamma(\gamma-\alpha)} \\ \frac{\Gamma(\alpha+1-\beta)\Gamma(1-\gamma)}{\Gamma(1-\beta)\Gamma(\alpha+1-\gamma)} & \frac{\Gamma(\beta+1-\alpha)\Gamma(1-\gamma)}{\Gamma(1-\alpha)\Gamma(\beta+1-\gamma)} \end{array} \right), \label{eq:ci0} \\
C^{1 \infty} &= \left(\begin{array}{cc} e^{i \pi (\gamma-\beta)} \frac{\Gamma(\alpha+1-\beta)\Gamma(\alpha+\beta-\gamma)}{\Gamma(\alpha)\Gamma(\alpha+1-\gamma)} & e^{i \pi (\gamma-\alpha)} \frac{\Gamma(\beta+1-\alpha) \Gamma(\alpha+\beta-\gamma)}{\Gamma(\beta)\Gamma(\beta+1-\gamma)} \\ e^{i \pi \alpha} \frac{\Gamma(\alpha+1-\beta)\Gamma(\gamma-\alpha-\beta)}{\Gamma(1-\beta)\Gamma(\gamma-\beta)} & e^{i \pi \beta} \frac{\Gamma(\beta+1-\alpha)\Gamma(\gamma-\alpha-\beta)}{\Gamma(1-\alpha)\Gamma(\gamma-\alpha)} \end{array} \right), \label{eq:ci1} \\
C^{0 1} &= \left(\begin{array}{cc} \frac{\Gamma(\gamma+1-\alpha-\beta)\Gamma(\gamma-1)}{\Gamma(\gamma-\alpha)\Gamma(\gamma-\beta)} & \frac{\Gamma(\alpha+\beta+1-\gamma)\Gamma(\gamma-1)}{\Gamma(\alpha)\Gamma(\beta)} \\ \frac{\Gamma(\gamma+1-\alpha-\beta)\Gamma(1-\gamma)}{\Gamma(1-\alpha)\Gamma(1-\beta)} & \frac{\Gamma(\alpha+\beta+1-\gamma)\Gamma(1-\gamma)}{\Gamma(\alpha+1-\gamma)\Gamma(\beta+1-\gamma)} \end{array} \right). \label{eq:c10} \end{align}

We choose to normalise the monodromy data of Gauss hypergeometric equation with the fundamental solution $Y^{(\infty)}(x)$. Denote by $\gamma_{k}\left[Y^{(\infty)}\right](x)$ the analytic continuation of $Y^{(\infty)}(x)$ along an orientable, closed curve $\gamma_{k} : [0,1] \rightarrow \mathbb{C}$ with $\gamma_{k}(0) = \gamma_{k}(1) \in \Omega_{\infty}$, $k=0,1$, which encircles the singularity $x=0,1$ respectively in the positive (anti-clockwise) direction. The curves $\gamma_{0}$ and $\gamma_{1}$ are illustrated in Figure \ref{fig:hgloops} below, note that $\gamma_{\infty} := \gamma_{1}^{-1}\gamma_{0}^{-1}$. We have: 
\begin{align} \gamma_{k}\left[Y^{(\infty)}\right](x) = Y^{(k)}(x) M_{k}, \quad \quad k=0,1,\infty, \nonumber \end{align}
where,
\begin{align} M_{0} = \left(C^{0 \infty}\right)^{-1} e^{2 \pi i \Theta_{0}} C^{0 \infty}, \quad M_{1} = \left(C^{1 \infty}\right)^{-1} e^{2 \pi i \Theta_{1}} C^{1 \infty}, \quad M_{\infty} = e^{2 \pi i \Theta_{\infty}}. \label{eq:hgm} \end{align}
These matrices satisfy the cyclic relation, 
\begin{align} M_{\infty}M_{1}M_{0} = I. \label{eq:cyc1} \end{align}

\begin{figure}[H] \begin{center}
\includegraphics[scale=0.6]{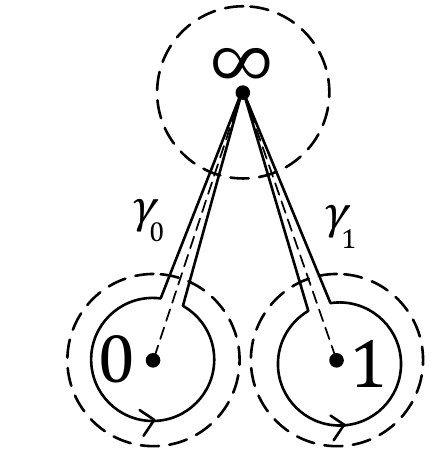}
\caption{\label{fig:hgloops}Curves defining the monodromy matrices $M_{k}$ of Gauss hypergeometric differential equation.}
\end{center} \end{figure} 

\begin{definition}We define the monodromy data of Gauss hypergeometric equation (\ref{eq:gauss}) as the set, 
\begin{align} \mathcal{M} := \left\{ \left(M_{0},M_{1},M_{\infty}\right) \in \left(\text{GL}_{2}(\mathbb{C})\right)^{3} \left| \begin{array}{c} M_{\infty}M_{1}M_{0} = I, \ M_{\infty} = e^{2 \pi i \Theta_{\infty}} \\ \text{eigenv}(M_{k}) = e^{2 \pi i \Theta_{k}}, \text{ k=0,1} \end{array} \right. \right\} _{\text{\Huge/\normalsize GL$_{2}(\mathbb{C})$}} \label{eq:hgmono} \end{align}
where eigenv$(M_{k}) = e^{2 \pi i \Theta_{k}}$ means that the eigenvalues of $M_{k}$ are given as the elements of the diagonal matrix $e^{2 \pi i \Theta_{k}}$ and the quotient is by global conjugation by a diagonal matrix. \end{definition}

\section{Kummer confluent hypergeometric equation} \label{sec:kummer}
We use $z$ as the variable of Kummer confluent hypergeometric equation, we also write tilde above some of the functions and parameters to distinguish from the Gauss hypergeometric equation. We recall the following,
\begin{lemma} \label{lemma:l2} Under the assumption $(\beta-1)(\beta-\gamma) \neq 0$, the matrix
\begin{eqnarray} \widetilde{Y}(z) = \left( \begin{array}{cc} \tilde{y}_{1}(z) & \tilde{y}_{2}(z) \\ \widetilde{\Psi}\left(\tilde{y}_{1},\tilde{y}_{1}';z\right) & \widetilde{\Psi}\left(\tilde{y}_{2},\tilde{y}_{2}';z\right) \end{array} \right), \label{eq:y2} \end{eqnarray}
where,
\[\widetilde{\Psi}\left(\tilde{y}_{k} , \tilde{y}_{k}' ; z \right) = \frac{\left(z + \beta - \gamma \right)\tilde{y}_{k}(z) - z \tilde{y}_{k}'(z)}{(\beta-1)(\beta-\gamma)}, \]
is a fundamental solution of the equation
\begin{eqnarray}
 \frac{\partial \widetilde{Y}}{\partial z} = \left( \left(\begin{array}{cc} 1 & 0 \\ 0 & 0 \end{array} \right) + \frac{\widetilde{A}_{0}}{z}\right) \widetilde{Y}, \text{ where } \widetilde{A}_{0} = \left( \begin{array}{cc} \beta-\gamma & (1-\beta)(\beta-\gamma) \\ 1 & 1-\beta \end{array} \right), \label{eq:chg1} 
\end{eqnarray}
if and only if $\tilde{y}_{1}(z)$ and $\tilde{y}_{2}(z)$ are linearly independent solutions of Kummer confluent hypergeometric equation (\ref{eq:kummer}), 
\begin{align} z \ \tilde{y}'' + (\gamma-z) \ \tilde{y}' - \beta \ \tilde{y}=0. \nonumber \end{align}\end{lemma}
Kummer confluent hypergeometric equation (\ref{eq:kummer}) has one Fuchsian singularity at $z=0$, since $\frac{\gamma-z}{z}$ and $\frac{-\beta}{z}$ have simple poles at $z=0$, and an irregular singularity at $z=\infty$ of Poincar\'{e} rank one. The exponents of the singularity $z=0$ are $1-\gamma$ and $0$ and at $z=\infty$ are $\gamma-\beta$ and $\beta-1$. We make the non-resonance assumption $\gamma \notin \mathbb{Z}$. 

\subsubsection{Local behaviour of the solutions} \label{sec:kumsol}
Kummer confluent hypergeometric equation has an irregular singularity at $z=\infty$ of Poincar\'{e} rank one and, as such, solutions around this point exhibit Stokes phenomenon. In this sub-section, we will state some definitions and theorems which precisely describe fundamental solutions of Kummer equation at the irregular point and the monodromy data, including Stokes matrices.

We first fix the pair of linearly independent local solutions of (\ref{eq:kummer}) as follows:
\begin{align}
 &\begin{matrix*}[l] \tilde{y}_{1}^{(0)}(z) = z^{1-\gamma} \ _{1}F_{1} \left( \begin{array}{c} \beta+1-\gamma \\ 2-\gamma \end{array} ; z \right), \\ \tilde{y}_{2}^{(0)}(z) = \ _{1}F_{1} \left( \begin{array}{c} \beta \\ \gamma \end{array} ; z \right), \end{matrix*} &&z \in \widetilde{\Omega}_{0}. \label{eq:yt0}
 \end{align} 
 where 
 \[\widetilde{\Omega}_{0} := \left\{z: - \frac{3}{2}\pi \leq \text{arg}(z) <  \frac{\pi}{2}\right\},\]
 is a punctured disk around $0$ with branch cut along the positive imaginary axis.
 
In terms of the linear system \eqref{eq:y2}, these solutions correspond to the following local fundamental solution of the matrix hypergeometric equation (\ref{eq:chg1}):
\begin{align} 
\widetilde{Y}^{(0)}(z) = \widetilde{R}_{0} H_{0}(z) z^{\widetilde{\Theta}_{0}}, &&z \in \widetilde{\Omega}_{0}, \label{eq:aroundzero} 
\end{align}
where $\widetilde{R}_{0}$ and $\widetilde{\Theta}_{0}$ are the following matrices:
\begin{eqnarray} 
\widetilde{R}_{0} = \left( \begin{array}{cc} 1 & 1 \\ \frac{1}{\beta-\gamma} & \frac{1}{\beta-1} \end{array} \right) \quad \text{and} \quad \widetilde{\Theta}_{0} = \left(\begin{array}{cc} 1-\gamma & 0 \\ 0 & 0 \end{array} \right), \nonumber 
\end{eqnarray}
which satisfy $\widetilde{R}_{0}^{-1} \widetilde{A}_{0} \widetilde{R}_{0} = \widetilde{\Theta}_{0}$, and $H_{0}(z)$ is the following series:
$$
H_{0}(z) = \left(\begin{array}{cc} _{1}F_{1} \left( \begin{array}{c} \beta-\gamma \\ 1-\gamma \end{array} ;z\right) & \frac{z (\gamma-\beta)}{\gamma(\gamma-1)} \ \ _{1}F_{1} \left(\begin{array}{c} \beta \\ \gamma+1 \end{array} ; z\right) \\ & \\ \frac{z(1-\beta)}{(1-\gamma)(2-\gamma)} \ _{1}F_{1} \left(\begin{array}{c} \beta+1-\gamma \\ 3-\gamma \end{array} ; z \right) & _{1}F_{1} \left(\begin{array}{c}\beta-1 \\ \gamma - 1 \end{array} ; z \right) \end{array} \right).
$$

We now turn our attention to the irregular singularity $z=\infty$. 

\begin{definition} 
The rays $\{z : \text{Re}(z) = 0 , \ \text{Im}(z) >0\}$ and $\{z : \text{Re}(z) = 0 , \ \text{Im}(z) <0\}$ are called the Stokes rays of Kummer equation (\ref{eq:kummer}). 
\end{definition}

We note that these rays constitute the borderline where the behaviour of $e^{z}$ changes, as $z \rightarrow \infty$; that is to say, on one side of each of these rays we have $e^{z} \rightarrow 0$, whereas on the other side of each ray we have $e^{z} \rightarrow \infty$. This is a key aspect of Stokes phenomenon and plays a role in understanding the following classical theorem. 
\begin{theorem} \label{theorem:kummerst} 
Let  
\[\widetilde{\Sigma}_{k} = \left\{ z : -\frac{\pi}{2} < \text{arg}(z) - k \pi < \frac{3\pi}{2} \right\}.\]
For all $k \in \mathbb{Z}$, there exists a solution $\widetilde{Y}^{(\infty,k)}(z)$ of equation (\ref{eq:chg1}) analytic in the sector $\widetilde{\Sigma}_{k}$ such that,
\begin{align} \widetilde{Y}^{(\infty,k)}(z) \sim \widetilde{R}_{\infty}\left(\sum_{n=0}^{\infty}h_{n,\infty}z^{-n}\right) \left(\begin{array}{cc} e^{z} z^{\beta-\gamma} & 0 \\ 0 & z^{1-\beta} \end{array} \right), \quad \quad \text{as } z \rightarrow \infty, \ z \in \widetilde{\Sigma}_{k}, \label{eq:chgasy} \end{align}
where $\widetilde{R}_{\infty}$ is the following matrix,
\begin{eqnarray} \widetilde{R}_{\infty} = \left(\begin{array}{cc} 1 & 0 \\ 0 & \frac{-1}{(\beta-1)(\beta-\gamma)} \end{array} \right), \nonumber \end{eqnarray}
and $H_{\infty}(z)$ is the following series
$$
H_{\infty}(z) = \left( \begin{array}{cc} \ _{2}F_{0} \left(1-\beta , \gamma-\beta ; z^{-1}\right) & \frac{-1}{z} \ _{2}F_{0} \left(\beta , \beta+1-\gamma ; -z^{-1} \right) \\ \frac{(1-\beta)(\beta-\gamma)}{z} \ _{2}F_{0} \left( 2-\beta , \gamma+1-\beta ; z^{-1} \right) & \ _{2}F_{0} \left(\beta-1, \beta-\gamma ; -z^{-1} \right) \end{array} \right).
$$
Moreover, each solution $\widetilde{Y}^{(\infty,k)}(z)$ is uniquely specified by the relation (\ref{eq:chgasy}). \end{theorem}

\begin{proof} A proof of the existence of fundamental solutions $\widetilde{Y}^{(\infty,k)}(z)$ which are analytic on sectors $\widetilde{\Sigma}_{k}$ may be found in \cite{BJL}. 
To find the asymptotic behaviour \eqref{eq:chgasy}, we make the following ansatz
\begin{align} 
&\widetilde{Y}^{(\infty,k)}(z) \sim \widetilde{R}_{\infty} H_{\infty}(z) \text{exp}\left( \int_{-\infty}^{z} \left(\Lambda_{0} + \frac{\Lambda_{1}}{z'}\right)\ dz' \right), &&\text{as } z \rightarrow \infty, \ z \in \widetilde{\Sigma}_{k}, \nonumber
 \end{align} 
where,
\[ \widetilde{R}_{\infty} = \left(\begin{array}{cc} 1 & 0 \\ 0 & \frac{-1}{(\beta-1)(\beta-\gamma)} \end{array} \right),\]
$\Lambda_{0}$ and $\Lambda_{1}$ are constant, diagonal matrices to be determined and $H_{\infty}(z)$ is a formal series
\[H_{\infty}(z) = \sum_{n=0}^{\infty} h_{n,\infty}z^{-n}.\]
where the coefficients $h_{n,\infty}$ are to be determined. 

By substitution in the equation (\ref{eq:chg1}), we obtain\begin{align} -\sum_{n=1}^{\infty} n h_{n,\infty} z^{-n-1} &+ \left(\sum_{n=0}^{\infty} h_{n,\infty} z^{-n}\right)\left( \Lambda_{0} + \frac{\Lambda_{1}}{z} \right) \nonumber \\
&= \widetilde{R}_{\infty}^{-1} \left( \left( \begin{array}{cc} 1 & 0 \\ 0 & 0 \end{array} \right) + \frac{\widetilde{A}_{0}}{z} \right)\widetilde{R}_{\infty}  \left( \sum_{n=0}^{\infty} h_{n,\infty} z^{-n} \right). \nonumber \end{align}
By setting $h_{0,\infty} = I$ and equating powers of $z^{-n}$ in this equation, for $n=0$ and $1$, we find:
\[\Lambda_{0} = \left(\begin{array}{cc} 1 & 0 \\ 0 & 0 \end{array} \right) \text{ and } \Lambda_{1} = \left(\begin{array}{cc} \beta-\gamma & 0 \\ 0 & 1-\beta \end{array}\right),\]
and, for $n \geq 1$, we find the recursion equation, 
\[\left[ h_{n,\infty} , \left(\begin{array}{cc} 1 & 0 \\ 0 & 0 \end{array} \right) \right] = (n-1) h_{n-1,\infty} + h_{n-1,\infty} \left(\begin{array}{cc} \gamma-\beta & 0 \\ 0 & \beta-1 \end{array} \right) + \widetilde{R}_{\infty}^{-1}\widetilde{A}_{0}\widetilde{R}_{\infty}h_{n-1,\infty}.\]
It can be verified that the general solution of this equation is, 
\begin{align}h_{n,\infty} &= \left(\begin{array}{cc} \frac{(1-\beta)_{n}(\gamma-\beta)_{n}}{n!} & \frac{(\beta)_{n-1}(\beta+1-\gamma)_{n-1}}{(-1)^{n}(n-1)!} \\ \frac{(1-\beta)(\beta-\gamma)(2-\beta)_{n-1}(\gamma+1-\beta)_{n-1}}{(n-1)!} & \frac{(\beta-1)_{n}(\beta-\gamma)_{n}}{(-1)^{n}n!} \end{array}\right), \label{eq:hninf} \end{align}
which are indeed the coeficients in the asymptotic series given.

To prove uniqueness of solutions, let $\widehat{Y}^{(\infty,k)}(z)$ denote another fundamental solution of equation (\ref{eq:chg1}) which is analytic on the sector $\widetilde{\Sigma}_{k}$ and has the correct asymptotic behavior, namely, 
\begin{align} \widehat{Y}^{(\infty,k)}(z) \sim \widetilde{R}_{\infty}\left(\sum_{n=0}^{\infty}h_{n,\infty}z^{-n}\right) \left(\begin{array}{cc} e^{z} z^{\beta-\gamma} & 0 \\ 0 & z^{1-\beta} \end{array} \right), \quad \quad \text{as } z \rightarrow \infty, \ z \in \widetilde{\Sigma}_{k}. \label{eq:chgasy2} \end{align}
Since $\widetilde{Y}^{(\infty,k)}(z)$ and $\widehat{Y}^{(\infty,k)}(z)$ are fundamental solutions defined on the same sector, there exists a constant matrix $C \in \text{GL}_{2}(\mathbb{C})$ such that,
\[\widetilde{Y}^{(\infty,k)}(z) = \widehat{Y}^{(\infty,k)}(z) C, \quad \quad z \in \widetilde{\Sigma}_{k}.\]
Using the asymptotic relations (\ref{eq:chgasy}) and (\ref{eq:chgasy2}), we deduce the following, 
\[\left(\begin{array}{cc} e^{z} z^{\beta-\gamma} & 0 \\ 0 & z^{1-\beta} \end{array} \right) C \left(\begin{array}{cc} e^{-z} z^{\gamma-\beta} & 0 \\ 0 & z^{\beta-1} \end{array} \right) \sim I, \quad \quad \text{as } z \rightarrow \infty, \ z \in \widetilde{\Sigma}_{k}.\]
From this relation, we immediately see that $(C)_{1,1} = (C)_{2,2}=1$. Moreover, since there exists rays belonging to $\widetilde{\Sigma}_{k}$ along which each exponential, $e^{z}$ and $e^{-z}$, explodes as $z \rightarrow \infty$, we conclude that $(C)_{1,2} = (C)_{2,1} = 0$. \end{proof}

\begin{remark} The matrices $\widetilde{R}_{0}$ and $\widetilde{R}_{\infty}$ in the above solutions (\ref{eq:aroundzero}) and (\ref{eq:chgasy}) have been chosen to satisfy $\widetilde{R}_{0}^{-1}\widetilde{A}_{0}\widetilde{R}_{0} = \widetilde{\Theta}_{0}$ and, 
\[ \left[\widetilde{R}_{\infty} , \left(\begin{array}{cc} 1 & 0 \\ 0 & 0 \end{array} \right)\right] = 0.\]
\end{remark}

We denote the asymptotic behaviour of true solutions at infinity as in (\ref{eq:chgasy}) by,
\begin{align} 
\widetilde{Y}_{f}^{(\infty)}(z) = \left( \sum_{n=0}^{\infty}h_{n,\infty}z^{-n}\right) \left(\begin{array}{cc} e^{z} z^{\beta-\gamma} & 0 \\ 0 & z^{1-\beta} \end{array} \right), \quad \quad z \in \widetilde{\Sigma}_{k}. \nonumber \end{align}
The series $H_{\infty}(z) = \sum_{n=0}^{\infty}h_{n,\infty}z^{-n}$ defines a formal gauge transformation which maps equation (\ref{eq:chg1}) to, 
\begin{align} \frac{\partial}{\partial z}\widehat{Y}(z) = \left( \left(\begin{array}{cc} 1 & 0 \\ 0 & 0 \end{array} \right) + \frac{1}{z} \left( \begin{array}{cc}\beta-\gamma & 0 \\ 0 & 1-\beta\end{array} \right) \right) \widehat{Y}, \label{eq:neweqn} \end{align}
via the transformation $\widetilde{Y}(z) = \widetilde{R}_{\infty} H_{\infty}(z) \widehat{Y}(z)$. We define the coefficient of $\frac{1}{z}$ in the new equation to be $-\widetilde{\Theta}_{\infty}$, namely,
\[\widetilde{\Theta}_{\infty} := \left(\begin{array}{cc} \gamma-\beta & 0 \\ 0 & \beta-1 \end{array} \right) \equiv - \text{diag}\left(\widetilde{A}_{0}\right).\]
In the generic case $a,b \notin \mathbb{Z}^{\leq 0}$, d'Alembert's ratio test shows that the series $\ _{2}F_{0}( a , b ; z^{-1} )$ diverges for all $z \in \mathbb{C}$. In this sense, the asymptotic behaviour $\widetilde{Y}_{f}^{(\infty)}(z)$ is a formal fundamental solution. \\

\begin{remark} \label{remark:formal} Using expression (\ref{eq:y2}) in Lemma \ref{lemma:l2}, the formal fundamental solution $\widetilde{Y}_{f}^{(\infty)}$ of (\ref{eq:chg1}) corresponds to the following standard formal basis of solutions of (\ref{eq:kummer}),
\begin{align} &\begin{matrix*}[l] \tilde{y}_{1,f}^{(\infty)}(z) = e^{z} z^{\beta-\gamma} \ _{2}F_{0} \left(\gamma-\beta, \ 1-\beta ; z^{-1} \right), \\ \tilde{y}_{2,f}^{(\infty)}(z) = -z^{-\beta}\ _{2}F_{0} \left( \beta , \ \beta+1-\gamma ; -z^{-1} \right). \end{matrix*} \label{eq:yformal} \end{align} \end{remark}

\subsubsection{Monodromy data} \label{sec:kummon}
We now define the monodromy data, including Stokes data, of Kummer equation (\ref{eq:kummer}) and recall how to express them in explicit form \cite{bateman,ww}. In Appendix B, we derive these classical formulae by representing solutions using Mellin-Barnes integrals. 

\begin{definition} \label{definition:kumst} Let $\widetilde{Y}^{(\infty,k)}(z)$ be the fundamental solutions given in Theorem \ref{theorem:kummerst} and define sectors,
\[\widetilde{\Pi}_{k} := \widetilde{\Sigma}_{k} \cap \widetilde{\Sigma}_{k+1} \equiv \left\{ z : |z|>0 , \ \frac{\pi}{2} < \text{arg}(z) - k \pi < \frac{3\pi}{2}\right\},\]
as illustrated in Figure \ref{fig:psec} below. We define Stokes matrices $\widetilde{S}_{k} \in \text{SL}_{2}(\mathbb{C})$ as follows,
\begin{align} \widetilde{Y}^{(\infty,k+1)}(z) = \widetilde{Y}^{(\infty,k)}(z) \widetilde{S}_{k} , \quad z \in \widetilde{\Pi}_{k}. \label{eq:kummerstokes} \end{align} \end{definition}

\begin{figure}[H] \begin{center}
\includegraphics[scale=.6]{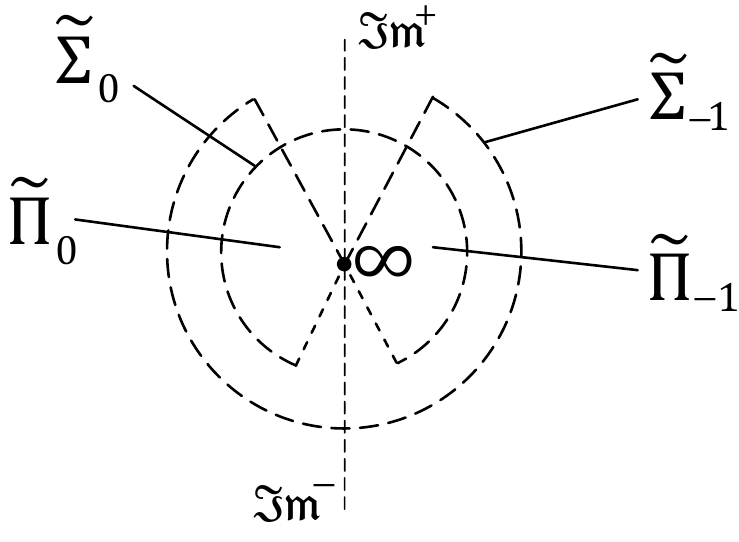}
\caption{\label{fig:psec}Sectors $\widetilde{\Pi}_{0}$, $\widetilde{\Pi}_{-1}$, $\widetilde{\Sigma}_{0}$ and $\widetilde{\Sigma}_{-1}$ projected onto the plane $\overline{\mathbb{C}}\backslash\{0\}$. The positive and negative imaginary axes are Stokes rays.}
\end{center} \end{figure}

From the asymptotic relation (\ref{eq:chgasy}), it is clear that 
\begin{align} 
\widetilde{Y}^{(\infty,k+2)}(z) = \widetilde{Y}^{(\infty,k)}\left(ze^{-2 \pi i}\right) e^{-2 \pi i \widetilde{\Theta}_{\infty}}, \quad \quad z \in \widetilde{\Sigma}_{k+2}. \label{eq:sameasy} 
\end{align}
due to the fact that these two solutions
have the same asymptotic behaviour as $z \rightarrow \infty$ in the sector $z \in \widetilde{\Sigma}_{k+2}$. Therefore all solutions $\widetilde{Y}^{(\infty,k)}(z)$ are categorised into two fundamentally distinct cases, namely, when $k$ is even and when $k$ is odd. 
Combining Definition \ref{definition:kumst} with the relation (\ref{eq:sameasy}), one can show that
\begin{align} 
e^{-2 \pi i \widetilde{\Theta}_{\infty}} \widetilde{S}_{k+1} = \widetilde{S}_{k-1} e^{-2\pi i \widetilde{\Theta}_{\infty}}, \nonumber \end{align}
which shows that Kummer equation has only two types of Stokes matrices $\widetilde{S}_{k}$ which are fundamentally different: one with $k$ odd and the other with $k$ even. 

Here we select to work with the fundamental solutions $\widetilde{Y}^{(\infty,-1)}(z)$ in the sector $\widetilde\Sigma_{-1}$ and $\widetilde{Y}^{(\infty,0)}(z)$ in the sector $\widetilde\Sigma_{0}$ and with the Stokes matrices $\widetilde S_{0}$ and $\widetilde S_{-1}$. The explicit form of  the Stokes matrices are derived in the Appendix B where the following Lemma is proved:

\begin{lemma} \label{lemma:stokes} We have the following classical formulae:
\begin{align}\widetilde{S}_{0} &= \left(\begin{array}{cc} 1 & \frac{2 \pi i }{\Gamma(\beta)\Gamma(\beta+1-\gamma)} e^{i \pi (\gamma-2\beta)} \\ 0 & 1 \end{array} \right) \quad \text{and} \quad \widetilde{S}_{-1} &= \left(\begin{array}{cc} 1 & 0 \\ \frac{2 \pi i}{\Gamma(1-\beta)\Gamma(\gamma-\beta)} & 1 \end{array} \right). \label{eq:kummers-1} \end{align} \end{lemma}

We choose to normalise our monodromy data with respect to the fundamental solution $\widetilde{Y}^{(\infty,0)}(z)$. Denote by $\gamma_{\infty,0}\left[\widetilde{Y}^{(\infty,0)}\right](z)$ the analytic continuation of $\widetilde{Y}^{(\infty,0)}(z)$ along an orientable curve $\gamma_{\infty,0} : [0,1] \rightarrow \mathbb{C}$ with $\gamma_{\infty,0}(0) \in \widetilde{\Sigma}_{0}$ and $\gamma_{\infty,0}(1) \in \widetilde{\Omega}_{0}$. We have,
\[\gamma_{\infty 0} \left[\widetilde{Y}^{(\infty,0)}\right](z) = \widetilde{Y}^{(0)}(z) \widetilde{C}^{0 \infty},\]
where,
\begin{align} \widetilde{C}^{0 \infty} &= \left(\begin{array}{cc} e^{i \pi (\beta-1)} \frac{\Gamma(\gamma-1)}{\Gamma(\gamma-\beta)} & -\frac{\Gamma(\gamma-1)}{\Gamma(\beta)} \\ e^{i \pi (\beta-\gamma)} \frac{\Gamma(1-\gamma)}{\Gamma(1-\beta)} & -\frac{\Gamma(1-\gamma)}{\Gamma(\beta+1-\gamma)} \end{array} \right). \label{eq:kummerc0} \end{align}

Denote by $\gamma_{0}\left[\widetilde{Y}^{(\infty,0)}\right](z)$ the analytic continuation of $\widetilde{Y}^{(\infty,0)}(z)$ along an orientable, closed curve $\gamma_{0} : [0,1] \rightarrow \mathbb{C}$ with $\gamma_{0}(0) = \gamma_{0}(1) \in \widetilde{\Sigma}_{0}$ which encircles the singularity $z = 0$ in the positive (anti-clockwise) direction. The curve $\gamma_{0}$ is illustrated below, note that $\gamma_{\infty}:=\gamma_{0}^{-1}$.
\begin{figure}[H] \begin{center}
\includegraphics[scale=.8]{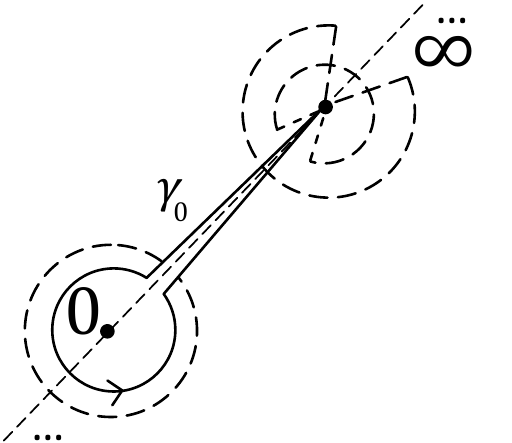}
\caption{\label{fig:chgloops}Curves defining the monodromy matrices $\widetilde{M}_{k}$ of Kummer hypergeometric differential equation.}
\end{center} \end{figure} 

We have,
\begin{align} \gamma_{k}\left[ \widetilde{Y}^{(\infty,0)}\right](z) = Y^{(\infty,k)}(z)\widetilde{M}_{k}, \quad \quad k=0,\infty, \nonumber \end{align}
where,
\begin{align} \widetilde{M}_{0} = \left(\widetilde{C}^{0\infty}\right)^{-1} e^{2 \pi i \widetilde{\Theta}_{0}} \widetilde{C}^{0 \infty} \quad \text{and} \quad \widetilde{M}_{\infty} = \widetilde{S}_{0}e^{2 \pi i \widetilde{\Theta}_{\infty}}\widetilde{S}_{-1}. \label{eq:chgm} \end{align}
These matrices satisfy the cyclic relation,
\begin{align} \widetilde{M}_{\infty}\widetilde{M}_{0} = I. \label{eq:cycrel} \end{align}

\begin{definition} We define the monodromy data of Kummer hypergeometric differential equation (\ref{eq:kummer}) as the set, 
\begin{align} 
\widetilde{\mathcal{M}} := \left\{ \begin{matrix*}[r] \left(\widetilde{M}_{0} , \widetilde{S}_{0} , \widetilde{S}_{-1}\right) \\ \in \left(\text{GL}_{2}(\mathbb{C})\right)^{3} \end{matrix*} \left| \begin{array}{c} \widetilde{S}_{0} \text{ unipotent, upper triangular,} \\ \widetilde{S}_{-1} \text{unipotent, lower triangular,} \\ \widetilde{S}_{0}e^{2 \pi i \widetilde{\Theta}_{\infty}} \widetilde{S}_{-1}\widetilde{M}_{0} = I, \\ \text{eigenv}\left(\widetilde{M}_{0}\right) = e^{2 \pi i \widetilde{\Theta}_{0}} \end{array} \right. \right\}_{\text{\Huge/\normalsize GL$_{2}(\mathbb{C})$}} \label{eq:chgmono} \end{align}
where eigenv$(\widetilde{M}_{0}) = e^{2\pi i \widetilde{\Theta}_{0}}$ means that the eigenvalues of $\widetilde{M}_{0}$ are given as the elements of the diagonal matrix $e^{2 \pi i \widetilde{\Theta}_{0}}$ and the quotient is by global conjugation by a diagonal matrix. \end{definition}

\section{Confluence from Gauss to Kummer equation} \label{sec:hgconf}

In this Section we analyse the confluence procedure from Gauss equation (\ref{eq:gauss}) to Kummer equation (\ref{eq:kummer}). We are primarily concerned with understanding how to produce the monodromy data of the Kummer equation, as defined in Section \ref{sec:kummon}, from the connection matrices of the Gauss equation (see Section \ref{sec:gg}), under the confluence procedure. 

We first explain how the confluence procedure works intuitively. By the substitution  $x=\frac{z}{\alpha}$, 
on the Gauss equation (\ref{eq:gauss})
\begin{align} x(1-x) \ y''(x) + (\gamma - (\alpha+\beta+1)x) \ y'(x) - \alpha \beta \ y(x) &= 0, \nonumber \\
\Leftrightarrow \frac{z}{\alpha} \left( \frac{\alpha-z}{\alpha} \right) \alpha^{2} \ y_{zz} + \left(\gamma - (\alpha+\beta+1)\frac{z}{\alpha}\right) \alpha \ y_{z} - \alpha \beta \ y &= 0, \nonumber \\
\Leftrightarrow z \ y_{zz} + (\gamma-z) \ y_{z} - \beta \ y - \frac{1}{\alpha} \left(z^{2} y_{zz} + (\beta+1) y_{z}\right) &=0. \nonumber \end{align}
we produce an differential equation with three Fuchsian singularitites at $z=0, \alpha$ and $\infty$ respectively. 

As a heuristic argument, one can see that the final equation becomes Kummer equation (\ref{eq:kummer}) as $\alpha \rightarrow \infty$ so that  a double pole is created at $z=\infty$ as the two simple poles $z=\alpha$ and $\infty$ merge. This derivation does not explain how of obtain solutions of the Kummer equation by taking limits as $\alpha \rightarrow \infty$ of certain solutions of Gauss equation under the substitution $x=\frac{z}{\alpha}$. To understand this, we need to use a result by Glutsyuk \cite{glutsyuk}, which deals with limits of solutions at merging simple poles under a generic confluence procedure. This is explained in the next sub-section.

\subsection{A result by Glutsyuk} \label{sec:gl}

Consider the following differential equation, 
\begin{align} \frac{\partial Y}{\partial \lambda} = \frac{A(\lambda, \eps)}{(\lambda-\eps)(\lambda+\eps)} Y, \quad A(\lambda,\eps) \in \text{GL}_{2}(\mathbb{C}), \label{eq:gl} \end{align}
with $A(\lambda,\eps)$ a holomorphic matrix about $\lambda = \pm \eps$ such that $A(\pm \eps,\eps) \neq 0$ for sufficiently small $\eps \geq 0$ satisfying the following limit,
\[\lim_{\eps \rightarrow 0} A(\lambda,\eps) = A(\lambda,0).\]
Hence, the non-perturbed, or \textit{confluent}, equation, 
\begin{align} \frac{\partial Y}{\partial \lambda} = \frac{A(\lambda, 0)}{\lambda^{2}} Y, \label{eq:gl2} \end{align}
has an irregular singularity at $\lambda=0$ of Poincar\'{e} rank one. Moreover, it is assumed that the eigenvalues of the residue matrices  $A(\pm \eps,\eps)$ of at $\lambda = \pm \eps$ are non resonant and thet  the eigenvalues of the leading matrix of $A(\lambda,0)$ at $\lambda = 0$ are distinct.

We first deal with the perturbed equation (\ref{eq:gl}). We define neighbourhoods $\Omega_{\pm \eps}$ of the points $\lambda = \pm \eps$ respectively whose radii are less than $2 |\eps|$ and with branch cuts made along the straight line passing through the points $\lambda = -\eps , 0 , \eps$, as illustrated in Figure \ref{fig:gl3} below. Equation (\ref{eq:gl}) has fundamental solutions $Y^{(\pm \eps)}(\lambda)$ which are analytic in the cut disks $\Omega_{\pm}(\eps)$ of the following form,
\begin{align} Y^{(\pm \eps)}(\lambda) &= \left( \sum_{n=0}^{\infty} G_{n,\pm \eps} (\lambda \mp \eps)^{n} \right) (\lambda \mp \eps)^{\Lambda_{\pm \eps}}, &&\lambda \in \Omega_{\pm \eps}, \nonumber \end{align}
where $G_{0,\pm \eps}$ are fixed matrices which diagonalise the residue matrices $A(\pm \eps,\eps) $ and all other terms of the series are determined by certain recursion formulae.

\begin{figure}[H] \begin{center}
\includegraphics[scale=.5]{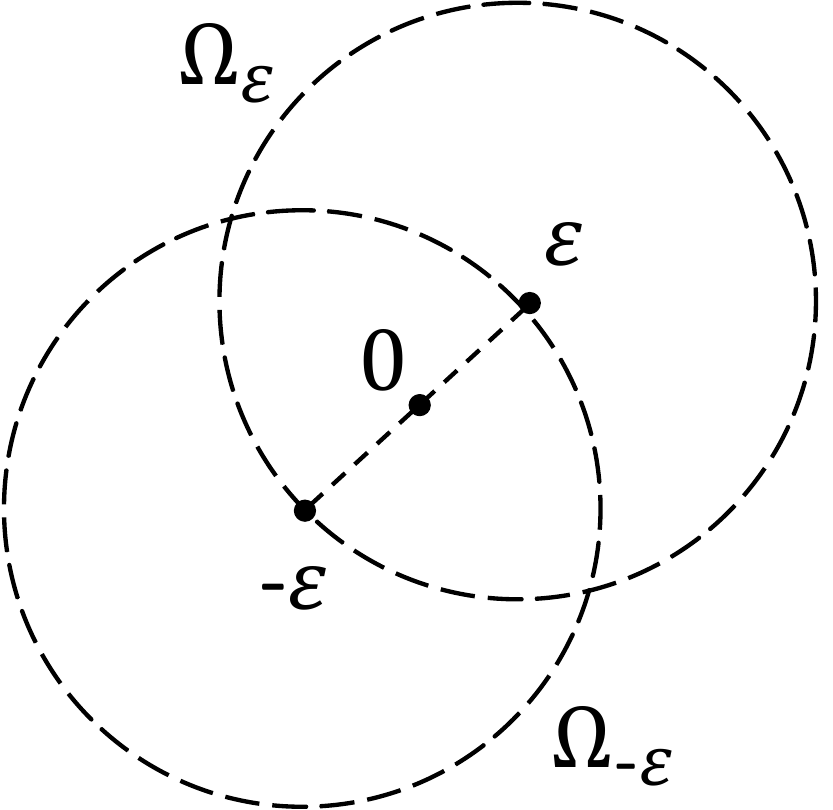}
\caption{\label{fig:gl3}An illustration of the neighbourhoods $\Omega_{\pm \eps}$ with branch cuts in which we define the fundamental solutions $Y^{(\pm \eps)}(\lambda)$.}
\end{center} \end{figure} 

We now turn our attention to the confluent equation (\ref{eq:gl2}). Denote by $\mu_{1}$ and $\mu_{2}$ the eigenvalues of the leading matrix of $A(\lambda,0)$ at $\lambda=0$ (by assumption, $\mu_{1}\neq\mu_{2}$) and let,
\[r_{i,j}=\left\{ \lambda : \ \text{Re}\left( \frac{\mu_{i}-\mu_{j}}{\lambda} \right) = 0, \ \text{Im}\left(\frac{\mu_{i}-\mu_{j}}{\lambda} \right) >0 \right\}, \quad i,j\in \{1,2\},\]
be the Stokes rays. We denote by $\mathscr{S}_{0}$ and $\mathscr{S}_{1}$ open sectors whose union is a punctured neighbourhood of $\lambda=0$, each of which: has an opening greater than $\pi$; contains only one Stokes ray and does not contain the other Stokes ray at its boundary. An illustration of such Stokes rays and sectors is given below. 
\begin{figure}[H] \begin{center}
\includegraphics[scale=.8]{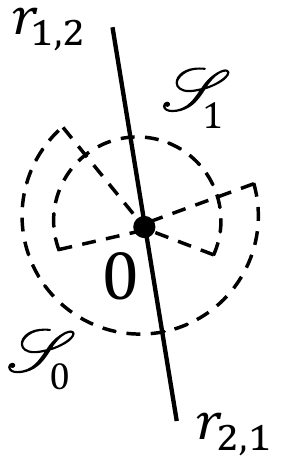}
\caption{\label{fig:gl}An illustration of the Stokes rays $r_{i,j}$ and sectors $\mathscr{S}_{0}$ and $\mathscr{S}_{1}$.}
\end{center} \end{figure} 
We can cover all of the sheets of the Riemann surface of the logarithm at $\lambda = 0$ by extending the notation as follows,
\[\lambda \in \mathscr{S}_{k+2} \Leftrightarrow \lambda e^{-2 \pi i} \in \mathscr{S}_{k}.\]

From  the standard theory of linear systems of ordinary differential equations, there exists a number $R$ sufficiently large such that, for all $k \in \mathbb{Z}$, there exist fundamental solutions $Y^{(0,k)}(\lambda)$ of the non-perturbed equation (\ref{eq:gl2}) analytic in the sectors $\mathscr{S}_{k}$
such that,
\[Y^{(0,k)}(\lambda) \sim \left( \sum_{n=0}^{\infty} H_{n} \lambda^{n} \right) \lambda^{\Theta_{0}} \text{exp}\left(\lambda^{-1} \left(\begin{array}{cc} \mu_{1} & 0 \\ 0 & \mu_{2} \end{array} \right) \right), \text{ as } \lambda \rightarrow 0, \ \lambda \in \mathscr{S}_{k},\]
where $H_{0}$ is a fixed matrix which diagonalises the leading term of $A(\lambda,0)$ at $\lambda=0$, all other terms of the series and the diagonal matrix $\Theta$ are uniquely determined by certain recursion relations. Each solution $Y^{(0,k)}(\lambda)$ is uniquely specified by the above asymptotic relation. 

We define open sectors $\sigma_{\pm \eps}(\eps) \subset \Omega_{\pm}$ with base points at $\lambda = \pm \eps$ respectively whose openings do not contain the branch cut between $-\eps$ and $\eps$ as illustrated in Figure \ref{fig:gl2} below. 
\begin{figure}[H] \begin{center}
\includegraphics[scale=.8]{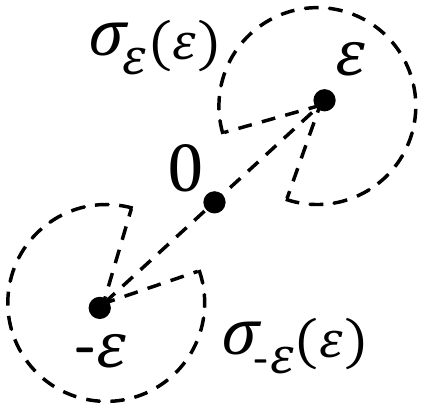}
\caption{\label{fig:gl2}An illustration of the sectors $\sigma_{\pm \eps}(\eps)$.}
\end{center} \end{figure} 
We impose the condition that, as $\eps \rightarrow 0$ along a ray, the sector $\sigma_{\eps}(\eps)$ (resp. $\sigma_{-\eps}(\eps)$) is translated along a ray to zero and becomes in agreement with the sector $\mathscr{S}_{k+1}$ (resp. $\mathscr{S}_{k}$), for some $k \in \mathbb{Z}$. We write this condition as follows,
\begin{align} \lim_{\eps \rightarrow 0} \sigma_{\eps}(\eps) = \mathscr{S}_{k+1} \quad \text{and} \quad \lim_{\eps \rightarrow 0} \sigma_{-\eps}(\eps) = \mathscr{S}_{k}. \label{eq:limits} \end{align}

\begin{theorem} \label{theorem:gl} Let the fundamental solutions $Y^{(\eps)}(\lambda)$, $Y^{(-\eps)}(\lambda)$ and $Y^{(0,k)}(\lambda)$ and the sectors $\sigma_{\eps}(\eps)$, $\sigma_{-\eps}(\eps)$ and $\mathscr{S}_{k}$ be defined as above. There exist diagonal matrices $K_{\eps}$ and $K_{-\eps}$ such that we have the following limits,
\begin{align} &\lim_{\eps \rightarrow 0} \left. Y^{(\eps)}(\lambda)\right|_{\lambda \in \sigma_{\eps}(\eps)} K_{\eps} = Y^{(0,k+1)}(\lambda), \nonumber \\
&\lim_{\eps \rightarrow 0} \left. Y^{(-\eps)}(\lambda)\right|_{\lambda \in \sigma_{-\eps}(\eps)} K_{-\eps} = Y^{(0,k)}(\lambda), \nonumber \end{align}
uniformly for $\lambda \in \mathscr{S}_{k+1}$, $\mathscr{S}_{k}$ respectively, as $\eps$ belongs to a fixed ray. \end{theorem}

\begin{remark} It is well-known that, when solving a linear ordinary differential equation around a Fuchsian singular point, the maximal radius we may take for the neighbourhood on which we can define an analytic solution is the distance to the nearest singularity. For the perturbed equation (\ref{eq:gl}), as $\eps$ becomes arbitrarily small it is clear from the hypotheses on $A(\lambda,\eps)$ that the closest singularity to $\lambda = \pm \eps$ will be $\lambda = \mp \eps$ respectively. We have illustrated the domains $\Omega_{\pm \eps}$ in Figure \ref{fig:gl3} with the maximal radii for which it is possible to define analytic solutions. Observe that the neighbourhoods of analyticity of the fundamental solutions diminish as $\eps \rightarrow 0$. The intelligent part of restricting the fundamental solutions $Y^{(\pm \eps)}(\lambda)$ to the sectors $\sigma_{\pm \eps}(\eps)$ as drawn in Figure \ref{fig:gl2}, rather than the neighbourhoods $\Omega_{\pm \eps}$, is that the radii of these sectors need not be restricted to the distance to the nearest singularity. Indeed, by construction, the singularity $\lambda = \pm \eps$ will not be inside the sector $\sigma_{\mp \eps}(\eps)$ respectively. In particular, this means that the radii of these sectors need not vanish. \end{remark}

By the same reasoning as in the previous remark, it is without loss of generality that we may assume $\sigma_{\eps}(\eps) \cap \sigma_{-\eps}(\eps) \neq \varnothing$ for $\eps$ sufficiently close to zero. Accordingly, since we have two fundamental solutions defined on this intersection, they must be related by multiplication by a constant invertible matrix on the right, namely, 
\begin{align} Y^{(\eps)}(\lambda) = Y^{(-\eps)}(\lambda) C, \quad \quad \lambda \in \sigma_{\eps}(\eps) \cap \sigma_{-\eps}(\eps), \label{eq:connection} \end{align}
for some connection matrix $C \in \text{GL}_{2}(\mathbb{C})$. Similarly, the two fundamental solutions $Y^{(0,0)}(\lambda)$ and $Y^{(0,1)}(\lambda)$ of the confluent equation must be related to each other by multiplication by a constant invertible matrix on the right on the intersection $\mathscr{S}_{0}$ and $\mathscr{S}_{1}$, namely,
\begin{align} Y^{(0,1)}(\lambda) = Y^{(0,0)}(\lambda) S, \quad \quad \lambda \in \mathscr{S}_{0} \cap \mathscr{S}_{1}, \label{eq:stokesmatrix} \end{align}
for some Stokes matrix $S \in \text{GL}_{2}(\mathbb{C})$.

\begin{corollary} \label{corollary:glcorollary} Let the fundamental solutions $Y^{(\eps)}(\lambda)$, $Y^{(-\eps)}(\lambda)$ and $Y^{(0,k)}(\lambda)$ and the sectors $\sigma_{\eps}(\eps)$, $\sigma_{-\eps}(\eps)$ and $\mathscr{S}_{k}$ be defined as above; let $K_{\pm \eps}$ be matrices satisfying Theorem \ref{theorem:gl} and let $C$ and $S$ be the matrices defined by (\ref{eq:connection}) and (\ref{eq:stokesmatrix}) respectively. We have the following limit,
\begin{align} \lim_{\eps \rightarrow 0} K_{-\eps}^{-1} C K_{\eps} = S, \label{eq:sat} \end{align}
as $\eps$ belongs to a fixed ray. \end{corollary}

In (\ref{eq:sat}) it is clear how to obtain one of the Stokes matrices at the point $\lambda=0$ of the confluent equation. In order to obtain the second Stokes matrix we take $\eps \rightarrow 0$ along the opposite ray to the one already considered. Rather than having the limits in (\ref{eq:limits}), we would now have, for example, that $\sigma_{\eps}(\eps)$ tends to $\mathscr{S}_{k}$ and $\sigma_{-\eps}(\eps)$ tends to $\mathscr{S}_{k-1}$. In this way, we use the limit in (\ref{eq:sat}) to produce the other Stokes matrix. We will explain all of these details and calculate everything explicitly for each of the cases we consider.

\subsubsection{Limits of solutions} \label{sec:331}
As outlined above, our confluence procedure is to introduce the new variable $z$ by the substitution $x = \frac{z}{\alpha}$ and take the limit $\alpha \rightarrow \infty$. For the remainder of this chapter we must be careful in which way we are taking $\alpha$ to infinity, for example it would be inconvenient for us if $\alpha$ spiralled towards infinity. We will consider two limits along fixed rays: one with $\arg(\alpha) = \frac{\pi}{2}$ and the other with $\arg(\alpha) = -\frac{\pi}{2}$.

\subsubsection{Obtaining the solutions $\widetilde{Y}^{(\infty,k)}(z)$} 
We now turn our attention to the main problem of how to obtain fundamental solutions at the double pole of the confluent equation from solutions at the merging simple poles of the original equation. We first examine the behaviour of the fundamental solutions at $x=\infty$, as given in (\ref{eq:yinf}). Observe that these solutions are expressed using the Gauss $\ _{2}F_{1}$ series in the variable $x^{-1} \equiv \frac{\alpha}{z}$, which diverge for $|x^{-1}|>1 \Leftrightarrow |z| < |\alpha|$. In this case, we clearly do not have uniform convergence with respect to $\alpha$ and we need to use Glutsyuk's Theorem \ref{theorem:gl}. 

The fundamental set of solutions (\ref{eq:yinf}) are written in canonical form. However, we will rewrite the solution $y_{1}^{(\infty)}(x)$ using one of Kummer relations as follows,
\begin{align} y_{1}^{(\infty)}(x) &= (-x)^{-\alpha} \ _{2}F_{1} \left(\begin{array}{c} \alpha , \ \alpha+1-\gamma \\ \alpha+1-\beta \end{array} ; x^{-1} \right), &&x \in \Omega_{\infty}, \nonumber \\
&= (-x)^{\beta - \gamma}(1-x)^{\gamma-\alpha-\beta} \ _{2}F_{1} \left(\begin{array}{c} 1-\beta , \ \gamma-\beta \\ \alpha+1-\beta \end{array} ; x^{-1} \right), &&x \in \widehat{\Omega}_{\infty}, \label{eq:kummerrelation} \end{align}
where the new domain $\widehat{\Omega}_{\infty}$ is defined as,
\[\widehat{\Omega}_{\infty} = \left\{x : |x|>1, \ -\pi \leq \text{arg}(-x) < \pi, \ -\pi \leq \text{arg}(1-x) < \pi\right\}.\]
There is no need to rewrite the solution $y_{2}^{(\infty)}(x)$ as given in (\ref{eq:yinf}) as it is already in a suitable form, this is explained in Lemma \ref{lemma:obtaining} below. We note that the above two forms of the solution $y_{1}^{(\infty)}(x)$ are equivalent on the domain $\Omega_{\infty} \cap \widehat{\Omega}_{\infty}$. The condition imposed on arg$(1-x)$ in $\widehat{\Omega}_{\infty}$ is only necessary to deal with the term $(1-x)^{\gamma-\alpha-\beta}$. After making the substitution $x=\frac{z}{\alpha}$ and taking the limit $\alpha \rightarrow \infty$we have
\begin{align} 
\left(1-\frac{z}{\alpha}\right)^{\gamma-\alpha-\beta} &= \text{exp}\left((\gamma-\alpha-\beta)\log\left(1-\frac{z}{\alpha}\right) \right), \nonumber \\
&= \text{exp}\left((\gamma-\alpha-\beta) \left( -\frac{z}{\alpha} + \mathcal{O}\left(\alpha^{-2}\right)\right)\right), \nonumber \\
&= e^{z} \left(1+\mathcal{O}\left(\alpha^{-1}\right)\right). \label{eq:k3} 
\end{align}
This computation shows how to asymptotically pass from power-like behaviour to exponential behaviour as $\alpha \rightarrow \infty$. Moreover, with this new form of $y_{1}^{(\infty)}(x)$ we are ready to state the following lemma. 

\begin{lemma} \label{lemma:obtaining} Let $y_{2}^{(\infty)}(x)$ be given by (\ref{eq:yinf}) and $y_{1}^{(\infty)}(x)$ be given in its new form by (\ref{eq:kummerrelation}). After the substitution $x=\frac{z}{\alpha}$, the terms of these series tend to the terms in the formal series solutions $\tilde{y}_{1,f}^{(\infty)}(z)$ and $\tilde{y}_{2,f}^{(\infty)}(z)$ as given by (\ref{eq:yformal}), namely we have the following limits: 
\begin{align}&\lim_{\alpha \rightarrow \infty} \frac{(1-\beta)_{n}(\gamma-\beta)_{n}\alpha^{n}}{(\alpha+1-\beta)_{n}n!z^{n}} = \frac{(\gamma-\beta)_{n}(1-\beta)_{n}}{n!z^{n}}, \nonumber \\
&\lim_{\alpha \rightarrow \infty} \frac{(\beta)_{n} (\beta+1-\gamma)_{n}\alpha^{n}}{(\beta+1-\alpha)_{n}n!z^{n}} = (-1)^{n}\frac{(\beta)_{n}(\beta+1-\gamma)_{n}}{n!z^{n}}. \nonumber \end{align} \end{lemma}
\begin{proof} By direct computation, using 
\begin{align} \frac{\alpha^{n}}{(\alpha+1-\beta)_{n}} = 1 + \mathcal{O}\left(\alpha^{-1}\right) \quad \text{and} \quad \frac{\alpha^{n}}{(\beta+1-\alpha)_{n}} = (-1)^{n} + \mathcal{O}\left(\alpha^{-1}\right).  \nonumber \end{align} \end{proof} 

\begin{remark} \label{remark:r1} Lemma \ref{lemma:obtaining} is stated in terms of the solutions of the \textit{scalar} hypergeometric equations (\ref{eq:gauss}) and (\ref{eq:kummer}). From the viewpoint of working with the $(2 \times 2)$ equations (\ref{eq:hg1}) and (\ref{eq:chg1}), we rewrite the solution $Y^{(\infty)}(x)$, as given in (\ref{eq:infinity}), as follows,
\begin{align} Y^{(\infty)}(x) &= R_{\infty} \sum_{n=0}^{\infty} g_{n,\infty}x^{-n} (-x)^{-\Theta_{\infty}}, &&x \in \Omega_{\infty}, \nonumber \\
&= R_{\infty} \sum_{n=0}^{\infty} \widehat{g}_{n,\infty}x^{-n} (-x)^{-\Theta_{\infty}-\Theta_{1}}(1-x)^{\Theta_{1}}, &&x \in \widehat{\Omega}_{\infty}, \label{eq:kummerrelation2} \end{align}
where $\widehat{g}_{0,\infty}=I$ and we find all other coefficients $\widehat{g}_{n,\infty}$, $n \geq 1$, from the recursive relation,
\[\  n \widehat{g}_{n,\infty} + \left[\widehat{g}_{n, \infty}, \Theta_{\infty Y} \right] = -R_{\infty Y}^{-1}A_{1Y}R_{\infty Y}\sum_{l=0}^{n-1}\widehat{g}_{l,\infty} + \sum_{l=0}^{n-1}\widehat{g}_{l,\infty}\Theta_{1}.\]
This recursion equation only differs from that for $g_{n,\infty}$, given in the proof of Lemma \ref{lemma:cup}, by the final summation term. We find the solution to this equation is, 
\begin{align} \widehat{g}_{n,\infty} = \left(\begin{array}{cc} \frac{(1-\beta)_{n}(\gamma-\beta)_{n}}{(\alpha+1-\beta)_{n}n!} & - \frac{(\beta)_{n-1}(\beta+1-\gamma)_{n-1}}{(\beta+1-\alpha)_{n-1}(n-1)!} \\ \frac{\alpha(1-\beta)(\beta-\gamma)(\alpha+1-\gamma)}{(\alpha-\beta)(\alpha+1-\beta)^{2}(\alpha+2-\beta)} \frac{(2-\beta)_{n-1}(\gamma+1-\beta)_{n-1}}{(\alpha+3-\beta)_{n-1}(n-1)!} & \frac{(\beta-1)_{n}(\beta-\gamma)_{n}}{(\beta-\alpha-1)_{n}n!} \end{array} \right). \label{eq:gninf} \end{align}
The transformation (\ref{eq:kummerrelation2}) is analogous to Kummer relation (\ref{eq:kummerrelation}). We note that, 
\begin{align} Y^{(\infty)}\left(\frac{z}{\alpha}\right) &= R_{\infty} \sum_{n=0}^{\infty} \widehat{g}_{n,\infty}\alpha^{n}z^{-n} \left(\begin{array}{cc} (-\alpha)^{\gamma-\beta}z^{\beta-\gamma}\left(1-\frac{z}{\alpha}\right)^{\gamma-\alpha-\beta} & 0 \\ 0 & (-\alpha)^{\beta-1}z^{1-\beta} \end{array} \right), \nonumber \\
&\equiv R_{\infty} \left(\begin{array}{cc} 1 & 0 \\ 0 & \alpha^{-1} \end{array} \right) \left(\begin{array}{cc} 1 & 0 \\ 0 & \alpha \end{array} \right) \sum_{n=0}^{\infty} \widehat{g}_{n,\infty}\alpha^{n}z^{-n} \left(\begin{array}{cc} 1 & 0 \\ 0 & \alpha^{-1} \end{array} \right) \nonumber \\
&\quad \quad \quad \quad \quad \quad \quad \quad \left(\begin{array}{cc} z^{\beta-\gamma}\left(1-\frac{z}{\alpha}\right)^{\gamma-\alpha-\beta} & 0 \\ 0 & z^{1-\beta} \end{array} \right) \left(\begin{array}{cc} (-\alpha)^{\gamma-\beta} & 0 \\ 0 & -(-\alpha)^{\beta} \end{array} \right). \nonumber \end{align}
The limits analogous to those in Lemma (\ref{lemma:obtaining}) are stated as follows: we have the following limit of the leading matrix, 
\begin{align} \lim_{\alpha \rightarrow \infty} R_{\infty} \left(\begin{array}{cc} 1 & 0 \\ 0 & \alpha^{-1} \end{array} \right) &= \lim_{\alpha \rightarrow \infty} \left(\begin{array}{cc} 1 & 0 \\ 0 & \frac{(\beta-\alpha)(\alpha+1-\beta)}{\alpha(\beta-1)(\beta-\gamma)}\end{array}\right)\left(\begin{array}{cc} 1 & 0 \\ 0 & \alpha^{-1} \end{array} \right), \nonumber \\
&= \left(\begin{array}{cc} 1 & 0\\ 0 & \frac{-1}{(\beta-1)(\beta-\gamma)}\end{array} \right) = \widetilde{R}_{\infty}, \nonumber \end{align}
and for the terms of the new series,
\[\lim_{\alpha \rightarrow \infty} \left(\begin{array}{cc} 1 & 0 \\ 0 & \alpha \end{array} \right) \alpha^{n} \widehat{g}_{n,\infty} \left(\begin{array}{cc} 1 & 0 \\ 0 & \alpha^{-1} \end{array} \right) = h_{n,\infty},\]
where $\widehat{g}_{n,\infty}$ and $h_{n,\infty}$ are given by (\ref{eq:gninf}) and (\ref{eq:hninf}) respectively. Hence, we understand that a \textit{term-by-term} limit of the solution,
\[Y^{(\infty)}\left(\frac{z}{\alpha}\right) \ \left(\begin{array}{cc} (-\alpha)^{\beta-\gamma} & 0 \\ 0 & -(-\alpha)^{-\beta}\end{array} \right),\]
produces the formal solution $\widetilde{Y}^{(\infty)}_{f}(z)$, which is analogous to (\ref{eq:chgasy2}). \end{remark}

We now turn our attention to the fundamental solutions at $x=1$, as given in canonical form in (\ref{eq:y1}). Observe that these solutions are expressed using Gauss hypergeometric $\ _{2}F_{1}$ series in the variable $(1-x) \equiv (1-\frac{z}{\alpha})$, which diverge for $|1-x|>1 \Leftrightarrow |z-\alpha|>|\alpha|$. As with the fundamental solutions at $x=\infty$, we do not have uniform convergence with respect to $\alpha$ here. Rather than keeping these solutions in canonical form, we use two more of Kummer relations to rewrite them as follows,
\begin{align} y_{1}^{(1)}(x) &= (1-x)^{\gamma-\alpha-\beta} \ _{2}F_{1}\left(\begin{array}{c} \gamma-\alpha , \ \gamma - \beta \\ \gamma+1-\alpha-\beta \end{array} ; 1-x \right) &&x \in \Omega_{1}, \nonumber \\
&= x^{\beta-\gamma}(1-x)^{\gamma-\alpha-\beta} \ _{2}F_{1}\left(\begin{array}{c} \gamma-\beta , \ 1-\beta \\ \gamma+1-\alpha-\beta \end{array} ; 1-x^{-1} \right), &&x \in \widehat{\Omega}_{1}, \label{eq:kummerrelation3} \\
y_{2}^{(1)}(x) &= \ _{2}F_{1} \left(\begin{array}{c} \alpha , \ \beta \\ \alpha+\beta+1-\gamma \end{array} ; 1-x \right) &&x \in \Omega_{1}, \nonumber \\
&= x^{-\beta} \ _{2}F_{1} \left(\begin{array}{c} \beta+1-\gamma , \ \beta \\ \alpha+\beta+1-\gamma \end{array} ; 1-x^{-1} \right), &&x \in \widehat{\Omega}_{1}, \label{eq:kummerrelation4} \end{align}
where the new domain $\widehat{\Omega}_{1}$ is defined as,
\[\widehat{\Omega}_{1} = \left\{ x : \left|1-x^{-1}\right| < 1, \ -\pi \leq \text{arg}(x) < \pi, \ -\pi \leq \text{arg}(1-x) < \pi\right\}.\]
We note that the two forms of these solutions are equivalent on the domain $\Omega_{1}\cap\widehat{\Omega}_{1}$. There is a very simple philosophical reason why we rewrite the series in these solutions with $(1-x^{-1})^{n}$, rather than $(1-x)^{n}$: after the change of variable $x=\frac{z}{\alpha}$, we want to produce a formal series in $z^{-n}$. Similarly as before, the computations ending in (\ref{eq:k3}) show how the solution $y_{1}^{(1)}(x)$ asymptotically passes from power-like behaviour to exponential behaviour as $\alpha \rightarrow \infty$. Moreover, the terms of the series in these new forms of $y_{1}^{(1)}(x)$ and $y_{2}^{(1)}(x)$ satisfy the lemma below.
\begin{lemma} \label{lemma:obtaining2} Let $y_{1}^{(1)}(x)$ and $y_{2}^{(1)}(x)$ be given in their new forms by (\ref{eq:kummerrelation3}) and (\ref{eq:kummerrelation4}) respectively. After the substitution $x = \frac{z}{\alpha}$, the terms of these series tend to the terms in the formal series solutions $\tilde{y}_{1,f}^{(\infty)}(z)$ and $\tilde{y}_{2,f}^{(\infty)}(z)$ as given by (\ref{eq:yformal}), namely we have the following limits:
\begin{align} &\lim_{\alpha \rightarrow \infty} \frac{(\gamma-\beta)_{n}(1-\beta)_{n}(z-\alpha)^{n}}{(\gamma+1-\alpha-\beta)_{n}n!z^{n}} = \frac{(\gamma-\beta)_{n}(1-\beta)_{n}}{n!z^{n}}, \nonumber \\
&\lim_{\alpha \rightarrow \infty} \frac{(\beta+1-\gamma)_{n}(\beta)_{n}(z-\alpha)^{n}}{(\alpha+\beta+1-\gamma)_{n}n!z^{n}} = (-1)^{n}\frac{(\beta)_{n}(\beta+1-\gamma)_{n}}{n!z^{n}}. \nonumber \end{align} \end{lemma}
\begin{proof} By direct computation, after expanding the powers of $(z-\alpha)$ and the Pochhammer symbols to find,
\begin{align} \frac{(z-\alpha)^{n}}{(\gamma+1-\alpha-\beta)_{n}} = 1 + \mathcal{O}\left(\alpha^{-1}\right) \quad \text{and} \quad \frac{(z-\alpha)^{n}}{(\alpha+\beta+1-\gamma)_{n}} = (-1)^{n} + \mathcal{O}\left(\alpha^{-1}\right). \nonumber \end{align} \end{proof}
This lemma shows that \textit{term-by-term} limits of the solutions,
\begin{align} y^{(1)}_{1}(z \alpha^{-1}) \ \alpha^{\beta-\gamma} \quad \text{and} \quad -y_{2}^{(1)}(z \alpha^{-1}) \ \alpha^{-\beta}, \label{eq:analogous2} \end{align}
produce the formal solutions,
\[\tilde{y}_{1,f}^{(\infty)}(z) \quad \text{and} \quad \tilde{y}_{2,f}^{(\infty)}(z),\]
respectively. The factors $\alpha^{\beta-\gamma}$ and $\alpha^{-\beta}$ in (\ref{eq:analogous2}) are necessary because of the terms,
\[x^{\beta-\gamma} \equiv z^{\beta-\gamma} \alpha^{\gamma-\beta} \quad \text{and} \quad x^{-\beta} \equiv z^{-\beta} \alpha^{\beta},\]
in the solutions $y_{1}^{(1)}(x)$ and $y_{2}^{(1)}(x)$ respectively. We note that the direction in which $\alpha \rightarrow \infty$ is not yet important for this lemma. The importance of this lemma is shown in the proof of our Main Theorem \ref{main:importance}. 

\begin{remark} \label{remark:r2} Similarly as in Remark \ref{remark:r1}, we may consider the viewpoint of working with the $(2\times2)$ equations (\ref{eq:hg1}) and (\ref{eq:chg1}) and rewrite the solution $Y^{(1)}(x)$, as given in (\ref{eq:ak2}), as follows,
\begin{align} Y^{(1)}(x) &= R_{1} \sum_{n=0}^{\infty} g_{n,1} (1-x)^{n} (1-x)^{\Theta_{1}}, &&x \in \Omega_{1}, \nonumber \\
&= R_{1} \sum_{n=0}^{\infty} \widehat{g}_{n,1} \left(1-x^{-1}\right)^{n} x^{-\Theta_{\infty}-\Theta_{1}}(1-x)^{\Theta_{1}}, &&x \in \widehat{\Omega}_{1}, \label{eq:kummerrelation5} \end{align}
where $\widehat{g}_{0,1} = I$ and we find all other coefficients $\widehat{g}_{n,1}$, $n \geq 1$, from the recursive equation,
\[\left[ \widehat{g}_{n,1} , \Theta_{1} \right] + n \widehat{g}_{n,1} = (n-1) \widehat{g}_{n-1,1} + \widehat{g}_{n-1,1}(\Theta_{1}+\Theta_{\infty}) + R_{1}^{-1}A_{0}R_{1}\widehat{g}_{n-1,1}.\]
This recursion equation differs quite significantly from that for $g_{n,1}$, given in the proof of Lemma \ref{lemma:cup}. We find the solution to this equation is, 
\begin{align} \widehat{g}_{n,1} &= \left(\begin{array}{cc} \frac{(1-\beta)_{n}(\gamma-\beta)_{n}}{(\gamma+1-\alpha-\beta)_{n}n!} & \frac{(\beta)_{n}(\beta+1-\gamma)_{n}}{(\alpha+\beta+1-\gamma)_{n}n!}-\frac{(\beta)_{n-1}(\beta+1-\gamma)_{n-1}}{(\alpha+\beta+1-\gamma)_{n-1}(n-1)!} \\ \frac{1}{\alpha} \left(\frac{(2-\beta)_{n}(\gamma+1-\beta)_{n}}{(\gamma+1-\alpha-\beta)_{n}n!}-\frac{(2-\beta)_{n-1}(\gamma+1-\beta)_{n-1}}{(\gamma+1-\alpha-\beta)_{n-1}(n-1)!}\right) & \frac{\alpha+1-\gamma}{(\beta-1)(\beta-\gamma)}\frac{(\beta-1)_{n}(\beta-\gamma)_{n}}{(\alpha+\beta+1-\gamma)_{n}n!} \end{array} \right). \label{eq:gn1} \end{align}
The transformation (\ref{eq:kummerrelation5}) is analogous to Kummer relations (\ref{eq:kummerrelation3}) and (\ref{eq:kummerrelation4}). We note that,
\begin{align} Y^{(1)}\left(\frac{z}{\alpha}\right) &= R_{1} \sum_{n=0}^{\infty} \widehat{g}_{n,1}\left(1-\frac{\alpha}{z}\right)^{n} \left(\begin{array}{cc} \alpha^{\gamma-\beta}z^{\beta-\gamma}\left(1-\frac{z}{\alpha}\right)^{\gamma-\alpha-\beta} & 0 \\ 0 & \alpha^{\beta-1}z^{1-\beta} \end{array} \right), \nonumber \\
&\equiv R_{1} \left(\begin{array}{cc} 1 & 0 \\ 0 & -\alpha^{-1} \end{array} \right)\left(\begin{array}{cc} 1 & 0 \\ 0 & -\alpha \end{array} \right), \sum_{n=0}^{\infty} \widehat{g}_{n,1} \left(1-\frac{\alpha}{z}\right)^{n} \left(\begin{array}{cc} 1 & 0 \\ 0 & -\alpha^{-1} \end{array} \right) \nonumber \\
&\quad \quad \quad \quad \quad \quad \quad \quad \left(\begin{array}{cc} z^{\beta-\gamma}\left(1-\frac{z}{\alpha}\right)^{\gamma-\alpha-\beta} & 0 \\ 0 & z^{1-\beta} \end{array} \right) \left(\begin{array}{cc} \alpha^{\gamma-\beta} & 0 \\ 0 & -\alpha^{\beta} \end{array} \right). \nonumber \end{align}
The limits analogous to those in Lemma \ref{lemma:obtaining2} are stated as follows: we have the following limit of the leading matrix,
\begin{align} \lim_{\alpha \rightarrow \infty} R_{1} \left(\begin{array}{cc} 1 & 0 \\ 0 & -\alpha^{-1}\end{array} \right) &= \lim_{\alpha \rightarrow \infty} \left(\begin{array}{cc} 1 & 1 \\ \frac{1}{\alpha} & \frac{\alpha+1-\gamma}{(\beta-1)(\beta-\gamma)}\end{array} \right) \left(\begin{array}{cc} 1 & 0 \\ 0 & -\alpha^{-1} \end{array} \right), \nonumber \\
&= \left(\begin{array}{cc} 1 & 0 \\ 0 & \frac{-1}{(\beta-1)(\beta-\gamma)} \end{array} \right) = \widetilde{R}_{\infty},\nonumber \end{align}
and for the terms of the new series, 
\begin{align} \lim_{\alpha \rightarrow \infty} \left(\begin{array}{cc} 1 & 0 \\ 0 & - \alpha \end{array} \right) (-\alpha)^{n} \widehat{g}_{n,1} \left(\begin{array}{cc} 1 & 0 \\ 0 & -\alpha^{-1} \end{array} \right) = h_{n,\infty},\nonumber \end{align}
where $\widehat{g}_{n,1}$ and $h_{n,\infty}$ are given by (\ref{eq:gn1}) and (\ref{eq:hninf}) respectivey. Hence, we understand that a \textit{term-by-term} limit of the solution,
\[Y^{(1)}\left(\frac{z}{\alpha}\right) \ \left(\begin{array}{cc} \alpha^{\beta-\gamma} & 0 \\ 0 & -\alpha^{-\beta} \end{array} \right),\]
produces the formal solution $\widetilde{Y}^{(\infty)}_{f}(z)$, which is analogous to (\ref{eq:analogous2}). \end{remark}

Having understood how to take term-by-term limits of the series solutions of Gauss equation around $x=1$ and $\infty$ to produce the formal solutions of Kummer equation around $z=\infty$, we now show how to apply Glutsyuk's Theorem \ref{theorem:gl} to Gauss hypergeometric equation. Let $\eta \in \left(0,\frac{\pi}{2}\right)$ be some fixed value. We define the following sectors,
\begin{align} \widetilde{\mathscr{S}}_{k} &:= \left\{ z : \text{arg}(z) -k \pi \in \left(\eta-\frac{\pi}{2} , \frac{3\pi}{2} - \eta \right)\right\}, \label{eq:sectors3} \end{align}
we note that if $z \in \widetilde{\mathscr{S}}_{k}$ then $z \in \widetilde{\Sigma}_{k}$. The presence of $\eta$ is to ensure that the boundaries of the sectors $\widetilde{\mathscr{S}}_{k}$ do not contain a Stokes ray, as is necessary in the hypothesis of Glutsyuk's Theorem \ref{theorem:gl}. We note that this condition is not satisfied by the sectors $\widetilde{\Sigma}_{k}$ defined in Theorem \ref{theorem:kummerst}, which are the maximal sectors on which we can define single-valued analytic fundamental solutions. \\

We also define the following sectors,
\begin{align} \sigma_{\alpha}(\alpha) &:= \left\{ z : \begin{array}{c} \left|1-\frac{\alpha}{z}\right|< |\alpha|^{2} , \ \text{arg}\left(\frac{z}{\alpha}\right) \in (\eta-\pi,\pi-\eta), \\ \text{arg}\left(1-\frac{z}{\alpha}\right) \in (\eta-\pi , \pi - \eta) \end{array} \right\}, \label{eq:sectors1} \\
\sigma_{\infty}(\alpha) &:= \left\{ z : \begin{array}{c} \text{arg}\left(-z\alpha^{-1}\right) \in (\eta -\pi , \pi-\eta), \\ \text{arg}\left(1-\frac{z}{\alpha}\right)\in(\eta-\pi,\pi-\eta) \end{array} \right\}. \label{eq:sectors2} \end{align}
We note that if $z$ is sufficiently close to $\alpha$ with $z \in \sigma_{\alpha}(\alpha)$ then $x = \frac{z}{\alpha} \in \widehat{\Omega}_{1}$ and if $z$ is sufficiently large with $z \in \sigma_{\infty}(\alpha)$ then $x = \frac{z}{\alpha} \in \widehat{\Omega}_{\infty}$. These sectors will be the new domains of our solutions $y_{1}^{(1)}(z \alpha^{-1})$, $y_{2}^{(1)}(z\alpha^{-1})$ and $y_{1}^{(\infty)}(z \alpha^{-1})$, $y_{2}^{(\infty)}(z\alpha^{-1})$ respectively, they are illustrated below.

\begin{figure}[H] \begin{center}
\includegraphics[scale=1]{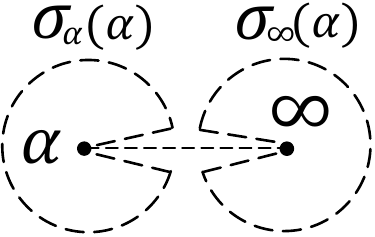}
\caption{\label{fig:sss}Sectors $\sigma_{\alpha}(\alpha)$ and $\sigma_{\infty}(\alpha)$.}
\end{center} \end{figure} 

Compared with the domains $\widehat{\Omega}_{1}$ and $\widehat{\Omega}_{\infty}$, which are disks with branch cuts, the sectors $\sigma_{\alpha}(\alpha)$ and $\sigma_{\infty}(\alpha)$ have larger radii and do not contain any part of the branch cut between $\alpha$ and $\infty$. We can analytically extend our solutions $y_{k}^{(1)}(z \alpha^{-1})$ and $y_{k}^{(\infty)}(z \alpha^{-1})$, $k=1,2$, to these larger domains because the singularity $z=\infty$ (resp. $z=\alpha$) can never lie inside the sector $\sigma_{\alpha}(\alpha)$ (resp. $\sigma_{\infty}(\alpha)$) or on its boundary. That is the key reason to restrict our solutions to sectors rather than disks. \\

We examine the sector $\sigma_{\alpha}(\alpha)$ more closely. From the first condition,
\[\left|1-\frac{\alpha}{z} \right| < |\alpha|^{2} \quad \Leftrightarrow \quad \left|\frac{1}{\alpha}-\frac{1}{z}\right| < |\alpha|,\]
observe that as $\alpha \rightarrow \infty$ the radius of this sector becomes infinite, indeed the above inequality becomes simply $|z|>0$. Furthermore, as $\alpha \rightarrow \infty$ along a ray, the base point of the sector $\sigma_{\alpha}(\alpha)$ is translated along that ray, tending to infinity. We illustrate this phenomenon in Figure \ref{fig:new} below. 

\begin{figure}[H] \begin{center}
\includegraphics[scale=.7]{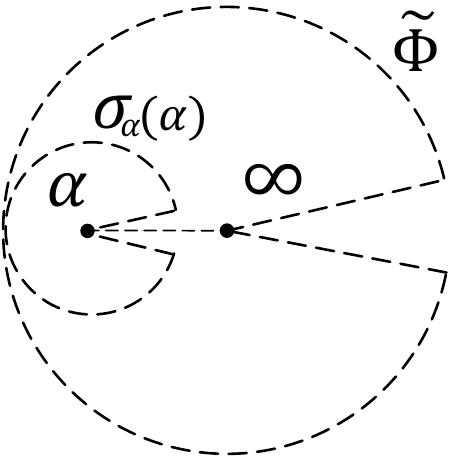}
\caption{\label{fig:new}As $\alpha \rightarrow \infty$ along a ray, the sector $\sigma_{\alpha}(\alpha)$ is translated along the branch cut and becomes in agreement with the sector $\widetilde{\Phi} := \left\{ z : \left| \text{arg}\left(\frac{z}{\alpha}\right)\right|<\pi-\eta \right\}$.}
\end{center} \end{figure} 

In the two limit directions we are concerned with, for $\arg(\alpha)=\pm\frac{\pi}{2}$, we have,
\begin{align} &\text{arg}\left(\frac{z}{\alpha}\right) \in \left(\eta-\pi, \pi-\eta \right) \quad \Leftrightarrow \quad \text{arg}(z) \in \left(\eta - \pi \pm \frac{\pi}{2} , \pi \pm \frac{\pi}{2} - \eta \right), \nonumber \end{align}
For the sector $\sigma_{\infty}(\alpha)$, whose base point is already fixed at infinity, we have, 
\begin{align} &\text{arg}\left(-\frac{z}{\alpha} \right) \in \left(\eta - \pi , \pi - \eta \right) \quad \Leftrightarrow \quad \text{arg}(z) \in \left(\eta \pm \frac{\pi}{2} , 2 \pi \pm \frac{\pi}{2} - \eta \right), \nonumber \end{align}
recall from (\ref{eq:k3}) that the condition on arg$\left(1-\frac{z}{\alpha}\right)$ in $\sigma_{\infty}(\alpha)$ does not play a role after taking the limit. With these considerations in mind, we write, 
\begin{align} \lim_{\substack{\alpha \rightarrow \infty \\ \text{arg}(\alpha) = -\frac{\pi}{2}}} \sigma_{\alpha}(\alpha) &= \widetilde{\mathscr{S}}_{-1}, &&\lim_{\substack{\alpha \rightarrow \infty \\ \text{arg}(\alpha) = -\frac{\pi}{2}}} \sigma_{\infty}(\alpha) = \widetilde{\mathscr{S}}_{0}, \nonumber \\
\lim_{\substack{\alpha \rightarrow \infty \\ \text{arg}(\alpha) = \frac{\pi}{2}}} \sigma_{\alpha}(\alpha) &= \widetilde{\mathscr{S}}_{0}, &&\lim_{\substack{\alpha \rightarrow \infty \\ \text{arg}(\alpha) = \frac{\pi}{2}}} \sigma_{\infty}(\alpha) = \widetilde{\mathscr{S}}_{1}. \nonumber \end{align}

We now apply Glutsyuk's Theorem \ref{theorem:gl} with the $(2\times2)$ hypergeometric equation (\ref{eq:hg1}) in place of the perturbed equation and the confluent hypergeometric equation (\ref{eq:chg1}) in place of the non-perturbed equation. Glutsyuk's theorem asserts the existence of invertible diagonal matrices $K^{\pm}_{\infty}(\alpha)$ and $K^{\pm}_{1}(\alpha)$ such that:
\begin{align} &\lim_{\substack{\alpha \rightarrow \infty \\ \text{arg}(\alpha) = -\frac{\pi}{2}}} \left. Y^{(1)}\left(z \alpha^{-1}\right) \right|_{z \in \sigma_{\alpha}(\alpha)} K^{-}_{1}(\alpha) = \widetilde{Y}^{(\infty,-1)}(z), \label{eq:sat5} \\
&\lim_{\substack{\alpha \rightarrow \infty \\ \text{arg}(\alpha) = -\frac{\pi}{2}}} \left. Y^{(\infty)}\left(z \alpha^{-1}\right) \right|_{z \in \sigma_{\infty}(\alpha)} K^{-}_{\infty}(\alpha) = \widetilde{Y}^{(\infty,0)}(z), \label{eq:sat6} \end{align}
uniformly for $z \in \widetilde{\mathscr{S}}_{-1}$ and $z \in \widetilde{\mathscr{S}}_{0}$ respectively, and:
\begin{align} &\lim_{\substack{\alpha \rightarrow \infty \\ \text{arg}(\alpha) = \frac{\pi}{2}}} \left. Y^{(1)}\left(z \alpha^{-1}\right) \right|_{z \in \sigma_{\alpha}(\alpha)} K^{+}_{1}(\alpha) = \widetilde{Y}^{(\infty,0)}(z), \label{eq:sat7} \\
&\lim_{\substack{\alpha \rightarrow \infty \\ \text{arg}(\alpha) =\frac{\pi}{2}}} \left. Y^{(\infty)}\left(z \alpha^{-1}\right) \right|_{z \in \sigma_{\infty}(\alpha)} K^{+}_{\infty}(\alpha) = \widetilde{Y}^{(\infty,1)}(z), \label{eq:sat8} \end{align}
uniformly for $z \in \widetilde{\mathscr{S}}_{0}$ and $z \in \widetilde{\mathscr{S}}_{1}$ respectively. We note that since we are considering two limits, namely one with $\arg(\alpha) = \frac{\pi}{2}$ and another with $\arg(\alpha) = -\frac{\pi}{2}$, we have distinguished the diagonal matrices in each case with a superscript $+$ or $-$ respectively. Due to the asymptotics of the fundamental solutions of Kummer equation as given in Theorem \ref{theorem:kummerst}, each of these four limits is asymptotic to the formal fundamental solution $\widetilde{Y}_{f}^{(\infty)}(z)$ as $z \rightarrow \infty$ with $z$ belonging to the corresponding sector. \\

Equivalently, from the viewpoint of studying the classical scalar hypergeometric equations (\ref{eq:gauss}) and (\ref{eq:kummer}), Glutsyuk's Theorem \ref{theorem:gl} asserts the existence of scalars $k_{1,\infty}^{\pm}(\alpha)$, $k_{2,\infty}^{\pm}(\alpha)$, $k_{1,1}^{\pm}(\alpha)$ and $k_{2,1}^{\pm}(\alpha)$ such that, for $j \in \{1,2\}$:
\begin{align} &\lim_{\substack{\alpha \rightarrow \infty \\ \text{arg}(\alpha) = -\frac{\pi}{2}}} \left. y_{j}^{(1)}(z \alpha^{-1})\right|_{z \in \sigma_{\alpha}(\alpha)} \ k_{j,1}^{-}(\alpha) = \tilde{y}_{j}^{(\infty,-1)}(z), \label{eq:sat1} \\ 
&\lim_{\substack{\alpha \rightarrow \infty \\ \text{arg}(\alpha) = -\frac{\pi}{2}}} \left. y_{j}^{(\infty)}(z \alpha^{-1})\right|_{z \in \sigma_{\infty}(\alpha)} \ k_{j,\infty}^{-}(\alpha) = \tilde{y}_{j}^{(\infty,0)}(z), \label{eq:sat2} \end{align}
uniformly for $z \in \widetilde{\mathscr{S}}_{-1}$ and $\widetilde{\mathscr{S}}_{0}$ respectively, and: 
\begin{align} &\lim_{\substack{\alpha \rightarrow \infty \\ \text{arg}(\alpha) = \frac{\pi}{2}}} \left. y_{j}^{(1)}(z \alpha^{-1})\right|_{z \in \sigma_{\alpha}(\alpha)} \ k_{j,1}^{+}(\alpha) = \tilde{y}_{j}^{(\infty,0)}(z), \label{eq:sat3} \\ 
&\lim_{\substack{\alpha \rightarrow \infty \\ \text{arg}(\alpha) = \frac{\pi}{2}}} \left. y_{j}^{(\infty)}(z \alpha^{-1})\right|_{z \in \sigma_{\infty}(\alpha)} \ k_{j,\infty}^{+}(\alpha) = \tilde{y}_{j}^{(\infty,1)}(z), \label{eq:sat4} \end{align}
uniformly $z \in \widetilde{\mathscr{S}}_{0}$ and $\widetilde{\mathscr{S}}_{1}$ respectively. \\

Having applied Glutsyuk's theorem to our confluence of the hypergeometric equation, we now focus on understanding what we can deduce about these scalars $k_{j,\infty}^{\pm}(\alpha)$ and $k_{j,1}^{\pm}(\alpha)$, $j=1,2$. We are ready to state our first main theorem. 

\begin{theorem} \label{main:importance} If $k_{j,\infty}^{\pm}(\alpha)$ and $k_{j,1}^{\pm}(\alpha)$ are scalars satisfying (\ref{eq:sat1})-(\ref{eq:sat4}), then these numbers satisfy the following limits,
\begin{align} 
\lim_{ \substack{ \alpha\rightarrow \infty \\ \text{arg}(\alpha) = \pm \frac{\pi}{2}}} k_{1,\infty}^{\pm}(\alpha) \ (-\alpha)^{\gamma-\beta}= 1, \label{eq:n1} \\
\lim_{ \substack{ \alpha\rightarrow \infty \\ \text{arg}(\alpha) = \pm \frac{\pi}{2}}} -k_{2,\infty}^{\pm}(\alpha) \ (-\alpha)^{\beta}= 1, \label{eq:n2} \\
\lim_{ \substack{ \alpha\rightarrow \infty \\ \text{arg}(\alpha) = \pm \frac{\pi}{2}}} k_{1,1}^{\pm}(\alpha) \ \alpha^{\gamma-\beta}= 1, \label{eq:n3} \\
\lim_{ \substack{ \alpha\rightarrow \infty \\ \text{arg}(\alpha) = \pm \frac{\pi}{2}}} -k_{2,1}^{\pm}(\alpha) \ \alpha^{\beta}= 1. \label{eq:n4} \end{align}\end{theorem}

\begin{proof} In either case $\arg(\alpha) = \frac{\pi}{2}$ or $-\frac{\pi}{2}$, let $\mathscr{S}^{*}$ be a closed, proper subsector of $\widetilde{\mathscr{S}}_{1}$ or $\widetilde{\mathscr{S}}_{0}$ respectively. Combining the statements (\ref{eq:sat2}) and (\ref{eq:sat4}), together with the asymptotic behaviour (\ref{eq:chgasy}), we have,
\begin{align} \lim_{\substack{\alpha \rightarrow \infty \\ \text{arg}(\alpha) = \pm \frac{\pi}{2}}} \left. y_{1}^{(\infty)}(z \alpha^{-1})\right|_{z \in \sigma_{\infty}(\alpha)} \ k_{1,\infty}^{\pm}(\alpha) \sim \tilde{y}_{1,f}^{(\infty)}(z), \quad \text{as } z \rightarrow \infty, \ z \in \mathscr{S}^{*}. \label{eq:combining} \end{align}
We now re-write $y_{1}^{(\infty)}(z \alpha^{-1})$ using Kummer transformation as in (\ref{eq:kummerrelation}),
\begin{align} &\left. y_{1}^{(\infty)}\left(z \alpha^{-1}\right) \right|_{z \in \sigma_{\infty}(\alpha)} = z^{\beta-\gamma}(-\alpha)^{\gamma-\beta}\left(1-\frac{z}{\alpha}\right)^{\gamma-\alpha-\beta} \left. \sum_{n=0}^{\infty}\frac{(1-\beta)_{n}(\gamma-\beta)_{n}\alpha^{n}}{(\alpha+1-\beta)_{n}n!z^{n}} \right|_{z \in \sigma_{\infty}(\alpha)}. \nonumber \end{align}
We therefore deduce,
\begin{align} &\lim_{\substack{\alpha \rightarrow \infty \\ \text{arg}(\alpha) = \pm \frac{\pi}{2}}} z^{\beta-\gamma}(-\alpha)^{\gamma-\beta}\left(1-\frac{z}{\alpha}\right)^{\gamma-\alpha-\beta} \left. \sum_{n=0}^{\infty}\frac{(1-\beta)_{n}(\gamma-\beta)_{n}\alpha^{n}}{(\alpha+1-\beta)_{n}n!z^{n}} \right|_{z \in \sigma_{\infty}(\alpha)}k_{1,\infty}^{\pm}(\alpha) \nonumber \\
&\quad = \lim_{\substack{\alpha \rightarrow \infty \\ \text{arg}(\alpha) = \pm \frac{\pi}{2}}} z^{\beta-\gamma}(-\alpha)^{\gamma-\beta}e^{z} \left. \sum_{n=0}^{\infty}\frac{(1-\beta)_{n}(\gamma-\beta)_{n}\alpha^{n}}{(\alpha+1-\beta)_{n}n!z^{n}} \right|_{z \in \sigma_{\infty}(\alpha)}k_{1,\infty}^{\pm}(\alpha). \nonumber \end{align}

Combining this with (\ref{eq:combining}) and writing $\tilde{y}_{1,f}^{(\infty)}(z)$ as in (\ref{eq:yformal}), we have,
\begin{align} \lim_{\substack{\alpha \rightarrow \infty \\ \text{arg}(\alpha) = \pm \frac{\pi}{2}}} \left. \sum_{n=0}^{\infty}\frac{(1-\beta)_{n}(\gamma-\beta)_{n}\alpha^{n}}{(\alpha+1-\beta)_{n}n!z^{n}} \right|_{z \in \sigma_{\infty}(\alpha)}(-\alpha)^{\gamma-\beta} k_{1,\infty}^{\pm} \sim \sum_{n=0}^{\infty}\frac{(\gamma-\beta)_{n}(1-\beta)_{n}}{n!z^{n}},\nonumber \end{align}
as $z \rightarrow \infty$ for $z \in \mathscr{S}^{*}$. \\

We now define $w = z^{-1}$ so that $w \rightarrow 0 \Leftrightarrow z \rightarrow \infty$ and we can apply 
the following classical result  \cite{wasow}:
\begin{lemma}
 \label{lemma:wasow} Let $f(w)$ be holomorphic in an open sector $\sigma$ at $w=0$ and let $\sigma^{*}$ be a closed, proper sub-sector of $\sigma$. If,
\[f(w) \sim \sum_{n=0}^{\infty} a_{n}w^{n}, \quad \quad \text{as } w \rightarrow 0, \ w \in \sigma,\]
then:
\[a_{n} = \frac{1}{n!} \lim_{\substack{w \rightarrow 0 \\ w \in \sigma^{*}}} f^{(n)}(z),\]
where $f^{(n)}(w)$ denotes the $n^{\text{th}}$ derivative of $f(w)$,
\end{lemma}
\noindent to find, 
\begin{align} &\frac{(\gamma-\beta)_{n}(1-\beta)_{n}}{n!} = \nonumber \\
&\quad \frac{1}{n!} \lim_{\substack{w \rightarrow 0 \\ w^{-1} \in \mathscr{S}^{*}}} \frac{d^{n}}{dw^{n}} \lim_{\substack{\alpha \rightarrow \infty \\ \text{arg}(\alpha) = \pm \frac{\pi}{2}}} \left. \sum_{l=0}^{\infty}\frac{(1-\beta)_{l}(\gamma-\beta)_{l}\alpha^{l}w^{l}}{(\alpha+1-\beta)_{l}l!} \right|_{w^{-1} \in \sigma_{\infty}(\alpha)}(-\alpha)^{\gamma-\beta} k_{1,\infty}^{\pm}(\alpha). \nonumber \end{align}
We proceed to treat the limits on the right hand side with special care. We first note that, due to the uniformity of the limits (\ref{eq:sat2}) and (\ref{eq:sat4}), we may interchange the limit in $\alpha$ with the derivative and the limit in $w$ as follows,
\begin{align} &\frac{(\gamma-\beta)_{n}(1-\beta)_{n}}{n!} = \nonumber \\
&\quad \frac{1}{n!} \lim_{\alpha \rightarrow \infty} \lim_{\substack{w \rightarrow 0 \\ w^{-1} \in \mathscr{S}^{*}}} \frac{d^{n}}{dw^{n}} \left. \sum_{l=0}^{\infty}\frac{(1-\beta)_{l}(\gamma-\beta)_{l}\alpha^{l}w^{l}}{(\alpha+1-\beta)_{l}l!} \right|_{w^{-1} \in \sigma_{\infty}(\alpha)}(-\alpha)^{\gamma-\beta} k_{1,\infty}^{\pm}(\alpha). \nonumber \end{align}
The next step is to notice that the series inside the limits on the right hand side represents an analytic function (or at least its analytic extension to the sector $\sigma_{\infty}(\eps)$ does). We may therefore interchange the derivative and series as follows,
\begin{align} &\frac{(\gamma-\beta)_{n}(1-\beta)_{n}}{n!} = \nonumber \\
&\ \frac{1}{n!} \lim_{\substack{\alpha \rightarrow \infty \\ \text{arg}(\alpha) = \pm \frac{\pi}{2}}} \lim_{\substack{w \rightarrow 0 \\ w^{-1} \in \mathscr{S}^{*}}} \left. \sum_{l=0}^{\infty}\frac{d^{n}}{dw^{n}}\frac{(1-\beta)_{l}(\gamma-\beta)_{l}\alpha^{l}w^{l}}{(\alpha+1-\beta)_{l}l!} \right|_{w^{-1} \in \sigma_{\infty}(\alpha)}(-\alpha)^{\gamma-\beta} k_{1,\infty}^{\pm}(\alpha) = \nonumber \\
&\ \frac{1}{n!} \lim_{\substack{\alpha \rightarrow \infty \\ \text{arg}(\alpha) = \pm \frac{\pi}{2}}} \lim_{\substack{w \rightarrow 0 \\ w^{-1} \in \mathscr{S}^{*}}} \left. \sum_{l=0}^{\infty}\frac{(l+n)!}{l!} \frac{(1-\beta)_{l+n}(\gamma-\beta)_{l+n}\alpha^{l+n}w^{l}}{(\alpha+1-\beta)_{l+n}(l+n)!} \right|_{w^{-1} \in \sigma_{\infty}(\alpha)}(-\alpha)^{\gamma-\beta} k_{1,\infty}^{\pm}(\alpha). \nonumber \end{align}
Furthermore, due to the analyticity of the series on the right hand side, its limit as $w \rightarrow 0$ certainly exists and is simply equal to the first term of the series. We finally deduce, 
\begin{align} \frac{(\gamma-\beta)_{n}(1-\beta)_{n}}{n!} = \frac{1}{n!} \lim_{\substack{\alpha \rightarrow \infty \\ \text{arg}(\alpha) = \pm \frac{\pi}{2}}} n! \frac{(1-\beta)_{n}(\gamma-\beta)_{n}\alpha^{n}}{(\alpha+1-\beta)_{n}n!} (-\alpha)^{\gamma-\beta} k_{1,\infty}^{\pm}(\alpha). \label{eq:fff} 
\end{align}
Therefore 
\begin{align} &\lim_{\substack{\alpha \rightarrow \infty \\ \text{arg}(\alpha) = \pm \frac{\pi}{2}}} \frac{(1-\beta)_{n}(\gamma-\beta)_{n}\alpha^{n}}{(\alpha+1-\beta)_{n}n!} (-\alpha)^{\gamma-\beta} k_{1,\infty}^{\pm}(\alpha) \nonumber \\ 
&\quad = \frac{(1-\beta)_{n}(\gamma-\beta)_{n}}{n!} \lim_{\substack{\alpha \rightarrow \infty \\ \text{arg}(\alpha) = \pm \frac{\pi}{2}}} (-\alpha)^{\gamma-\beta} k_{1,\infty}^{\pm}(\alpha). \nonumber \end{align}
Comparing with the left hand side of (\ref{eq:fff}) we deduce the desired result (\ref{eq:n1}). The limit (\ref{eq:n2}) can be proved by using $y_{2}^{(\infty)}(z \alpha^{-1})$ as given by (\ref{eq:yinf}). The limits (\ref{eq:n3}) and (\ref{eq:n4}) can be proved using $y_{1}^{(1)}(z \alpha^{-1})$ and $y_{2}^{(1)}(z \alpha^{-1})$ as given by (\ref{eq:kummerrelation3}) and (\ref{eq:kummerrelation4}) and using Lemma \ref{lemma:obtaining2} in place of Lemma \ref{lemma:obtaining}.\end{proof}

\begin{remark} Returning to the point of view of studying the hypergeometric equations as the $(2 \times 2)$ equations (\ref{eq:hg1}) and (\ref{eq:chg1}), our Main Theorem \ref{main:importance} may be equivalently stated as follows. If $K_{1}^{\pm}(\alpha)$ and $K_{\infty}^{\pm}(\alpha)$ are diagonal matrices satisfying (\ref{eq:sat5})-(\ref{eq:sat8}), then they satisfy the following:
\begin{align} &\lim_{\substack{\alpha \rightarrow \infty \\ \text{arg}(\alpha) = \pm \frac{\pi}{2}}} K_{\infty}^{\pm}(\alpha) \left(\begin{array}{cc} (-\alpha)^{\gamma-\beta}  & 0 \\ 0 & -(-\alpha)^{\beta}\end{array}\right) = I, \label{eq:this1} \\
&\lim_{\substack{\alpha \rightarrow \infty \\ \text{arg}(\alpha) = \pm \frac{\pi}{2}}} K_{1}^{\pm}(\alpha) \left(\begin{array}{cc} \alpha^{\gamma-\beta}  & 0 \\ 0 & -\alpha^{\beta} \end{array}\right) = I. \label{eq:this2} \end{align} 
These limits can be proved in an analogous way to the limits in our Main Theorem \ref{main:importance} by using Remarks \ref{remark:r1} and \ref{remark:r2} in place of Lemmas \ref{lemma:obtaining} and \ref{lemma:obtaining2} respectively. \end{remark}

\subsubsection{Obtaining $\widetilde{Y}^{(0)}(z)$ from $Y^{(0)}(z)$}\ \\
Since the substitution $x=\frac{z}{\alpha}$ and limit $\alpha \rightarrow \infty$ do not interfere with the nature of the Fuchsian singularity $x=0$, corresponding to $z=0$, this limit is much easier. We will only consider the limit along $\arg(\alpha) = -\frac{\pi}{2}$, the other case is completely analogous even though it requires to change the branch cut in $\widetilde\Omega_0$.

\begin{lemma} \label{lemma:first} We have the following limit,
\begin{align} \lim_{\alpha \rightarrow \infty} \ _{2}F_{1}\left( \begin{array}{c} \alpha , \ \beta \\ \gamma \end{array} ; \frac{z}{\alpha} \right) = \ _{1}F_{1} \left(\begin{array}{c} \beta \\ \gamma \end{array} ; z \right). \nonumber \end{align} \end{lemma}

\begin{proof} By taking the term by term limit in the series for  ${}_{2}F_{1}$ we obtain a uniformly convergent series that coincides with ${} _{1}F_{1}$. We conclude by uniqueness of Taylor series expansion for analytic functions. \end{proof}

\begin{theorem} \label{theorem:let} 
Let $y_{k}^{(0)}(x)$ and $\tilde{y}^{(0)}_{k}(z)$, $k=1,2$, be defined as in (\ref{eq:y0}) and (\ref{eq:yt0}) respectively. For $\arg(\alpha) = - \frac{\pi}{2}$, we have the following limits,
\begin{align} 
\begin{split} 
&\lim_{\substack{\alpha \rightarrow \infty \\ z \in \omega_{0}(\alpha)}} y_{1}^{(0)}\left(z \alpha^{-1} \right) \alpha^{1-\gamma} =
 \tilde{y}^{(0)}_{1}(z), \\ &\lim_{\substack{\alpha \rightarrow \infty \\ z \in \omega_{0}(\alpha)}} y_{2}^{(0)}\left(z\alpha^{-1}\right) = \tilde{y}^{(0)}_{2}(z), \end{split} 
 \quad \quad z \in \widetilde{\Omega}_{0}. \label{eq:020} 
 \end{align}
 where 
\[\omega_{0}(\alpha) = \left\{ z : |z| < |\alpha| , \ -\frac{3}{2}\pi  \leq \text{arg}(z) < \frac{\pi}{2} \right\}.\]
 \end{theorem}

\begin{proof}  Notice that  for $\arg(\alpha) =\frac{\pi}{2}$, $x \in \Omega_{0} \Leftrightarrow z \in \omega_{0}(\alpha)$. Since the radius of this neighbourhood clearly becomes infinite as $\alpha \rightarrow \infty$, if $z \in \omega_{0}(\alpha)$ for all $|\alpha|$ sufficiently large, then  the domain $\omega_{0}$ tends to the domain $\widetilde{\Omega}_{0}$ (i.e.  the domain in our definition of the fundamental solutions of Kummer equation around $z=0$ as given in Section \ref{sec:kumsol}). 

Using Lemma \ref{lemma:first}, we compute the limits as follows,
\begin{align} \lim_{\alpha \rightarrow \infty} y_{1}^{(0)}\left(z \alpha^{-1} \right) \alpha^{1-\gamma} &= \lim_{\alpha \rightarrow \infty} z^{1-\gamma} \ _{2}F_{1} \left(\begin{array}{c} \alpha+1-\gamma , \ \beta+1-\gamma \\ 2-\gamma \end{array} ; \frac{z}{\alpha}\right) \nonumber \\
&= z^{1-\gamma} \ _{1}F_{1} \left(\begin{array}{c} \beta+1-\gamma \\ 2-\gamma \end{array} ; z \right) = \tilde{y}_{1}^{(0)}(z), \quad \quad z \in \widetilde{\Omega}_{0},\nonumber \\
\text{and } \lim_{\alpha \rightarrow \infty} y_{2}^{(0)}\left(z \alpha^{-1} \right)&= \lim_{\alpha \rightarrow \infty} \ _{2}F_{1}\left(\begin{array}{c} \alpha , \ \beta \\ \gamma \end{array} ; \frac{z}{\alpha} \right) \nonumber \\
&= \ _{1}F_{1} \left(\begin{array}{c} \beta \\ \gamma \end{array} ; z \right) = \tilde{y}_{2}^{(0)}(z), \quad \quad z \in \widetilde{\Omega}_{0}, \nonumber \end{align} 
as required. \end{proof}

\begin{remark} The factor $\alpha^{1-\gamma}$ in the first limit of Theorem \ref{theorem:let} is necessary because of the term,
\[x^{1-\gamma} \equiv z^{1-\gamma} \alpha^{\gamma-1},\]
in the solution $y_{1}^{(0)}(x)$, as given in (\ref{eq:y0}). \end{remark}

\begin{remark} We have stated Theorem \ref{theorem:let} in terms of the solutions of the \textit{scalar} hypergeometric equatons (\ref{eq:gauss}) and (\ref{eq:kummer}). The limits (\ref{eq:020}) can be equivalently stated in terms of the solutions of the $(2\times2)$ equations (\ref{eq:hg1}) and (\ref{eq:chg1}): for $\arg(\alpha) = \pm \frac{\pi}{2}$, 
\begin{align} &\lim_{\substack{\alpha \rightarrow \infty \\ z \in \omega_{0}(\alpha)}} Y^{(0)}\left(\frac{z}{\alpha}\right) \alpha^{\Theta_{0}} = \widetilde{Y}^{(0)}(z), &&z\in \widetilde{\Omega}_{0}. \label{eq:y020} \end{align} 
To see how this is equivalent to (\ref{eq:020}), we observe that for the diagonalising matrices we have
\begin{align} 
\lim_{\alpha \rightarrow \infty} R_{0} = \lim_{\alpha \rightarrow \infty} \left(\begin{array}{cc} 1 & 1 \\ \frac{\alpha+1-\gamma}{\alpha(\beta-\gamma)} & \frac{1}{\beta-1} \end{array} \right) = \left(\begin{array}{cc} 1 & 1 \\ \frac{1}{\beta-\gamma} & \frac{1}{\beta-1} \end{array} \right) = \widetilde{R}_{0},\nonumber 
\end{align}
and for the series, using Lemma \ref{lemma:first}, 
\begin{align} 
&\lim_{\alpha \rightarrow \infty} G_{0}\left(z \alpha^{-1} \right) = \lim_{\alpha \rightarrow \infty} \left( \begin{matrix*}[l] \ _{2}F_{1} \left( \begin{array}{c} \alpha+1-\gamma,\ \beta-\gamma \\ 1-\gamma \end{array} ; \frac{z}{\alpha} \right) \text{\LARGE ,} \\ \frac{z(\alpha+1-\gamma)(1-\beta)}{\alpha(1-\gamma)(2-\gamma)} \ _{2}F_{1} \left(\begin{array}{c} \alpha+2-\gamma, \  \beta+1-\gamma \\ 3-\gamma \end{array} ; \frac{z}{\alpha}  \right) \text{\LARGE ,} \end{matrix*} \right. \nonumber \\
&\quad \quad \quad \quad \quad \quad \quad \quad \quad \quad \quad \quad \quad \quad \quad \quad \quad \quad \quad \quad \quad \left. \begin{matrix*}[r] \frac{z(\gamma-\beta)}{\gamma(\gamma-1)} \ \ _{2}F_{1} \left(\begin{array}{c} \alpha+1, \ \beta \\ \gamma+1 \end{array} ; \frac{z}{\alpha} \right) \\ \ _{2}F_{1} \left(\begin{array}{c}\alpha , \ \beta-1 \\ \gamma - 1 \end{array} ; \frac{z}{\alpha} \right) \end{matrix*} \right), \nonumber \\
&\quad = \left(\begin{array}{cc} _{1}F_{1} \left( \begin{array}{c} \beta-\gamma \\ 1-\gamma \end{array} ;z\right) \text{\LARGE ,} & \frac{z (\gamma-\beta)}{\gamma(\gamma-1)} \ \ _{1}F_{1} \left(\begin{array}{c} \beta \\ \gamma+1 \end{array} ; z\right) \\ & \\ \frac{z(1-\beta)}{(1-\gamma)(2-\gamma)} \ _{1}F_{1} \left(\begin{array}{c} \beta+1-\gamma \\ 3-\gamma \end{array} ; z \right) \text{\LARGE ,} & _{1}F_{1} \left(\begin{array}{c}\beta-1 \\ \gamma - 1 \end{array} ; z \right) \end{array} \right) = H_{0}(z). 
\nonumber \end{align} 
\end{remark}

\subsubsection{Limits of monodromy data} \label{sec:332}

Summarising the results so far, in section \ref{sec:331} we showed how \textit{term-by-term} limits of the solutions of Gauss equation around $x=\infty$ and $x=1$ produce the formal solutions of Kummer equation aroud $z=\infty$. We then explained how Glutsyuk's Theorem \ref{theorem:gl} asserts the existence of certain scalars which multiply Gauss solutions so that their true limits exist and are equal to the solutions of Kummer equation analytic in sectors at $z=\infty$. We have also proved our Main Theorem \ref{main:importance}, which establishes some important limits which these factors must satisfy. We now bring these results together to prove our second main theorem, concerned with explicitly producing the set of monodromy data $\widetilde{\mathcal{M}}$ from the set $\mathcal{M}$.
\begin{theorem} \label{main:top} 
Define the monodromy data of Gauss equation as given in (\ref{eq:ci0})-(\ref{eq:hgmono}) and of Kummer equation as in (\ref{eq:kummers-1})-(\ref{eq:chgmono}). We have the following limits, 
\begin{align} &\lim_{\substack{\alpha \rightarrow \infty \\ \text{arg}(\alpha) = \frac{\pi}{2}}} \left(\begin{array}{cc} \alpha^{\gamma-\beta} & 0 \\ 0 & -\alpha^{\beta}\end{array}\right) C^{1 \infty} \left(\begin{array}{cc} (-\alpha)^{\beta-\gamma} & 0 \\ 0 & -(-\alpha)^{-\beta}\end{array}\right) = \widetilde{S}_{0}, \label{eq:firstlimit} \\
&\lim_{\substack{\alpha \rightarrow \infty \\ \text{arg}(\alpha) = -\frac{\pi}{2}}} \left(\begin{array}{cc} \alpha^{\gamma-\beta} & 0 \\ 0 & -\alpha^{\beta}\end{array}\right) C^{1 \infty} \left(\begin{array}{cc} (-\alpha)^{\beta-\gamma} & 0 \\ 0 & -(-\alpha)^{-\beta}\end{array}\right) = \widetilde{S}_{-1}, \label{eq:secondlimit} \\
&\lim_{\substack{\alpha \rightarrow \infty \\ \text{arg}(\alpha) = -\frac{\pi}{2}}} \left(\begin{array}{cc} \alpha^{\gamma-1} & 0 \\ 0 & 1\end{array}\right) C^{0 \infty} \left(\begin{array}{cc} (-\alpha)^{\beta-\gamma} & 0 \\ 0 & -(-\alpha)^{-\beta}\end{array}\right) = \widetilde{C}^{0 \infty} \label{eq:thirdlimit}
 \end{align} 
Furthermore, as immediate consequences of the above limits of connection matrices, we have the following limits of monodromy matrices,
\begin{align} &\lim_{\substack{\alpha \rightarrow \infty \\ \text{arg}(\alpha) = -\frac{\pi}{2}}} \left(\begin{array}{cc} (-\alpha)^{\gamma-\beta} & 0 \\ 0 & -(-\alpha)^{\beta} \end{array}\right) M_{0} \left(\begin{array}{cc} (-\alpha)^{\beta-\gamma} & 0 \\ 0 & -(-\alpha)^{-\beta}\end{array}\right) = \widetilde{M}_{0}, \label{eq:fifthlimit} \\
&\lim_{\substack{\alpha \rightarrow \infty \\ \text{arg}(\alpha) = -\frac{\pi}{2}}} \left(\begin{array}{cc} (-\alpha)^{\gamma-\beta} & 0 \\ 0 & -(-\alpha)^{\beta} \end{array}\right) M_{\infty}M_{1} \left(\begin{array}{cc} (-\alpha)^{\beta-\gamma} & 0 \\ 0 & -(-\alpha)^{-\beta}\end{array}\right) = \widetilde{M}_{\infty}, \label{eq:seventhlimit} 
\end{align} \end{theorem}

\begin{proof} As part of the proof of this theorem, we will use the following elementary lemma.

\begin{lemma} \label{lemma:elementary} Let $f(\alpha)$ and $g(\alpha)$ be matrices such that $\lim_{\alpha \rightarrow \infty} f(\alpha)g(\alpha)$ exists. \\
\textbf{i)} If $\lim_{\alpha \rightarrow \infty}$ det$(f(\alpha))$ exists and is non-zero and det$(f(\alpha)) \neq 0$ for all $\alpha$ sufficiently large and if the limit $\lim_{\alpha \rightarrow \infty} f(\alpha)$ exists and is invertible, then the limit $\lim_{\alpha \rightarrow \infty} g(\alpha)$ exists. \\
\textbf{ii)} If $\lim_{\alpha \rightarrow \infty}$ det$(g(\alpha))$ exists and is non-zero and det$(g(\alpha)) \neq 0$ for all $\alpha$ sufficiently large and if the limit $\lim_{\alpha \rightarrow \infty} g(\alpha)$ exists, then the limit $\lim_{\alpha \rightarrow \infty} f(\alpha)$ exists. \end{lemma}

Let $\sigma_{\alpha}(\alpha)$ and $\sigma_{\infty}(\alpha)$ be the sectors defined in (\ref{eq:sectors1}) and (\ref{eq:sectors2}) respectively. As mentioned previously, if $z \in \sigma_{\alpha}(\alpha)$ then $x \in \Omega_{1}$ and if $z \in \sigma_{\infty}(\alpha)$ then $x \in \Omega_{\infty}$, so that the connection matrix $C^{1 \infty}$ remains valid for the solutions $Y^{(1)}(z \alpha^{-1})$ and $Y^{(\infty)}(z \alpha^{-1})$ restricted to the sectors $\sigma_{\alpha}(\alpha)$ and $\sigma_{\infty}(\alpha)$ respectively. Since the radii of these sectors do not diminish as $\alpha \rightarrow \infty$, for $|\alpha|$ sufficiently large we must have,
\[\sigma_{\alpha}(\alpha) \cap \sigma_{\infty}(\alpha) \neq \varnothing,\]
recall Figure \ref{fig:sss}. Therefore, for $|\alpha|$ sufficiently large, we have, 
\begin{align} Y^{(\infty)}\left(z \alpha^{-1} \right) = Y^{(1)}\left(z \alpha^{-1} \right) C^{1 \infty}, \quad \quad z \in \sigma_{\alpha}(\alpha) \cap \sigma_{\infty}(\alpha). \label{eq:bec} \end{align}
Let $\widetilde{\mathscr{S}}_{k}$ be the sectors defined in (\ref{eq:sectors3}). To prove the first limit (\ref{eq:firstlimit}), we first give a proof of Glutsyuk's Corollary \ref{corollary:glcorollary} in our case. We multiply by the matrices $K_{\infty}^{+}(\alpha)$ and $K_{1}^{+}(\alpha)$ and take the limit $\alpha \rightarrow \infty$, with $\arg(\alpha)=\frac{\pi}{2}$, so that (\ref{eq:bec}) becomes, 
\begin{align} &\lim_{\substack{\alpha \rightarrow \infty \\ \text{arg}(\alpha)=\frac{\pi}{2}}} \left. Y^{(\infty)}(z \alpha^{-1})\right|_{z \in \sigma_{\infty}(\alpha)}K_{\infty}^{+}(\alpha) \nonumber \\
&\quad = \lim_{\substack{\alpha \rightarrow \infty \\ \text{arg}(\alpha)=\frac{\pi}{2}}} \left. Y^{(1)}(z \alpha^{-1}) \right|_{z \in \sigma_{\alpha}(\alpha)}K_{1}^{+}(\alpha) \left(K_{1}^{+}(\alpha)\right)^{-1}C^{1\infty}K_{\infty}^{+}(\alpha), \label{eq:bec2} \end{align}
for $z \in \widetilde{\mathscr{S}}_{0}\cap \widetilde{\mathscr{S}}_{1}$. We apply Lemma \ref{lemma:elementary} \textbf{i)} with, 
\[f(\alpha) = \left. Y^{(1)}(z \alpha^{-1})\right|_{z \in \sigma_{\alpha}(\alpha)}K_{1}^{+}(\alpha) \quad \text{and} \quad g(\alpha)=\left(K_{1}^{+}(\alpha)\right)^{-1}C^{1\infty}K_{\infty}^{+}(\alpha).\]
Observe that the hypotheses of Lemma \ref{lemma:elementary} hold: the limit, 
\[\lim_{\substack{\alpha \rightarrow \infty \\ \text{arg}(\alpha) = \frac{\pi}{2}}} f(\alpha)g(\alpha),\]
exists and equals $\widetilde{Y}^{(\infty,1)}(z)$, by (\ref{eq:sat8}), and the limit, 
\[\lim_{\substack{\alpha \rightarrow \infty \\ \text{arg}(\alpha)=\frac{\pi}{2}}} f(\alpha),\]
exists and equals $\widetilde{Y}^{(\infty,0)}(z)$, by (\ref{eq:sat7}), which is clearly invertible because it is a fundamental solution. For all $\alpha$, $f(\alpha)$ is also clearly invertible because it is a fundamental solution. The limit,
\[\lim_{\substack{\alpha \rightarrow \infty \\ \text{arg}(\alpha)=\frac{\pi}{2}}} g(\alpha) = \lim_{\substack{\alpha \rightarrow \infty \\ \text{arg}(\alpha)=\frac{\pi}{2}}} \left(K_{1}^{+}(\alpha)\right)^{-1}C^{1\infty}K_{\infty}^{+}(\alpha),\]
therefore exists and, from (\ref{eq:bec2}), 
\begin{align} \widetilde{Y}^{(\infty,1)}(z) = \widetilde{Y}^{(\infty,0)}(z) \lim_{\substack{\alpha \rightarrow \infty \\ \text{arg}(\alpha)=\frac{\pi}{2}}} \left(K_{1}^{+}(\alpha)\right)^{-1}C^{1 \infty}K_{\infty}^{+}(\alpha), \quad \quad z \in \widetilde{\mathscr{S}_{0}} \cap \widetilde{\mathscr{S}}_{1}.\nonumber \end{align}
Recall that if $z \in \widetilde{\mathscr{S}}_{k}$ then $z \in \widetilde{\Sigma}_{k}$ and recall Definition \ref{definition:kumst} of Stokes matrices, namely we have, 
\[\widetilde{Y}^{(\infty,1)}(z) = \widetilde{Y}^{(\infty,0)}(z) \widetilde{S}_{0}, \quad \quad z \in \widetilde{\Sigma}_{0}\cap\widetilde{\Sigma}_{1}.\]
We conclude that,
\[\lim_{\substack{\alpha \rightarrow \infty \\ \text{arg}(\alpha)=\frac{\pi}{2}}} \left(K_{1}^{+}(\alpha)\right)^{-1}C^{1\infty}K_{\infty}^{+}(\alpha) = \widetilde{S}_{0},\]
which is precisely Glutsyuk's Corollary \ref{corollary:glcorollary} in our case. Combining this with (\ref{eq:this1}) and (\ref{eq:this2}), we compute,
\begin{align} \widetilde{S}_{0} &= \lim_{\substack{\alpha \rightarrow \infty \\ \text{arg}(\alpha)=\frac{\pi}{2}}} \left(K_{1}^{+}(\alpha)\right)^{-1} C^{1 \infty} K_{\infty}^{+}(\alpha), \nonumber \\
&= \lim_{\substack{\alpha \rightarrow \infty \\ \text{arg}(\alpha)=\frac{\pi}{2}}} \left(K_{1}^{+}(\alpha) \left(\begin{array}{cc} \alpha^{\gamma-\beta} & 0 \\ 0 & -\alpha^{\beta} \end{array} \right)\left(\begin{array}{cc} \alpha^{\beta-\gamma} & 0 \\ 0 & -\alpha^{-\beta} \end{array} \right) \right)^{-1} \nonumber \\
&\quad \quad \quad \quad \quad \quad \quad \quad C^{1 \infty} K_{\infty}^{+}(\alpha) \left(\begin{array}{cc} (-\alpha)^{\gamma-\beta} & 0 \\ 0 & -(-\alpha)^{\beta} \end{array} \right)\left(\begin{array}{cc} (-\alpha)^{\beta-\gamma} & 0 \\ 0 & -(-\alpha)^{-\beta} \end{array}  \right), \nonumber \\
&= \lim_{\substack{\alpha \rightarrow \infty \\ \text{arg}(\alpha)=\frac{\pi}{2}}}\left(\begin{array}{cc} \alpha^{\gamma-\beta} & 0 \\ 0 & -\alpha^{\beta} \end{array} \right) C^{1 \infty} \left(\begin{array}{cc} (-\alpha)^{\beta-\gamma} & 0 \\ 0 & -(-\alpha)^{-\beta} \end{array} \right), \nonumber \end{align}
where we have implicitly used Lemma \ref{lemma:elementary} again, this proves the first limit (\ref{eq:firstlimit}) of the theorem. To prove the second limit (\ref{eq:secondlimit}), we multiply by the matrices $K_{\infty}^{-}(\alpha)$ and $K_{1}^{-}(\alpha)$ and take the limit $\alpha \rightarrow \infty$, with $\arg(\alpha)=-\frac{\pi}{2}$, so that (\ref{eq:bec}) becomes, 
\begin{align} &\lim_{\substack{\alpha \rightarrow \infty \\ \text{arg}(\alpha)=-\frac{\pi}{2}}} \left. Y^{(\infty)}(z \alpha^{-1})\right|_{z \in \sigma_{\infty}(\alpha)}K_{\infty}^{-}(\alpha) \nonumber \\
&\quad = \lim_{\substack{\alpha \rightarrow \infty \\ \text{arg}(\alpha)=-\frac{\pi}{2}}} \left. Y^{(1)}(z \alpha^{-1}) \right|_{z \in \sigma_{\alpha}(\alpha)}K_{1}^{-}(\alpha) \left(K_{1}^{-}(\alpha)\right)^{-1}C^{1\infty}K_{\infty}^{-}(\alpha), \label{eq:bec3} \end{align}
for $z \in \widetilde{\mathscr{S}}_{-1}\cap\widetilde{\mathscr{S}}_{0}$. By following a similar procedure as above, using Lemma \ref{lemma:elementary} and the relations (\ref{eq:sat5}) and (\ref{eq:sat6}), we deduce,
\[\lim_{\substack{\alpha \rightarrow \infty \\ \text{arg}(\alpha)=-\frac{\pi}{2}}} \left(K_{1}^{-}(\alpha)\right)^{-1}C^{1\infty}K_{\infty}^{-}(\alpha) = \widetilde{S}_{-1}.\]
Combining this with (\ref{eq:this1}) and (\ref{eq:this2}), we compute,
\begin{align} \widetilde{S}_{-1} &= \lim_{\substack{\alpha \rightarrow \infty \\ \text{arg}(\alpha)=-\frac{\pi}{2}}} \left(K_{1}^{-}(\alpha)\right)^{-1} C^{1 \infty} K_{\infty}^{-}(\alpha), \nonumber \\
&= \lim_{\substack{\alpha \rightarrow \infty \\ \text{arg}(\alpha)=-\frac{\pi}{2}}} \left(K_{1}^{-}(\alpha) \left(\begin{array}{cc} \alpha^{\gamma-\beta} & 0 \\ 0 & -\alpha^{\beta} \end{array} \right)\left(\begin{array}{cc} \alpha^{\beta-\gamma} & 0 \\ 0 & -\alpha^{-\beta} \end{array} \right) \right)^{-1} \nonumber \\
&\quad \quad \quad \quad \quad \quad \quad \quad C^{1 \infty} K_{\infty}^{-}(\alpha) \left(\begin{array}{cc} (-\alpha)^{\gamma-\beta} & 0 \\ 0 & -(-\alpha)^{\beta} \end{array} \right)\left(\begin{array}{cc} (-\alpha)^{\beta-\gamma} & 0 \\ 0 & -(-\alpha)^{-\beta} \end{array} \right), \nonumber \\
&= \lim_{\substack{\alpha \rightarrow \infty \\ \text{arg}(\alpha)=-\frac{\pi}{2}}}\left(\begin{array}{cc} \alpha^{\gamma-\beta} & 0 \\ 0 & -\alpha^{\beta} \end{array} \right) C^{1 \infty} \left(\begin{array}{cc} (-\alpha)^{\beta-\gamma} & 0 \\ 0 & -(-\alpha)^{-\beta} \end{array} \right), \nonumber \end{align}
where we have implicitly used Lemma \ref{lemma:elementary}, this proves the second limit (\ref{eq:secondlimit}) of the theorem. \\

To prove the third limit (\ref{eq:thirdlimit}) we first note that the curve $\gamma_{\infty 0}$ which defines the connection matrix $C^{0 \infty}$ survives the confluence limit. In other words, after the substitution $x=\frac{z}{\alpha}$, the curve does not diminish or become broken under the limit $\alpha \rightarrow \infty$. This fact is expressed as follows,
\begin{align} &\lim_{\substack{\alpha \rightarrow \infty \\ \text{arg}(\alpha) = -\frac{\pi}{2}}} \gamma_{\infty 0} \left[Y^{(\infty)}K_{\infty}^{-}(\alpha)\right]\left(z \alpha^{-1}\right) = \gamma_{\infty 0} \left[\widetilde{Y}^{(\infty,0)}\right](z), \nonumber \end{align}
or equivalently, using the domains $\omega_{0}^{-}(\alpha)$ and $\widetilde{\Omega}_{0}^{-}$ defined in Sections \ref{sec:331} and \ref{sec:kumsol} respectively,
\begin{align} &\lim_{\substack{\alpha \rightarrow \infty \\ \text{arg}(\alpha) = -\frac{\pi}{2}}} \left. Y^{(0)}\left(z \alpha^{-1}\right) \right|_{z \in \omega_{0}^{-}(\alpha)} C^{0 \infty} \left(C^{1\infty}\right)^{-1} K_{\infty}^{-}(\alpha) = \widetilde{Y}^{(0)}(z) \widetilde{C}^{0 \infty}, \quad \quad z \in \widetilde{\Omega}_{0}^{-}.\nonumber \end{align}

Combining this with the limits (\ref{eq:y020}) and (\ref{eq:this1}), we deduce the required result (\ref{eq:thirdlimit}) as follows, for $z \in \widetilde{\Omega}_{0}^{-}$:
\begin{align} &\lim_{\substack{\alpha \rightarrow \infty \\ \text{arg}(\alpha)=-\frac{\pi}{2}}} \left. Y^{(0)}\left(z \alpha^{-1}\right) \right|_{z \in \omega_{0}^{-}(\alpha)} \left(\begin{array}{cc} \alpha^{1-\gamma} & 0 \\ 0 & 1 \end{array} \right) \left(\begin{array}{cc} \alpha^{\gamma-1} & 0 \\ 0 & 1 \end{array}\right) C^{0 \infty} \nonumber \\
&\quad \hspace{85pt} K_{\infty}^{-}(\alpha) \left(\begin{array}{cc} (-\alpha)^{\gamma-\beta} & 0 \\ 0 & -(-\alpha)^{\beta} \end{array} \right)\left(\begin{array}{cc} (-\alpha)^{\beta-\gamma} & 0 \\ 0 & -(-\alpha)^{-\beta} \end{array}  \right) \nonumber \\
&\quad \quad = \widetilde{Y}^{(0)}(z) \lim_{\substack{\alpha \rightarrow \infty \\ \text{arg}(\alpha)=
-\frac{\pi}{2}}} \left(\begin{array}{cc}\alpha^{\gamma-1} & 0 \\ 0 & 1 \end{array} \right) C^{0 \infty}\left(\begin{array}{cc} (-\alpha)^{\beta-\gamma} & 0 \\ 0 & -(-\alpha)^{-\beta} \end{array}  \right) = \widetilde{Y}^{(0)}(z) \widetilde{C}^{0 \infty}, \nonumber \\ 
&\hspace{40pt} \Leftrightarrow \quad \lim_{\substack{\alpha \rightarrow \infty \\ \text{arg}(\alpha)=-\frac{\pi}{2}}} \left(\begin{array}{cc}\alpha^{\gamma-1} & 0 \\ 0 & 1 \end{array} \right) C^{0 \infty}\left(\begin{array}{cc} (-\alpha)^{\beta-\gamma} & 0 \\ 0 & -(-\alpha)^{-\beta} \end{array}  \right) = \widetilde{C}^{0\infty}, \nonumber \end{align}
where we have implicitly used Lemma \ref{lemma:elementary}. \\

Having deduced the limit (\ref{eq:thirdlimit}) of the connection matrix, the limit (\ref{eq:fifthlimit}) follow directly since $M_{0} = \left(C^{0\infty}\right)^{-1} e^{2 \pi i \Theta_{0}}C^{0\infty}$ and $\Theta_{0} \equiv \widetilde{\Theta}_{0}$. For (\ref{eq:fifthlimit}), we have,
\begin{align} &\lim_{\substack{\alpha \rightarrow \infty \\ \text{arg}(\alpha) = -\frac{\pi}{2}}} \left(\begin{array}{cc}(-\alpha)^{\gamma-\beta} & 0 \\ 0 & -(-\alpha)^{\beta} \end{array} \right) M_{0} \left(\begin{array}{cc}(-\alpha)^{\beta-\gamma} & 0 \\ 0 & -(-\alpha)^{-\beta}\end{array}\right) \nonumber \\
&\quad \quad = \lim_{\substack{\alpha \rightarrow \infty \\ \text{arg}(\alpha) = -\frac{\pi}{2}}} \left(\begin{array}{cc}(-\alpha)^{\gamma-\beta} & 0 \\ 0 & -(-\alpha)^{\beta} \end{array} \right) \left(C^{0\infty}\right)^{-1} e^{2 \pi i \Theta_{0}}C^{0\infty} \left(\begin{array}{cc}(-\alpha)^{\beta-\gamma} & 0 \\ 0 & -(-\alpha)^{-\beta}\end{array}\right), \nonumber \\
&\quad \quad = \lim_{\substack{\alpha \rightarrow \infty \\ \text{arg}(\alpha) = -\frac{\pi}{2}}} \left(\begin{array}{cc}(-\alpha)^{\gamma-\beta} & 0 \\ 0 & -(-\alpha)^{\beta} \end{array} \right) \left(C^{0 \infty}\right)^{-1} \left(\begin{array}{cc}\alpha^{\gamma-1} & 0 \\ 0 & 1 \end{array}\right) e^{2 \pi i \Theta_{0}} \nonumber \\
&\quad \hspace{60pt} \left(\begin{array}{cc} \alpha^{1-\gamma} & 0 \\ 0 & 1 \end{array} \right) C^{0\infty} \left(\begin{array}{cc}(-\alpha)^{\beta-\gamma} & 0 \\ 0 & -(-\alpha)^{-\beta}\end{array}\right), \nonumber \\
&\quad \quad = \left(\widetilde{C}^{0\infty}\right)^{-1} e^{2 \pi i \widetilde{\Theta}_{0}} \widetilde{C}^{0 \infty} = \widetilde{M}_{0}, \nonumber \end{align}
as required. \end{proof}

\subsubsection{Explicit computations of limits of monodromy data} 
Here we apply Theorem \ref{main:top} to calculate explicitly  the Stokes' matrices. We will use the following classical facts: 
\begin{align} &\lim_{\alpha \rightarrow \infty} a^{c-b} \frac{\Gamma(a+b)}{\Gamma(a+c)} = 1, \text{ as $a \rightarrow \infty$, $|\text{arg}(a)|<\pi$},\label{eq:f1} \\
&\Gamma(a) \equiv \frac{\pi}{\sin(\pi a)\Gamma(1-a)}, \label{eq:f2} \\
&\lim_{a \rightarrow \infty} e^{i \pi a}\csc(\pi a) = 2i \text{ for Im}(a)<0. \label{eq:f3} \end{align}
The proof of (\ref{eq:f3}) is elementary, the proofs of (\ref{eq:f1}) and (\ref{eq:f2}) can be found in \cite{ww} and \cite{bateman}.\\ 

Let $C^{1 \infty}$ be given by (\ref{eq:ci1}). Using $(-\alpha) \equiv \alpha e^{i \pi}$, we calculate, 
\begin{align} &\left(\begin{array}{cc} \alpha^{\gamma-\beta} & 0 \\ 0 & -\alpha^{-\beta}\end{array}\right) C^{1 \infty} \left(\begin{array}{cc} (-\alpha)^{\beta-\gamma} & 0 \\ 0 & -(-\alpha)^{\beta}\end{array}\right) \nonumber \\
&\quad = \left(\begin{array}{cc} \alpha^{\gamma-\beta} & 0 \\ 0 & -\alpha^{-\beta}\end{array}\right) \nonumber \\
&\quad \quad \left(\begin{array}{cc} e^{i \pi (\gamma-\beta)} \frac{\Gamma(\alpha+1-\beta)\Gamma(\alpha+\beta-\gamma)}{\Gamma(\alpha)\Gamma(\alpha+1-\gamma)} & e^{i \pi (\gamma-\alpha)} \frac{\Gamma(\beta+1-\alpha) \Gamma(\alpha+\beta-\gamma)}{\Gamma(\beta)\Gamma(\beta+1-\gamma)} \\ e^{i \pi \alpha} \frac{\Gamma(\alpha+1-\beta)\Gamma(\gamma-\alpha-\beta)}{\Gamma(1-\beta)\Gamma(\gamma-\beta)} & e^{i \pi \beta} \frac{\Gamma(\beta+1-\alpha)\Gamma(\gamma-\alpha-\beta)}{\Gamma(1-\alpha)\Gamma(\gamma-\alpha)} \end{array} \right) \left(\begin{array}{cc} (-\alpha)^{\beta-\gamma} & 0 \\ 0 & -(-\alpha)^{\beta}\end{array}\right), \nonumber \\
&\quad = \left(\begin{array}{cc} \frac{\Gamma(\alpha+1-\beta)\Gamma(\alpha+\beta-\gamma)}{\Gamma(\alpha)\Gamma(\alpha+1-\gamma)} & -e^{\pi i (\gamma-\alpha-\beta)}\alpha^{\gamma-2\beta} \frac{\Gamma(\beta+1-\alpha)\Gamma(\alpha+\beta-\gamma)}{\Gamma(\beta)\Gamma(\beta+1-\gamma)} \\ -e^{\pi i(\alpha+\beta-\gamma)}\alpha^{2\beta-\gamma}\frac{\Gamma(\alpha+1-\beta)\Gamma(\gamma-\alpha-\beta)}{\Gamma(1-\beta)\Gamma(\gamma-\beta)} & \frac{\Gamma(\beta+1-\alpha)\Gamma(\gamma-\alpha-\beta)}{\Gamma(1-\alpha)\Gamma(\gamma-\alpha)} \end{array} \right). \nonumber \end{align} 

Using (\ref{eq:f1}), we find for the (1,1) and (2,2) elements: 
\begin{align} &\lim_{\substack{\alpha \rightarrow \infty \\ \text{arg}(\alpha) = \pm \frac{\pi}{2}}} \frac{\Gamma(\alpha+1-\beta)\Gamma(\alpha+\beta-\gamma)}{\Gamma(\alpha)\Gamma(\alpha+1-\gamma)} = 1, \nonumber \\
\text{and } &\lim_{\substack{\alpha \rightarrow \infty \\ \text{arg}(\alpha) = \pm \frac{\pi}{2}}} \frac{\Gamma(\beta+1-\alpha)\Gamma(\gamma-\alpha-\beta)}{\Gamma(1-\alpha)\Gamma(\gamma-\alpha)} = 1, \nonumber \end{align}
respectively, as required. We rewrite the (1,2) and (2,1) elements using (\ref{eq:f2}) as follows:
\begin{align} -e^{\pi i (\gamma-\alpha-\beta)}\alpha^{\gamma-2\beta} &\frac{\Gamma(\beta+1-\alpha)\Gamma(\alpha+\beta-\gamma)}{\Gamma(\beta)\Gamma(\beta+1-\gamma)} \nonumber \\
&= \frac{-e^{i \pi (\gamma-\alpha-\beta)}}{\sin(\pi(\alpha+\beta-\gamma))} \alpha^{\gamma-2\beta} \frac{\Gamma(\beta+1-\alpha)}{\Gamma(\gamma+1-\alpha-\beta)} \frac{\pi}{\Gamma(\beta)\Gamma(\beta+1-\gamma)}, \nonumber \end{align}
and,
\begin{align} -e^{i \pi (\alpha+\beta-\gamma)}\alpha^{2\beta-\gamma} &\frac{\Gamma(\alpha+1-\beta)\Gamma(\gamma-\alpha-\beta)}{\Gamma(1-\beta)\Gamma(\gamma-\beta)} \nonumber \\
&= \frac{-e^{i \pi (\alpha+\beta-\gamma)}}{\sin(\pi(\gamma-\alpha-\beta))} \alpha^{2\beta-\gamma} \frac{\Gamma(\alpha+1-\beta)}{\Gamma(\alpha+\beta+1-\gamma)} \frac{\pi}{\Gamma(1-\beta)\Gamma(\gamma-\beta)},\nonumber \end{align}
respectively. As $\alpha \rightarrow \infty$, the dominant terms in these expressions are $e^{\mp i \pi \alpha}$ respectively; observe that, if arg$(\alpha) = \pm \frac{\pi}{2}$ then $e^{\pm i \pi \alpha} \rightarrow 0$ as $\alpha \rightarrow \infty$, as required. Finally, for the most important computations, we have:
\begin{align} \lim_{\substack{\alpha \rightarrow \infty \\ \text{arg}(\alpha) = - \frac{\pi}{2}}} &\underbrace{\frac{-e^{i \pi (\alpha+\beta-\gamma)}}{\sin(\pi(\gamma-\alpha-\beta))}}_{\text{\normalsize $\rightarrow 2i \text{ by (\ref{eq:f3})}$}} \underbrace{\alpha^{2\beta-\gamma} \frac{\Gamma(\alpha+1-\beta)}{\Gamma(\alpha+\beta+1-\gamma)}}_{\text{\normalsize $\rightarrow 1 \text{ by (\ref{eq:f1})}$}} \frac{\pi}{\Gamma(1-\beta)\Gamma(\gamma-\beta)}, \nonumber \\
&= \frac{2 \pi i}{\Gamma(1-\beta)\Gamma(\gamma-\beta)} \equiv \left(S_{-1}\right)_{2,1}, \nonumber \end{align}
and,
\begin{align} \lim_{\substack{\alpha \rightarrow \infty \\ \text{arg}(\alpha) = \frac{\pi}{2}}} &\underbrace{\frac{-e^{i \pi (\gamma-\alpha-\beta)}}{\sin(\pi(\alpha+\beta-\gamma))}}_{\text{\normalsize $\rightarrow 2i \text{ by (\ref{eq:f3})}$}} \underbrace{\alpha^{\gamma-2\beta} \frac{\Gamma(\beta+1-\alpha)}{\Gamma(\gamma+1-\alpha-\beta)}}_{\text{\normalsize $\rightarrow e^{i \pi (\gamma-2\beta)} \text{ by (\ref{eq:f1})}$}} \frac{\pi}{\Gamma(\beta)\Gamma(\beta+1-\gamma)}, \nonumber \\
&= \frac{2 \pi i e^{i \pi (\gamma-2\beta)}}{\Gamma(\beta)\Gamma(\beta+1-\gamma)} \equiv \left(S_{0}\right)_{1,2}, \nonumber \end{align}
as required by formulae \eqref{eq:kummers-1}.


\appendix
\section{Appendix A: Gauss monodromy data and  Mellin-Barnes integral}

Here, following  \cite{bateman}, \cite{ww} and \cite{andrews}, we re-derive the classical formulae (\ref{eq:ci0})-(\ref{eq:c10}). This is a worthwhile exercise as it gives a greater understanding of how to analytically continue solutions and compute their monodromy data.

We will work with the following Mellin-Barnes integral, 
\begin{align} \frac{1}{2 \pi i} \int^{+i \infty}_{-i \infty} I(s,x) \ ds \quad \text{where} \quad I(s,x) = \frac{\Gamma(\alpha+s)\Gamma(\beta+s)\Gamma(-s)}{\Gamma(c+s)}(-x)^{s}, \label{eq:is} \end{align}
with $|\text{arg}(-x)|<\pi$ and whose path of integration is along the imaginary axis with indentations as necessary so that the poles of $\Gamma(\alpha+s)\Gamma(\beta+s)$ lie on its left and the poles of $\Gamma(-s)$ lie on its right, as shown in Figure \ref{fig:poles} below. It is always possible to construct such a path as long as $\alpha$ and $\beta \notin \mathbb{Z}^{\leq 0}$, which is a general assumption since the case in which $\alpha$ or $\beta \in \mathbb{Z}^{\leq 0}$ corresponds to some of the solutions in (\ref{eq:y0})-(\ref{eq:yinf}) being polynomials. 
\begin{figure}[H] \begin{center}
\includegraphics[scale=.8]{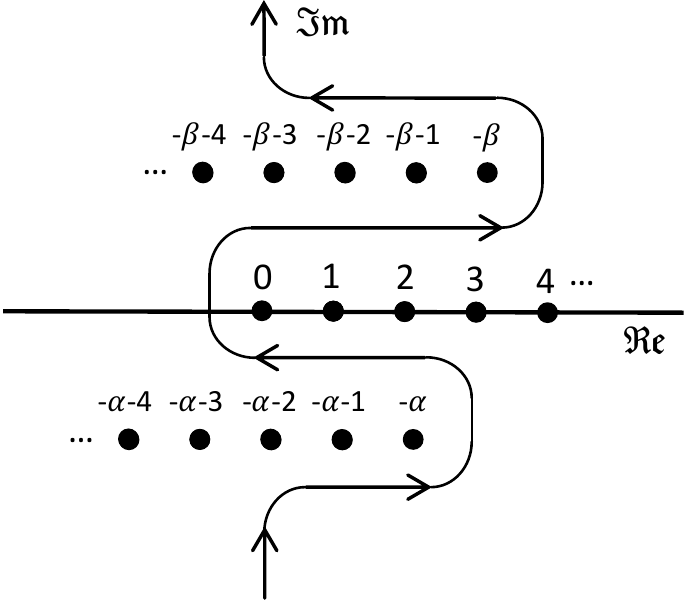}
\caption{\label{fig:poles}Path of integration with indentations as in (\ref{eq:is}).}
\end{center} \end{figure} 
We will prove the following proposition, which is sufficient to derive the connection formulae (\ref{eq:ci0})-(\ref{eq:c10}). 

\begin{proposition} \label{proposition:p} The integral as given by (\ref{eq:is}) satisfies the following properties:
\begin{enumerate} 
\item for $|\text{arg}(-x)| < \pi$,
\[\frac{1}{2 \pi i} \int^{+i \infty}_{-i \infty} I(s,x) \ ds,\]
defines an analytic function of $x$; 
\item for $|\text{arg}(-x)| < \pi$ and $|x| < 1$,
\begin{align} \frac{1}{2 \pi i} \int^{+i \infty}_{-i \infty} I(s,x) \ ds = \frac{\Gamma(\alpha)\Gamma(\beta)}{\Gamma(\gamma)} \ y_{2}^{(0)}(x), \nonumber \end{align} 
where $y_{2}^{(0)}(x)$ is the solution of Gauss equation as given by (\ref{eq:y0}). 
\item for $|\text{arg}(-x)|<\pi$ and $|x|>1$, 
\begin{align} \frac{1}{2 \pi i} \int^{+i \infty}_{-i \infty} I(s,x) \ ds &= \frac{\Gamma(\alpha)\Gamma(\beta-\alpha)}{\Gamma(\gamma-\alpha)} y_{1}^{(\infty)}(x) + \frac{\Gamma(\beta)\Gamma(\alpha-\beta)}{\Gamma(\gamma-\beta)} y_{2}^{(\infty)}(x), \nonumber \end{align}
where $y_{1}^{(\infty)}(x)$ and $y_{2}^{(\infty)}(x)$ are the solutions of Gauss equation as given by (\ref{eq:yinf}). \end{enumerate} \end{proposition}

\begin{proof} This proof is organised into three parts to prove each statement consecutively.  

We start by proving the analyticity of the integral. We use Euler's reflection formula $\Gamma(-s)\Gamma(s+1) = - \pi \csc(\pi s)$ to re-write the integrand,
\begin{align} I(s,x) = -\frac{\Gamma(\alpha+s)\Gamma(\beta+s)}{\Gamma(c+s)\Gamma(s+1)}\frac{\pi}{\sin(\pi s)}(-x)^{s}. \label{eq:euler} \end{align}
Using the following asymptotic expansion of the Gamma function \cite{ww} \textsection 13.6,
\begin{align} \Gamma(s+a) = s^{s+a-\frac{1}{2}}e^{-s}\sqrt{2\pi} (1+ o(1)), \quad \text{with $|s|$ large,} \label{eq:gammaasy} \end{align}
which is valid for $|\text{arg}(s+a)| < \pi$, we deduce, 
\begin{align}\frac{\Gamma(\alpha+s)\Gamma(\beta+s)}{\Gamma(c+s)\Gamma(s+1)} = \mathcal{O}\left(|s|^{\alpha+\beta-\gamma-1}\right), \quad \text{as $|s| \rightarrow \infty$}. \label{eq:ded} \end{align}
Writing $\sin(\pi s) = \frac{1}{2i}(e^{i\pi s} - e^{-i\pi s})$ we also deduce, 
\begin{align} \sin(\pi s) = \mathcal{O}\left(e^{|s| \pi} \right), \quad \text{as $|s| \rightarrow \infty$}, \label{eq:ded2} \end{align}
along the contour of integration (the imaginary axis). Combining (\ref{eq:ded}) and (\ref{eq:ded2}), the integrand has the following asymptotic behavior, 
\[I(s,x) = \mathcal{O}\left(|s|^{\alpha+\beta-\gamma-1}e^{- |s| \pi} (-x)^{s} \right), \quad \text{as $|s| \rightarrow \infty$},\]
along the contour of integration, we therefore need only consider the analyticity of the following integral, 
\begin{align} &\int^{+i \infty}_{-i\infty} e^{-|s| \pi}(-x)^{s} \ ds \nonumber \\ 
&\quad \equiv i \int^{\infty}_{0}e^{-\sigma \pi} e^{i \sigma(\log |x| + i \text{arg}(-x))} \ d \sigma - i \int^{\infty}_{0} e^{- \sigma \pi} e^{- i \sigma (\log|x| + i \text{arg}(-x))} \ d \sigma. \label{eq:int} \end{align}
We recall the following lemma, see for instance \cite{ww} \textsection 5.32,

\begin{lemma}\label{lemma:ww} If $f: \mathbb{R} \rightarrow \mathbb{R}$ is a continuous function such that $|f(t)| \leq Ke^{rt}$ for constants $K$ and $r$, then the integral $\int_{0}^{\infty} f(t) e^{-\lambda t} \ dt$ defines an analytic function of $\lambda$ for $r<$ Re$(\lambda)$. \end{lemma}

Applying this lemma to the first integral in (\ref{eq:int}), with $r=-\pi$, $K=1$ and $\lambda = \text{arg}(-x)$, we find an analytic function for $- \pi < \text{arg}(-x)$. Applying this lemma to the second integral in (\ref{eq:int}), with $r=-\pi$, $K=1$ and $\lambda = -\text{arg}(-x)$, we find an analytic function for $\text{arg}(-x) < \pi$. This concludes the proof that the integral (\ref{eq:is}) defines an analytic function for $-\pi<\text{arg}(-x)<\pi$. 

We now represent $y_{2}^{(0)}(x)$ using a Mellin-Barnes integral.
We write $I(s,x)$ as in (\ref{eq:euler}) and consider the following integral, 
\[\frac{1}{2 \pi i} \int_{C_{N}} I(s,x) \ ds,\]
for $N \in \mathbb{N}^{\geq 0}$, where $C_{N}$ is the following semicircle, 
\[C_{N} = \left\{s = \left(N+\frac{1}{2}\right)e^{i \theta} : \theta \in \left[-\frac{\pi}{2} , \frac{\pi}{2} \right] \right\}.\]
Let $s \in C_{N}$, using formula (\ref{eq:gammaasy}) from above, we deduce the following asymptotic behavior, 
\begin{align} \frac{\Gamma(\alpha+s)\Gamma(\beta+s)}{\Gamma(\gamma+s)\Gamma(s+1)} = \mathcal{O}\left(N^{\alpha+\beta-\gamma-1}\right), \quad \text{as $N \rightarrow \infty$}, \label{eq:ded3} \end{align}
and, using $\sin(\pi s) = \frac{1}{2i}(e^{i \pi s} - e^{-i \pi s})$, 
\begin{align} \frac{(-x)^{s}}{\sin(\pi s)} = \mathcal{O}\left( \text{\large $e^{\left(N+\frac{1}{2}\right) \left(\cos(\theta) \log|x| - \sin(\theta) \text{arg}(-x) - \pi |\sin(\theta)|\right)}$}\right), \quad \text{as $N \rightarrow \infty$}. \label{eq:ded4} \end{align}
Since $|\text{arg}(-x)|<\pi$, we write $|\text{arg}(-x)| \leq \pi-\delta$ for some $\delta >0$, so that,
\begin{align} \pm \ \text{arg}(-x) + \pi \geq \delta \quad \quad &\Leftrightarrow \quad \quad \sin(\theta) \text{arg}(-x) + |\sin(\theta)| \pi \geq |\sin(\theta)| \delta, \nonumber \\
&\Leftrightarrow \quad \quad \text{\large $e^{-\sin(\theta) \text{arg}(-x) - \pi |\sin(\theta)|}$} \leq \text{\large $e^{-|\sin(\theta)| \delta}$}. \label{eq:ded5} \end{align}
Combining (\ref{eq:ded3})-(\ref{eq:ded5}), the integrand has the following asymptotic behaviour for $s \in C_{N}$, 
\[I(s,x) = \mathcal{O}\left(N^{\alpha+\beta-\gamma-1} \text{\large $e^{\left(N+\frac{1}{2}\right) \left(\cos(\theta) \log|x| - |\sin(\theta)|\delta\right)}$}\right), \quad \text{as $N \rightarrow \infty$}. \]
Since $\cos(\theta)$ and $|\sin(\theta)|$ are even functions, we need only consider $\theta \in \left[ 0,\frac{\pi}{2}\right]$. For $\theta \in \left[0,\frac{\pi}{4}\right]$, $\cos(\theta) \geq \frac{1}{\sqrt{2}}$ and for $\theta \in \left[\frac{\pi}{4} , \frac{\pi}{2} \right]$, $\sin(\theta) \geq \frac{1}{\sqrt{2}}$. Henceforth, we impose the condition that $|x|<1$, or equivalently $\log|x|<0$. For $s \in C_{N}$ we deduce: 
\begin{align}I(s,x) = \left\{ \begin{matrix*}[l] 
\mathcal{O}\left(N^{\alpha+\beta-\gamma-1} \text{\large $e^{\left(N+\frac{1}{2}\right) \frac{1}{\sqrt{2}} \log|x|}$}\right), & \theta \in \left[0,\frac{\pi}{4}\right), \\ 
\mathcal{O}\left(N^{\alpha+\beta-\gamma-1} \text{\large $e^{\left(N+\frac{1}{2}\right)\frac{1}{\sqrt{2}} \left(\log|x| - \delta \right)}$}\right), & \theta = \frac{\pi}{4}, \\
\mathcal{O}\left(N^{\alpha+\beta-\gamma-1} \text{\large $e^{-\left(N+\frac{1}{2}\right) \frac{1}{\sqrt{2}}\delta}$}\right), & \theta \in \left(\frac{\pi}{4} , \frac{\pi}{2} \right], \end{matrix*} \right. \nonumber \end{align}
as $N \rightarrow \infty$. This shows that the integral of $I(s,x)$ along the semicircle $C_{N}$ tends to zero as $N$ tends to infinity, for $|x|<1$ and $|\text{arg}(-x)|<\pi$. Due to Cauchy's theorem, we have,
\begin{align} \frac{1}{2 \pi i} \left(\int^{+i \infty}_{-i\infty} - \int^{+i\infty}_{\left(N+\frac{1}{2}\right)i} - \int_{C_{N}} - \int^{-\left(N+\frac{1}{2}\right) i }_{-i\infty} \right) I(s,x) \ ds = - \sum_{n=0}^{N} \underset{s=n}{\text{Res}} \ I(s,x). \label{eq:figfig} \end{align}
We note that there is a minus sign since the path of integration is a contour oriented clockwise, see Figure \ref{fig:figfig} below. 
\begin{figure}\begin{center}
\includegraphics[scale=.8]{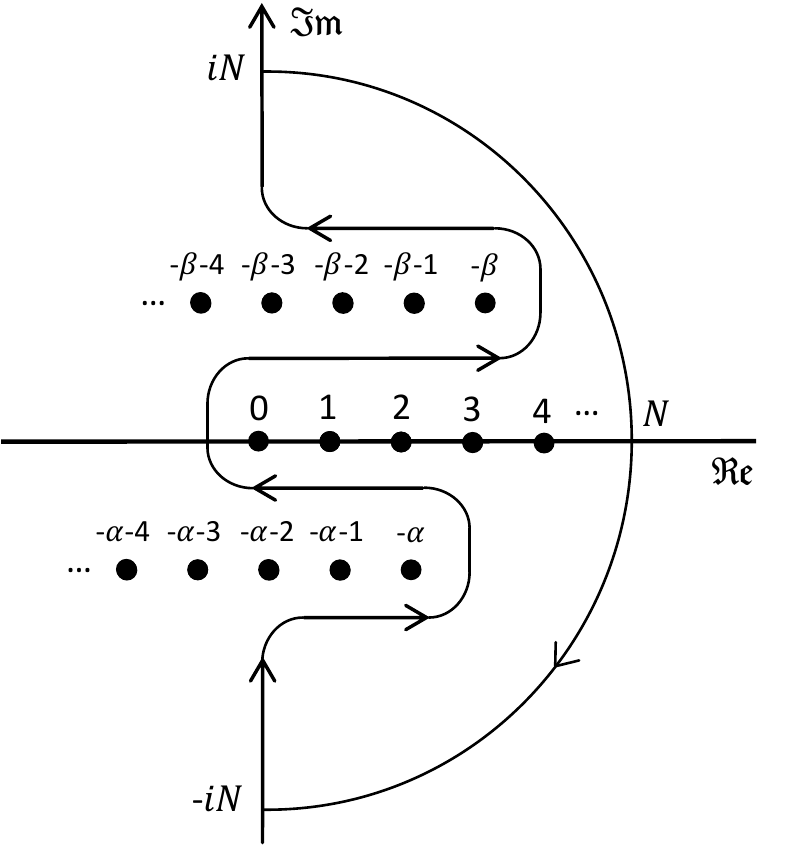}
\caption{\label{fig:figfig}Paths of integration along the imaginary axis and the semicircle $C_{N}$ as in (\ref{eq:figfig}).}
\end{center} \end{figure} 
Using $\underset{\lambda=-n}{\text{Res}}\Gamma(\lambda)=\frac{(-1)^{n}}{n!}$, for $n\geq0$, we compute the residues to find, 
\[\frac{1}{2 \pi i} \int^{+i\infty}_{-i\infty} I(s,x) \ ds = \lim_{N \rightarrow \infty} \sum_{n=0}^{N} \frac{\Gamma(\alpha+n)\Gamma(\beta+n)}{\Gamma(\gamma+n)\Gamma(n+1)}x^{n},\]
for $|x|<1$ and $|\text{arg}(-x)|<\pi$ and the desired result is proved after noting $\frac{\Gamma(\alpha+n)}{\Gamma(\alpha)} \equiv (\alpha)_{n}$. 

Finally, we carry out the analytic continuation of $y_{2}^{(0)}(x)$ for $|x|>1$.
The technique to derive the connection formulae is similar to that already used in the second part of this proof, the main difference being that we will now consider taking an integral on the left hand side of the imaginary axis. For $N \in \mathbb{N}$ consider the integral, 
\[\frac{1}{2 \pi i} \int_{C'_{N}} I(s,x) \ ds,\]
where $C'_{N}$ is the semicircle,
\[C'_{N} = \left\{ s = Ne^{i \theta} : \theta \in \left[ -\frac{3\pi}{2} , -\frac{\pi}{2}\right] \right\}.\]
We summarise the results, following a similar procedure as before. Using (\ref{eq:gammaasy}) we deduce, 
\[\frac{\Gamma(\alpha+s)\Gamma(\beta+s)\Gamma(-s)}{\Gamma(\gamma+s)} = \mathcal{O}\left(N^{\alpha+\beta-\gamma-1}\text{\large $e^{-N \pi |\sin(\theta)|}$}\right),\]
for $s \in C'_{N}$ as $N \rightarrow \infty$, and hence, 
\begin{align} I(s,x) &= \mathcal{O}\left(N^{\alpha+\beta-\gamma-1} \text{\large $e^{N\left(\cos(\theta)\log|x| -\sin(\theta)\text{arg}(-x) - \pi |\sin(\theta)|\right)}$}\right),\nonumber \\
&= \mathcal{O}\left(N^{\alpha+\beta-\gamma-1} \text{\large $e^{N \left(\cos(\theta)\log|x| - |\sin(\theta)| \delta \right)}$}\right), \nonumber \end{align}
where $\delta$ is a small positive number such that $|\text{arg}(-x)|\leq \pi - \delta$. Clearly $\cos(\theta)$ and $-|\sin(\theta)|$ are both non-positive for $\theta \in \left[-\frac{3\pi}{2} , -\frac{\pi}{2} \right]$ and they are never both simultaneously zero. Furthermore, for $|x|>1$ we have $\log|x|>0$, so that the integral of $I(s,x)$ along the semicircle $C'_{N}$ tends to zero as $N$ tends to infinity, for $|x|>1$ and $|\text{arg}(-x)|<\pi$. Due to Cauchy's theorem, we have,
\begin{align} &\frac{1}{2 \pi i} \left(\int^{+i\infty}_{-i\infty} - \int^{+i\infty}_{Ni} - \int_{C'_{N}} - \int^{-Ni}_{-i\infty}\right) I(s,x) \ ds \nonumber \\
&\quad \quad = \sum_{n=0}^{M_{1}(N)} \underset{s=\alpha-n}{\text{Res}} \ I(s,x) + \sum_{n=0}^{M_{2}(N)} \underset{s=\beta-n}{\text{Res}} \ I(s,x), \label{eq:figg} \end{align}
where $M_{1}(N)$ and $M_{2}(N)$ are the number of poles $-\alpha$, $-\alpha-1$, $\ldots$ and $-\beta$, $-\beta-1$, $\ldots$ which lie to the right of the semicircle respectively. Clearly $M_{1}(N)$ and $M_{2}(N)$ become infinite as $N$ tends to infinity, see Figure \ref{fig:figg} below. 
\begin{figure} \begin{center}
\includegraphics[scale=.8]{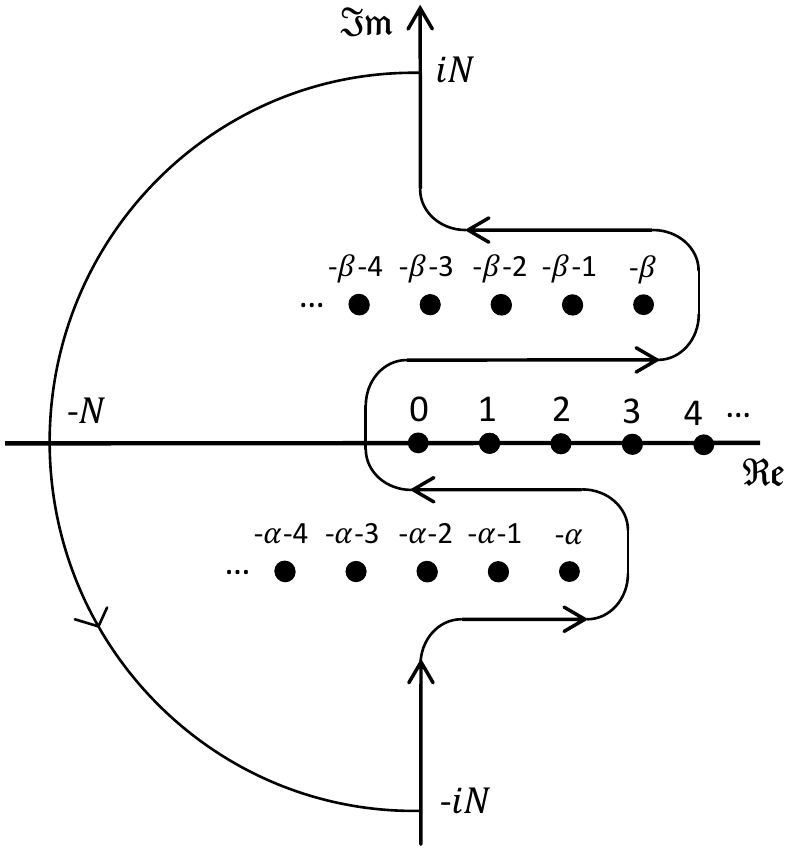}
\caption{\label{fig:figg}Paths of integration along the imaginary axis and the semicircle $C'_{N}$ as in (\ref{eq:figg}).}
\end{center} \end{figure} 
We compute the residues to find, 
\begin{align} \frac{1}{2\pi i} \int^{+i\infty}_{-i\infty}I(s,x) \ ds &= \frac{\Gamma(\alpha)\Gamma(\beta-\alpha)}{\Gamma(\gamma-\alpha)} (-x)^{-\alpha} \lim_{N \rightarrow \infty} \sum_{n=0}^{M_{1}(N)} \frac{(\alpha)_{n}(\alpha+1-\gamma)_{n}}{(\alpha+1-\beta)_{n}n!x^{n}} \nonumber \\
&\quad \quad + \frac{\Gamma(\beta)\Gamma(\alpha-\beta)}{\Gamma(\gamma-\beta)}(-x)^{-\beta} \lim_{N \rightarrow \infty} \sum_{n=0}^{M_{2}(N)} \frac{(\beta)_{n}(\beta+1-\gamma)_{n}}{(\beta+1-\alpha)_{n}n!x^{n}}, \nonumber \end{align}
for $|x|>1$ and $|\text{arg}(-x)|<\pi$ and the desired result is proved. \end{proof}

We conclude these computations by explaining how Proposition \ref{proposition:p} leads to the formulae (\ref{eq:ci0})-(\ref{eq:c10}). Let $\gamma_{j,k}$ be a curve as described at the beginning of this subsection. The second statement in proposition \ref{proposition:p} shows how to represent Gauss $\ _{2}F_{1}$ series using a Mellin-Barnes integral. Due to the analyticity of this integral, as shown in the first statement, the third statement provides the formula for the analytic continuation of Gauss hypergeometric series beyond its radius of convergence. That is to say,
\begin{align} \gamma_{0,\infty} \left[y_{2}^{(0)}\right](x) = \frac{\Gamma(\alpha-\beta)\Gamma(\gamma)}{\Gamma(\alpha-\gamma)\Gamma(\beta)} \ y_{1}^{(\infty)}(x) + \frac{\Gamma(\beta-\alpha)\Gamma(\gamma)}{\Gamma(\beta-\gamma)\Gamma(\alpha)} \ y_{2}^{(\infty)}(x).\nonumber \end{align}
By manipulating the parameters as follows: $\alpha \mapsto \alpha+1-\gamma$, $\beta \mapsto \beta+1-\gamma$, $\gamma \mapsto 2-\gamma$ and multiplying through by $x^{1-\gamma}$ we also deduce,
\begin{align} \gamma_{0,\infty}\left[y_{1}^{(0)}\right](x) &= -e^{-i \pi \gamma}\frac{\Gamma(\beta-\alpha)\Gamma(2-\gamma)}{\Gamma(1-\alpha)\Gamma(\beta+1-\gamma)} \ y_{1}^{(\infty)}(x)\nonumber \\
&\quad \quad - e^{-i \pi \gamma} \frac{\Gamma(\alpha-\beta)\Gamma(2-\gamma)}{\Gamma(1-\beta)\Gamma(\alpha+1-\gamma)} \ y_{2}^{(\infty)}(x), \nonumber \end{align}
recall that we have selected a branch of $\log(x)$ in the definition of our solutions (\ref{eq:y0}) around zero so $x^{1-\gamma}$ is well-defined. These factors constitute the entries of the connection matrix,
\begin{align} &\left(\gamma_{0,\infty}\left[y_{1}^{(0)}\right](x) , \ \gamma_{0,\infty}\left[y_{2}^{(0)}\right](x)\right) = \left(y_{1}^{(\infty)}(x) , \ y_{2}^{(\infty)}(x) \right) \ C^{\infty 0}, \nonumber \end{align}
where,
\begin{align} C^{\infty 0} = \left(\begin{array}{cc} -e^{-i \pi \gamma}\frac{\Gamma(\beta-\alpha)\Gamma(2-\gamma)}{\Gamma(1-\alpha)\Gamma(\beta+1-\gamma)} & \frac{\Gamma(\alpha-\beta)\Gamma(\gamma)}{\Gamma(\alpha-\gamma)\Gamma(\beta)} \\ - e^{-i \pi \gamma} \frac{\Gamma(\alpha-\beta)\Gamma(2-\gamma)}{\Gamma(1-\beta)\Gamma(\alpha+1-\gamma)} & \frac{\Gamma(\beta-\alpha)\Gamma(\gamma)}{\Gamma(\beta-\gamma)\Gamma(\alpha)} \end{array} \right), \nonumber \end{align}
which is indeed the inverse of the connection matrix $C^{0 \infty}$ as given by (\ref{eq:ci0}). To find the analytic continuation of the solutions around $x=1$ we manipulate the variable $x$ as well as the parameters. From the transformations $\alpha \mapsto \alpha$, $\beta \mapsto \beta$, $\gamma \mapsto \alpha+\beta+1-\gamma$ and $x \mapsto 1-x$, we have,
\begin{align} &\gamma_{1,\infty}\left[y_{2}^{(1)}\right](x) = \nonumber \\
&\quad \quad e^{-i \pi \alpha} \frac{\Gamma(\beta-\alpha)\Gamma(\alpha+\beta+1-\gamma)}{\Gamma(\beta)\Gamma(\beta+1-\gamma)} (1-x)^{-\alpha} \ _{2}F_{1}\left(\begin{array}{c} \alpha , \ \gamma-\beta \\ \alpha+1-\beta \end{array} ; (1-x)^{-1} \right) \nonumber \\
&\quad \quad + e^{-i \pi \beta} \frac{\Gamma(\alpha-\beta)\Gamma(\alpha+\beta+1-\gamma)}{\Gamma(\alpha)\Gamma(\alpha+1-\gamma)} (1-x)^{-\beta} \ _{2}F_{1}\left(\begin{array}{c} \beta , \ \gamma-\alpha \\ \beta+1-\alpha \end{array} ; (1-x)^{-1} \right), \nonumber \end{align}
and from the transformations $\alpha \mapsto \gamma-\alpha$, $\beta \mapsto \gamma-\beta$, $\gamma \mapsto \gamma+1-\alpha-\beta$ and $x \mapsto 1-x$,
\begin{align} &\gamma_{1,\infty}\left[y_{1}^{(1)}\right](x) = \nonumber \\
&\quad \quad e^{i \pi (\beta-\gamma)} \frac{\Gamma(\beta-\alpha)\Gamma(\gamma+1-\alpha-\beta)}{\Gamma(1-\alpha)\Gamma(\gamma-\alpha)} (1-x)^{-\alpha} \ _{2}F_{1}\left(\begin{array}{c} \alpha , \ \gamma-\beta \\ \alpha+1-\beta \end{array} ; (1-x)^{-1} \right) \nonumber \\
&\quad \quad + e^{i \pi (\alpha-\gamma)} \frac{\Gamma(\alpha-\beta)\Gamma(\gamma+1-\alpha-\beta)}{\Gamma(1-\beta)\Gamma(\gamma-\beta)} (1-x)^{-\beta} \ _{2}F_{1} \left(\begin{array}{c} \beta , \ \gamma-\alpha \\ \beta+1-\alpha \end{array} ; (1-x)^{-1} \right), \nonumber \end{align}
both for $|\text{arg}(x-1)|<\pi$ and $|x-1|>1$. After applying Kummer transformation, 
\[(1-x)^{-a} \ _{2}F_{1} \left(\begin{array}{c} a, \ c-b \\ a+1-b \end{array} ; (1-x)^{-1}\right) = (-x)^{-a} \ _{2}F_{1}\left(\begin{array}{c} a , \ a+1-c \\ a+1-b \end{array} ; x^{-1} \right),\]
which is valid for $|\text{arg}(x-1)|<\pi$, $|\text{arg}(-x)|<\pi$, $|x-1|>1$ and $|x|>1$, we deduce the connection matrix, 
\begin{align} &\left(\gamma_{0,\infty}\left[y_{1}^{(1)}\right](x) , \ \gamma_{0,\infty}\left[y_{2}^{(1)}\right](x)\right) = \left(y_{1}^{(\infty)}(x) , \ y_{2}^{(\infty)}(x) \right) \ C^{\infty 1}, \nonumber \end{align}
where, 
\begin{align} C^{\infty 1} = \left(\begin{array}{cc} e^{i \pi (\beta-\gamma)} \frac{\Gamma(\beta-\alpha)\Gamma(\gamma+1-\alpha-\beta)}{\Gamma(1-\alpha)\Gamma(\gamma-\alpha)} & e^{-i \pi \alpha} \frac{\Gamma(\beta-\alpha)\Gamma(\alpha+\beta+1-\gamma)}{\Gamma(\beta)\Gamma(\beta+1-\gamma)} \\ e^{i \pi (\alpha-\gamma)} \frac{\Gamma(\alpha-\beta)\Gamma(\gamma+1-\alpha-\beta)}{\Gamma(1-\beta)\Gamma(\gamma-\beta)} & e^{-i \pi \beta} \frac{\Gamma(\alpha-\beta)\Gamma(\alpha+\beta+1-\gamma)}{\Gamma(\alpha)\Gamma(\alpha+1-\gamma)} \end{array} \right), \nonumber \end{align}
which is indeed the inverse of the connection matrix $C^{1 \infty}$ as given by (\ref{eq:ci1}). The connection matrix $C^{01}$ as in (\ref{eq:c10}) can be deduced from the relation,
\[C^{01} = C^{\infty 1} C^{0\infty}.\]

\section{Appendix B: Mellin-Barnes integral for Kummer equation} \

In Appendix B, we follow the classical approach to show that these solutions can be expressed in closed form by certain Mellin-Barnes integrals and thus derive the connection matrices. This analysis allows us to explicitly compute the monodromy data, including Stokes matrices, of Kummer equation in the following section and thus obtain a richer understanding of Stokes phenomenon.

The remainder of this subsection is dedicated to deriving the classical formulae (\ref{eq:kummers-1})-(\ref{eq:kummerc0}). This is a valuable exercise in its own right as it gives us a richer understanding of Stokes phenomenon using a concrete example. Our approach is to use Mellin-Barnes integrals to represent the fundamental solutions $\widetilde{Y}^{(\infty,k)}(z)$, as defined in Theorem \ref{theorem:kummerst}, for which we are able to compute their analytic continuations. Our analysis of Mellin-Barnes integrals is based on Whittaker and Watson's \cite{ww} \textsection 16, who study a different form of the confluent hypergeometric differential equation but is equivalent to ours using analytic transformations. \\

Define the following functions,
\begin{align} &\begin{matrix*}[l] \tilde{y}_{1}^{(\infty,-1)}(z) = e^{-i \pi(\beta-\gamma)} e^{z} \varphi \left(\gamma-\beta, \gamma ; e^{i \pi}z \right), \\ \tilde{y}_{2}^{(\infty,-1)}(z) = -\varphi(\beta, \gamma ; z ), \end{matrix*} &&z \in \widetilde{\Sigma}_{-1}, \label{eq:true2} \\
&\begin{matrix*}[l] \tilde{y}_{1}^{(\infty,0)}(z) = e^{i \pi(\beta-\gamma)} e^{z} \varphi \left(\gamma-\beta, \gamma ; e^{-i \pi}z \right), \\ \tilde{y}_{2}^{(\infty,0)}(z) = -\varphi(\beta, \gamma ; z ), \end{matrix*} &&z \in \widetilde{\Sigma}_{0}, \label{eq:true1} \end{align}
where $\varphi$ is the Mellin-Barnes integral, 
\begin{align} \varphi (\beta,\gamma;z) = \frac{1}{2 \pi i} \int _{-i\infty}^{+i\infty} \frac{\Gamma(s)\Gamma(\beta-s)\Gamma(\beta+1-\gamma-s)}{\Gamma(\beta)\Gamma(\beta+1-\gamma)}z^{s-\beta} \ ds, \label{eq:mb} \end{align}
whose path of integration is along the imaginary axis with indentations as necessary so that the poles of $\Gamma(s)$ lie on its left and the poles of $\Gamma(\beta-s)\Gamma(\beta+1-\gamma-s)$ lie on its right, as shown in Figure \ref{fig:mbint} below. When dealing with $\varphi(\beta,\gamma;z)$ it is to be understood that arg$(z)$ belongs to an interval of length at most $2\pi$, as in (\ref{eq:true2}) and (\ref{eq:true1}), so that we have a well-defined function. 
\begin{figure}\begin{center}
\includegraphics[scale=.8]{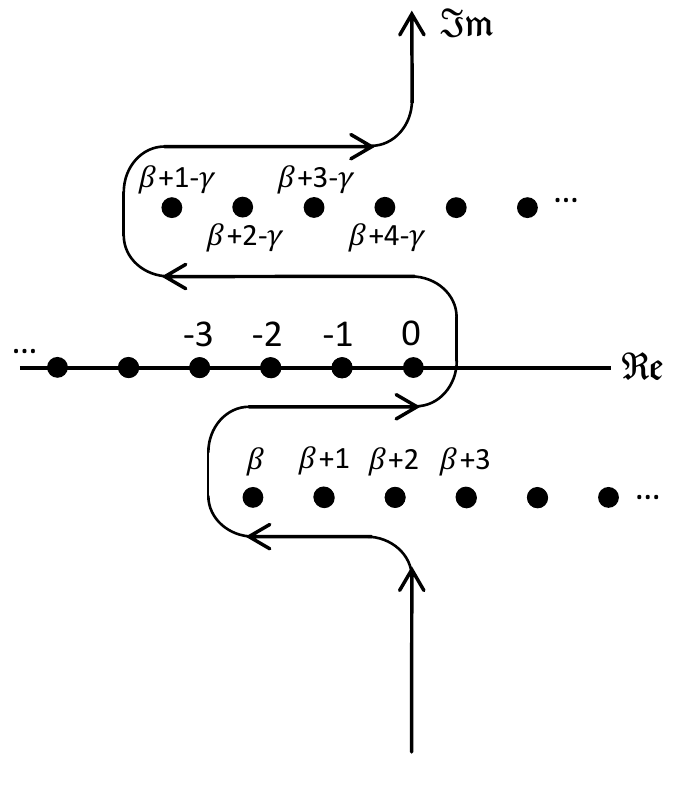}
\caption{\label{fig:mbint} Path of integration in the Mellin-Barnes integral $\varphi(\beta,\gamma;z)$, the dots represent the poles of the integrand.}
\end{center} \end{figure} 

\begin{proposition} \label{proposition:kp} Let $\widetilde{Y}^{(\infty,k)}(z)$ be the fundamental solutions defined in Theorem \ref{theorem:kummerst}. Also let $\tilde{y}_{1}^{(\infty,k)}(z)$ and $\tilde{y}_{2}^{(\infty,k)}(z)$, $k=-1,0$, be the functions defined in (\ref{eq:true2}) and (\ref{eq:true1}) and denote by $\widetilde{Y}\left(\tilde{y}_{1} , \tilde{y}_{2} ; z\right)$ the matrix function given by (\ref{eq:y2}). We have, 
\begin{align} &\widetilde{Y}\left( \tilde{y}_{1}^{(\infty,k)} , \tilde{y}_{2}^{(\infty,k)} ; z \right) = \widetilde{Y}^{(\infty,k)}(z), &&z \in \widetilde{\Sigma}_{k}, \label{eq:toshow} \end{align} 
for $k=-1,0$. \end{proposition}

\begin{proof} We prove this proposition in three steps: we first show that the functions $\tilde{y}_{1}^{(\infty,k)}(z)$ and $\tilde{y}_{2}^{(\infty,k)}(z)$ are analytic on their respective sectors; using this fact, we secondly show that these functions satisfy Kummer equation (\ref{eq:kummer}); finally, we show that these functions have the correct asymptotic behaviour (\ref{eq:chgasy}). By the uniqueness statement of Theorem \ref{theorem:kummerst}, these conditions are sufficient to conclude (\ref{eq:toshow}).

\vskip 2mm
First step: analyticity of $\tilde{y}_{1}^{(\infty,k)}(z)$ and $\tilde{y}_{2}^{(\infty,k)}(z)$. 

We require formula (\ref{eq:gammaasy}) and Lemma \ref{lemma:ww}, as used in the derivation of Gauss monodromy data formulae. Using (\ref{eq:gammaasy}), we have the following behaviour in the integrand of $\varphi(a,c;z)$,
\begin{align}\Gamma(s)\Gamma(\beta-s)\Gamma(\beta+1-\gamma-s) = \mathcal{O}\left(e^{-\frac{3\pi}{2}|s|}|s|^{2\beta-\gamma-\frac{1}{2}}\right), \quad \text{as $|s|\rightarrow \infty$} \label{eq:behavior}\end{align}
along the contour of integration. We therefore need only consider the analyticity of the following integral, 
\begin{align} &\int_{-i\infty}^{+i\infty} e^{-\frac{3\pi}{2}|s|} z^{s-\beta} \ ds \nonumber \\ 
&\quad \equiv i \int_{0}^{\infty}e^{-\frac{3\pi}{2}|\sigma|} z^{-\beta}e^{i\sigma(\log|z| + i \text{arg}(z))} \ d\sigma - i \int_{0}^{\infty} e^{-\frac{3\pi}{2}|\sigma|} z^{-\beta}e^{-i\sigma(\log|z| + i \text{arg}(z))} \ d\sigma. \nonumber \end{align}
Applying Lemma \ref{lemma:ww} to the first integral, with $r=-\frac{3\pi}{2}$, $K=1$ and $\lambda = \text{arg}(z)$, we find an analytic function for $-\frac{3\pi}{2} < \text{arg}(z)$. Applying Lemma \ref{lemma:ww} to the second integral, with $r=-\frac{3\pi}{2}$, $K=1$ and $\lambda = -\text{arg}(z)$, we find an analytic function for arg$(z) < \frac{3\pi}{2}$. We conclude that $\varphi(\beta,\gamma;z)$ defines analytic functions $\tilde{y}_{2}^{(\infty,-1)}(z)$ and $\tilde{y}_{2}^{(\infty,0)}(z)$ on their respective sectors $\widetilde{\Sigma}_{-1}$ and $\widetilde{\Sigma}_{0}$. It therefore follows that $\tilde{y}_{1}^{(\infty,-1)}(z)$ and $\tilde{y}_{1}^{(\infty,0)}(z)$ are also analytic functions, since $\varphi \left(\gamma-\beta-1,\gamma;e^{i \pi z} \right)$ must be analytic on $z \in \widetilde{\Sigma}_{-1}$ and $\varphi\left(\gamma-\beta-1,\gamma;e^{-i \pi z}\right)$ must be analytic on $z \in \widetilde{\Sigma}_{0}$. 
\vskip 2mm
Second step:  Showing $\tilde{y}_{1}^{(\infty,k)}(z)$ and $\tilde{y}_{2}^{(\infty,k)}(z)$ satisfy the Kummer equation (\ref{eq:kummer}). 

We will now substitute $\varphi(\beta,\gamma;z)$ for $\tilde{y}(z)$ into the left hand side of Kummer equation (\ref{eq:kummer}) and show that the result is zero. Having established the analyticity of $\varphi(\beta,\gamma;z)$ on the sectors $\widetilde{\Sigma}_{-1}$ and $\widetilde{\Sigma}_{0}$, we can compute the derivatives of this integral by taking the derivatives inside the integral. After multiplying through by $2 \pi i \Gamma(\beta)\Gamma(\beta+1-\gamma)$ to cancel all multiplicative constant terms, we find,  
\begin{align} & \left(z \ \varphi''(\beta,\gamma;z) + (\gamma-z) \ \varphi'(\beta,\gamma;z) - \beta \ \varphi(\beta,\gamma;z)\right) 2 \pi i \Gamma(\beta)\Gamma(\beta+1-\gamma) \nonumber \\
&= \int_{-i\infty} ^{+i \infty} \Gamma(s)\Gamma(\beta+2-s)\Gamma(\beta+1-\gamma-s)z^{s-\beta-1} \ ds \nonumber \\
& \quad \quad \quad -\int_{-i\infty}^{+i\infty}\gamma\Gamma(s)\Gamma(\beta+1-s)\Gamma(\beta+1-\gamma-s) z^{s-\beta-1} \ ds \nonumber \\
& \quad \quad \quad + \int_{-\infty}^{+\infty} \Gamma(s)\Gamma(\beta+1-s)\Gamma(\beta+1-\gamma-s)z^{s-\beta} \ ds \nonumber \\
& \quad \quad \quad - \int_{-\infty}^{+\infty} (\beta)\Gamma(s)\Gamma(\beta-s)\Gamma(\beta+1-\gamma-s) z^{s-\beta} \ ds \nonumber \\
&= \int^{-1+i \infty}_{-1-i \infty} \Gamma(s+1) \Gamma(\beta-\gamma-s)z^{s-\beta} \left(\Gamma(\beta+1-s) - \gamma \Gamma(\beta-s) \right) \ ds  \nonumber \\
& \quad \quad \quad - \int_{-i \infty}^{+i \infty} \Gamma(s)\Gamma(\beta+1-\gamma-s) z^{s-\beta} \left((\beta)\Gamma(\beta-s)-\Gamma(\beta+1-s)\right)\ ds \nonumber \\
&= \left( \int^{-1+i \infty}_{-1-i \infty} - \int^{+i \infty}_{- i \infty} \right) \Gamma(s+1)\Gamma(\beta-s)\Gamma(\beta+1-\gamma-s)z^{s-\beta} \ ds. \label{eq:expression} \end{align}
Due to the choice of the path of integration, the final integrand has no poles between the contours of integration, see Figure \ref{fig:extra} below. Therefore, due to Cauchy's theorem, the expression equals zero and we have shown that $\varphi(\beta,\gamma;z)$ satisfies Kummer confluent hypergeometric equation (\ref{eq:kummer}) on $z \in \widetilde{\Sigma}_{-1}$ and $\widetilde{\Sigma}_{0}$. 
\begin{figure} \begin{center}
\includegraphics[scale=.8]{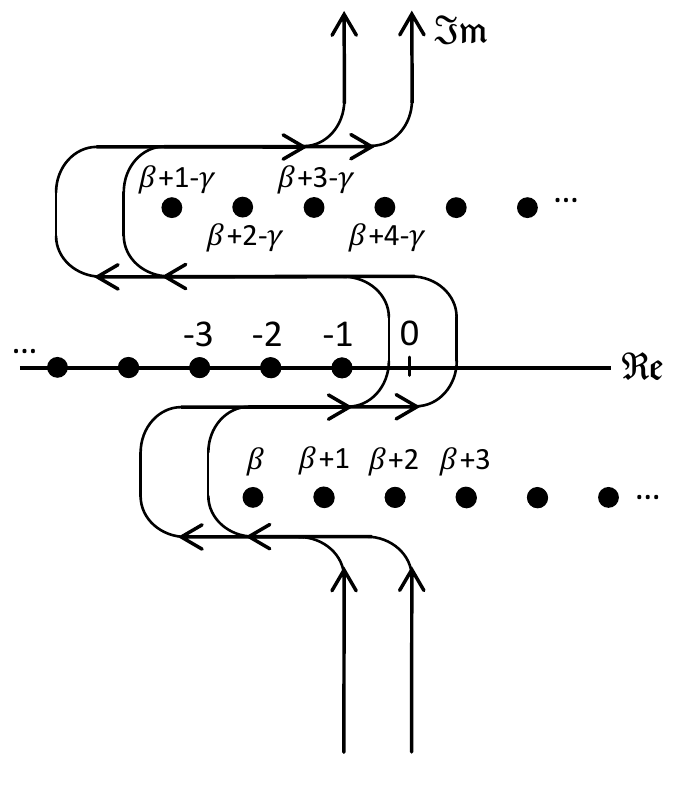}
\caption{\label{fig:extra}Paths of integration in (\ref{eq:expression}), the dots represent poles of the integrand. Note the crucial detail that $s=0$ is not a pole of the integrand, so there are no singularities between the two paths.}
\end{center} \end{figure} 

Observe the following differential identity, 
\begin{align} &z \ \frac{d^{2}}{dz^{2}}\left(e^{z}f(-z)\right) + (\gamma-z) \ \frac{d}{dz}\left(e^{z}f(-z)\right) - \beta \ e^{z}f(-z) \nonumber \\
&\quad \equiv e^{z} \left(z \ \frac{d^{2}}{dz^{2}}f(z) - (\gamma - (-z)) \ \frac{d}{dz}f(z) - (\gamma-\beta) \ f(z)\right). \nonumber \end{align} 
Given that $\varphi(\beta,\gamma;z)$ satisfies Kummer equation (\ref{eq:kummer}), it follows that the right hand side of this identity equals zero for $f(-z) = \varphi(\gamma-\beta,\gamma;-z)$. Looking at the left hand side of the identity, we deduce that $e^{z}\varphi(\gamma-\beta,\gamma;-z)$ also satisfies equation (\ref{eq:kummer}). 

\vskip 2mm Third step: Asymptotic behaviour of $\tilde{y}_{1}^{(\infty,k)}(z)$ and $\tilde{y}_{2}^{(\infty,k)}(z)$ for large $|z|$.

Recalling the formal solutions given in Remark \ref{remark:formal}, we will deduce the following asymptotics, for $j\in\{0,-1\}$:
\begin{align} y_{1}^{(\infty,j)}(z) &\sim e^{z}z^{\beta-\gamma} \ _{2}F_{0}\left(\gamma-\beta , \ 1-\beta ; z^{-1} \right), \quad \text{as } z \rightarrow \infty, \ z \in \widetilde{\Sigma}_{j}, \label{eq:firstasy} \\
y_{2}^{(\infty,j)}(z) &\sim -z^{-\beta} \ _{2}F_{0} \left(\beta , \ \beta+1-\gamma ; -z^{-1}\right), \quad \text{as } z \rightarrow \infty, \ z \in \widetilde{\Sigma}_{j}. \label{eq:secondasy} \end{align}
Denote the integrand of $\varphi(\beta,\gamma;z)$ by, 
\begin{align} I(s,z) = \frac{\Gamma(s)\Gamma(\beta-s)\Gamma(\beta+1-\gamma-s)}{\Gamma(\beta)\Gamma(\beta+1-\gamma)}z^{s-\beta}, \label{eq:integrand} \end{align}
and let $\tau$ be a large, positive real number. For $N \geq 0$, consider the path of integration along the rectangle $R$ with vertices at $\pm i \tau$ and $-N - \frac{1}{2} \pm i \tau$, with indentations so that the poles of the integrand are separated as usual and with a positive orientation as shown in Figure \ref{fig:curve7} below. 
\begin{figure} \begin{center}
\includegraphics[scale=.8]{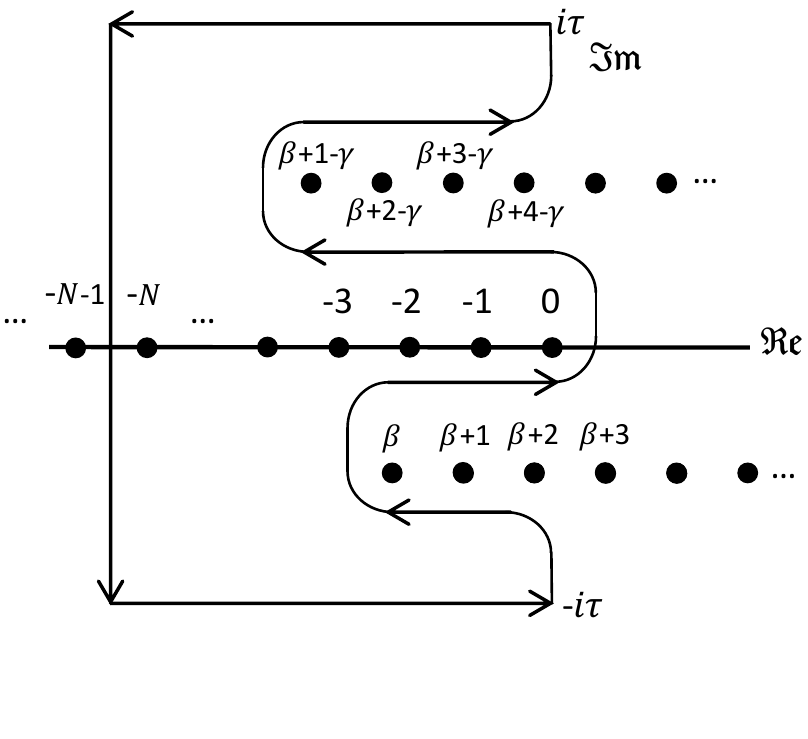}
\caption{\label{fig:curve7}Path of integration around the rectangle $R$, the dots represent the poles of the integrand of $\varphi(\beta,\gamma;z)$.}
\end{center} \end{figure} 
By Cauchy's theorem, we have, 
\begin{align} \frac{1}{2\pi i} \int_{R} I(s,z) \ ds &= \frac{1}{2\pi i} \left( \int_{-N-\frac{1}{2}-i \tau}^{-i \tau} + \int_{-i \tau}^{+i \tau} + \int_{+i \tau}^{-N -\frac{1}{2}+i \tau} + \int_{-N-\frac{1}{2}+i \tau}^{-N-\frac{1}{2}-i \tau} \right) I(s,z) \ ds. \nonumber \\
&= \sum _{n=0}^{N} \underset{s=-n}{\text{Res}}I(s,z), \nonumber \end{align}
We examine these integrals in the limit $\tau \rightarrow \infty^{+}$ one-by-one, using the asymptotics (\ref{eq:gammaasy}) of the Gamma function: 
\begin{enumerate} 
\item By writing $s=x-i \tau$ in the first integral we obtain, 
\begin{align} e^{\tau \left(\text{arg}(z)-\frac{3\pi}{2}\right)} \int_{-N-\frac{1}{2}}^{0}\mathcal{O}\left(|z|^{x}\tau^{\text{Re}(2\beta-\gamma)-x-\frac{1}{2}}\right) \ dx, \nonumber \end{align}
which tends to zero as $\tau \rightarrow \infty^{+}$, thanks to arg$(z) < \frac{3\pi}{2}$. \\
\item In the limit $\tau \rightarrow \infty^{+}$, the second integral becomes $\varphi(\beta,\gamma;z)$, by definition. \\
\item Similarly to the first integral, by writing $s=x+i\tau$ in the third integral, we obtain,
\begin{align} e^{-\tau \left(\text{arg}(z)+\frac{3\pi}{2}\right)} \int_{0}^{-N-\frac{1}{2}} \mathcal{O} \left(|z|^{x}\tau^{\text{Re}(2\beta-\gamma)-x-\frac{1}{2}} \right) \ dx, \nonumber \end{align}
which also tends to zero as $\tau \rightarrow \infty^{+}$, thanks to arg$(z) > -\frac{3\pi}{2}$.\\ 
\item We write $s=-N-\frac{1}{2}+i y$ in the fourth integral to obtain, 
\begin{align} \int_{-N-\frac{1}{2}+i\tau}^{-N-\frac{1}{2}-i \tau}I(s,z) \ ds &= i \left(\int_{0}^{-\tau} + \int_{\tau}^{0} \right) I\left(-N-\frac{1}{2}+iy,z\right) \ dy \nonumber \\
&= iz^{-N-\frac{1}{2}-\beta} \left( \int_{0}^{\tau} \mathcal{O} \left(|y|^{N+\text{Re}(2\beta-\gamma)-1}e^{-y\left(\frac{3\pi}{2}-\text{arg}(z)\right)}\right) \ dy  \right. \nonumber \\
& \quad \left. - \int^{\tau}_{0} \mathcal{O}\left( |y|^{N+\text{Re}(2\beta-\gamma)-1}e^{-y\left( \text{arg}(z)+\frac{3\pi}{2}\right)}\right) \ dy \right). \label{eq:important} \end{align}
Using the fact that $\lim_{\tau \rightarrow \infty^{+}} \int_{0}^{\tau}e^{-k y} \ dy$ for $k>0$ exists, the limit as $\tau \rightarrow \infty^{+}$ of the fourth integral exists and is of order $\mathcal{O}\left( |z|^{-N-\frac{1}{2}-\beta} \right)$ as $\tau \rightarrow \infty^{+}$, thanks to $|\text{arg}(z)|<\frac{3\pi}{2}$. \end{enumerate}
Summarising the above analysis, we have shown that for large $\tau$,  
\begin{align}\varphi(\beta,\gamma;z) &= \sum_{n=0}^{N} \underset{s=-n}{\text{Res}}I(s,z) + \mathcal{O}\left(|z|^{-N-\frac{1}{2}-\beta}\right), \label{eq:notfound} \\
&= z^{-\beta} \sum_{n=0}^{N} \frac{(\beta)_{n}(\beta+1-\gamma)_{n}}{(-z)^{n}n!} + \mathcal{O}\left(|z|^{-N-\frac{1}{2}-\beta}\right), \nonumber \end{align}
where we have used the formula $\underset{\lambda=-n}{\text{Res}}\Gamma(\lambda) = \frac{(-1)^{n}}{n!}$, for $n \geq 0$, to calculate the residues. This proves (\ref{eq:secondasy}). Moreover, for $N \geq 0$, we can immediately deduce, 
\[e^{\mp i \pi (\beta-\gamma)} e^{z} \varphi \left(\gamma-\beta, \gamma ; e^{\pm i \pi}z \right) = e^{z} z^{\beta-\gamma} \sum_{n=0}^{N} \frac{(\gamma-\beta)_{n}(1-\beta)_{n}}{z^{n}n!} + \mathcal{O}\left(e^{z}|z|^{-N-\frac{1}{2}+\beta-\gamma}\right),\]
which proves (\ref{eq:firstasy}). \end{proof}

\begin{remark} \label{remark:sub} The expression (\ref{eq:notfound}) is valid for all finite $N$. In order to take the limit as $N \rightarrow \infty$ it is important to understand that (\ref{eq:notfound}) becomes an asymptotic result. This is because the integrals in (\ref{eq:important}) depend on $N$ and, in particular, they diverge as $N \rightarrow \infty^{+}$, hence the interchange between limits $\lim_{N\rightarrow \infty^{+}}$ and $\lim_{\tau \rightarrow \infty^{+}}$ is not justified here. \end{remark}

Having established how to represent the fundamental solutions $\widetilde{Y}^{(\infty,k)}(z)$ using Mellin-Barnes integrals, we now show how to analytically continue them to $z=0$. We will prove the following proposition, which is sufficient to deduce the monodromy data formulae (\ref{eq:kummers-1})-(\ref{eq:kummerc0}).

\begin{proposition} \label{proposition:kpp} Let $\tilde{y}_{1}^{(0)}(z)$ and $\tilde{y}_{2}^{(0)}(z)$ be the solutions as given in (\ref{eq:yt0}). For $- \pi \pm\frac{\pi}{2} < \text{arg}(z) < \pi \pm \frac{\pi}{2}$, the integral as given by (\ref{eq:mb}) satisfies,
\[\varphi(\beta,\gamma;z) = \frac{\Gamma(\gamma-1)}{\Gamma(\beta)} \tilde{y}_{1}^{(0)}(z) + \frac{\Gamma(1-\gamma)}{\Gamma(\beta+1-\gamma)} \tilde{y}_{2}^{(0)}(z).\] \end{proposition}

\begin{proof} Let $I(s,z)$ be the integrand of $\varphi(\beta,\gamma;z)$ as given by (\ref{eq:integrand}). For large $\tau>0$ and an integer $N>0$, we now consider the integral around the rectangle $R'$ with vertices $\pm i \tau$ and $N + \frac{1}{2} \pm i \tau$, with indentations along the imaginary axis as usual and with a negative orientation as shown in Figure \ref{fig:curve9} below. Our analysis of this integral is analogous to that of the integral around the rectangle $R$, which lies to the left of the imaginary axis. 
\begin{figure}\begin{center}
\includegraphics[scale=.8]{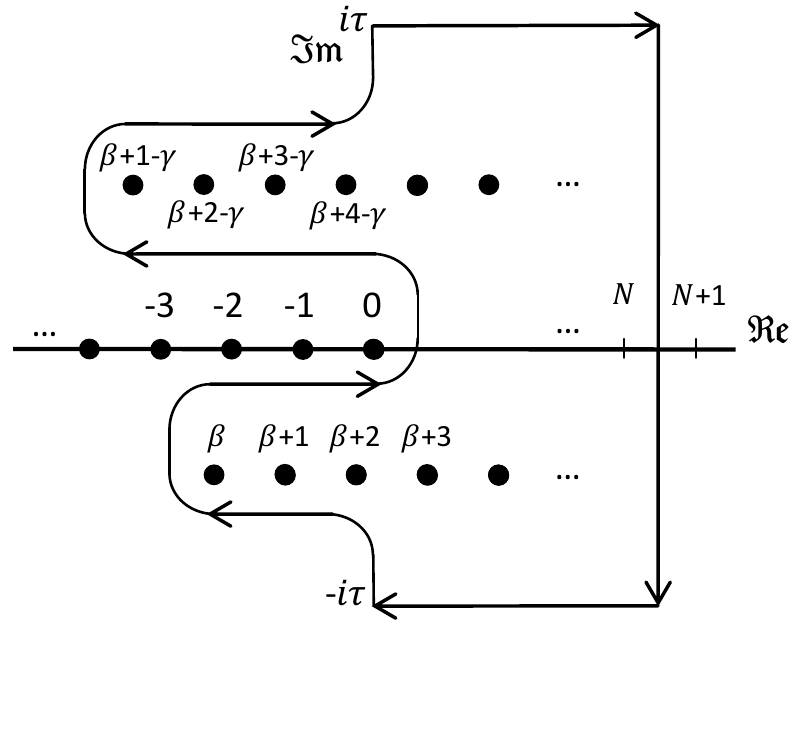}
\caption{\label{fig:curve9}Path of integration around the rectangle $R'$, the dots represent the poles of the integrand of $\varphi(\beta,\gamma;z)$.}
\end{center} \end{figure} 
By Cauchy's theorem, we have, 
\begin{align} \frac{1}{2\pi i} \int_{R'}I(s,z) \ ds &\equiv \frac{1}{2\pi i} \left( \int_{N + \frac{1}{2} -i \tau}^{-i \tau} + \int_{-i \tau}^{i \tau} + \int_{i \tau}^{N + \frac{1}{2}+i \tau} + \int_{N + \frac{1}{2} + i \tau}^{N + \frac{1}{2}-i \tau}\right) I(s,z) \ ds \nonumber \\
&= -\sum_{n=0}^{M_{1}(N)} \underset{s=\beta+1-\gamma+n}{\text{Res}}I(s,z) - \sum_{n=0}^{M_{2}(N)} \underset{s=\beta+n}{\text{Res}}I(s,z) \nonumber \end{align}
where $M_{1}(N)$ and $M_{2}(N)$ are the number of poles $\beta+1-\gamma$, $\beta+2-\gamma$, $\ldots$ and $\beta$, $\beta+1$, $\ldots$ which lie inside the rectangle respectively. We examine these integrals under the limit $\tau \rightarrow \infty^{+}$ one-by-one, using the asymptotics (\ref{eq:gammaasy}) of the Gamma function:
\begin{enumerate} \item By writing $s=x-i \tau$ in the first integral we obtain,
\[ e^{\tau \left(\text{arg}(z)-\frac{3\pi}{2}\right)} \int_{N + \frac{1}{2}}^{0} \mathcal{O} \left(|z|^{x} \tau^{\text{Re}(2\beta - \gamma)-x-\frac{1}{2}}\right) \ dx,\]
which tends to zero as $\tau \rightarrow \infty^{+}$, thanks to arg$(z) < \frac{3\pi}{2}$. \\
\item In the limit $\tau \rightarrow \infty^{+}$, the second integral becomes $\varphi(\beta,\gamma;z)$, by definition. \\ 
\item Similarly to the first integral, by writing $s=x+i\tau$ in the third integral, we obtain, 
\[e^{-\tau \left(\text{arg}(z)+\frac{3\pi}{2}\right)} \int_{i \tau}^{N + \frac{1}{2} + i \tau} \mathcal{O}\left(|z|^{x} \tau^{\text{Re}(2\beta-\gamma)-x-\frac{1}{2}}\right) \ dx, \]
which also tends to zero as $\tau \rightarrow \infty^{+}$, thanks to arg$(z) > -\frac{3\pi}{2}$. \\
\item We write $s=N + \frac{1}{2}+iy$ in the fourth integral, to obtain, 
\begin{align} \int_{N+\frac{1}{2}+i\tau}^{N+\frac{1}{2}-i \tau}I(s,z) \ ds &= i \left(\int_{0}^{-\tau} + \int_{\tau}^{0} \right) I\left(N+\frac{1}{2}+iy,z\right) \ dy \nonumber \\
&= iz^{N+\frac{1}{2}-\beta} \left( \int_{0}^{\tau} \mathcal{O} \left(|y|^{-N+\text{Re}(2\beta-\gamma)-2}e^{-y\left(\frac{3\pi}{2}-\text{arg}(z)\right)}\right) \ dy \right. \nonumber \\
& \quad \left. - \int_{0}^{\tau} \mathcal{O}\left( |y|^{-N+\text{Re}(2\beta-\gamma)-2}e^{-y\left( \text{arg}(z)+\frac{3\pi}{2}\right)}\right) \ dy \right). \label{eq:important2} \end{align}
Using the fact that $\lim_{\tau \rightarrow \infty^{+}} \int^{\tau}_{0}e^{-k y} \ dy$ for $k>0$ exists, the limit as $\tau \rightarrow \infty^{+}$ of fourth integral exists, thanks to $|\text{arg}(z)|< \frac{3\pi}{2}$. Moreover, for $|z|$ sufficiently small, this limit exists uniformly with respect to large $N$, due to the minus sign in the exponent of $|y|$. In particular, for $|z|$ sufficiently small, 
\[\lim_{N \rightarrow \infty^{+}} \int^{N+\frac{1}{2}-i \tau}_{N+\frac{1}{2}-i \tau} I(s,z) \ ds = 0.\] \end{enumerate}

Summarising the above analysis, we have shown the following,
\begin{align} \varphi(\beta,\gamma;z) = -\sum_{n=0}^{M_{1}(N)} \underset{s=\beta+1-\gamma-n}{\text{Res}}I(s,z) - \sum_{n=0}^{M_{2}(N)} \underset{s=\beta-n}{\text{Res}}I(s,z) + \lim_{\tau \rightarrow \infty^{+}} \int_{N + \frac{1}{2}-i \tau}^{N + \frac{1}{2}+i \tau}I(s,z) \ ds,\nonumber \end{align}
where the convergence of the limit of this integral is uniform with respect to $N \rightarrow \infty^{+}$. As such, we may interchange the limits $\lim_{\tau\rightarrow\infty^{+}}$ and $\lim_{N \rightarrow \infty^{+}}$ as follows,
\begin{align} \varphi(\beta,\gamma;z) 
&= -\lim_{N \rightarrow \infty^{+}}\sum_{n=0}^{M_{1}(N)} \underset{s=\beta+1-\gamma+n}{\text{Res}}I(s,z) - \lim_{N \rightarrow \infty^{+}} \sum_{n=0}^{M_{2}(N)} \underset{s=\beta+n}{\text{Res}}I(s,z) \nonumber \\
&\hspace{80pt} + \lim_{N \rightarrow \infty^{+}} \lim_{\tau \rightarrow \infty^{+}} \int_{N + \frac{1}{2}-i \tau}^{N + \frac{1}{2}+i \tau}I(s,z) \ ds,\nonumber \\
&= -\sum_{n=0}^{\infty} \underset{s=\beta+1-\gamma+n}{\text{Res}}I(s,z) - \sum_{n=0}^{\infty} \underset{s=\beta+n}{\text{Res}}I(s,z) \nonumber \\
&\hspace{80pt} + \lim_{\tau \rightarrow \infty^{+}} \lim_{N \rightarrow \infty^{+}} \int_{N + \frac{1}{2}-i \tau}^{N + \frac{1}{2}+i \tau}I(s,z) \ ds,\nonumber \\
&= -\sum_{n=0}^{\infty} \underset{s=\beta+1-\gamma+n}{\text{Res}}I(s,z) - \sum_{n=0}^{\infty} \underset{s=\beta+n}{\text{Res}}I(s,z) + \lim_{\tau \rightarrow \infty^{+}} 0. \nonumber \end{align}
We compute the residues to find,
\begin{align} \varphi(\beta,\gamma;z) &= \frac{\Gamma(\gamma-1)}{\Gamma(\beta)} z^{1-\gamma} \sum_{n=0}^{\infty} \frac{(\beta+1-\gamma)_{n}z^{n}}{(2-\gamma)_{n}n!} + \frac{\Gamma(1-\gamma)}{\Gamma(\beta+1-\gamma)} \sum_{n=0}^{\infty} \frac{(\beta)_{n}z^{n}}{(\gamma)_{n}n!}, \nonumber \end{align}
for $z \in \widetilde{\Sigma}_{-1}$ and $\widetilde{\Sigma}_{0}$ and the desired result is proved. \end{proof}

\begin{remark} \label{remark:subtle} Continuing with the issue raised in Remark \ref{remark:sub}, the fact is that integrating along the rectangle $R$ to the left of the imaginary axis is only able to produce an asymptotic result because we do not have uniform convergence with respect to $N$ in the integrals (\ref{eq:important}). This is to be expected, since we know $\varphi(\beta,\gamma;z)$ is analytic on sectors $\widetilde{\Sigma}_{-1}$ and $\widetilde{\Sigma}_{0}$, it certainly cannot be equal to a divergent $\ _{2}F_{0}$ series. However, when integrating along the rectangle $R'$ to the right of the imaginary axis we produce an equality with a linear combination of convergent series, namely this is the analytic continuation of the solutions at $z=\infty$ to $z=0$. This is shown in (\ref{eq:important2}), because the integrals here converge as $\tau\rightarrow \infty^{+}$ uniformly with respect to large $N$. \end{remark}

We conclude these computations by using Proposition \ref{proposition:kpp} to prove the formulae (\ref{eq:kummers-1})-(\ref{eq:kummerc0}) of Lemma \ref{lemma:stokes}. 

\begin{proofyyy} Recall from the definitions (\ref{eq:true2}) and (\ref{eq:true1}) of solutions,
\[\tilde{y}_{2}^{(\infty,0)}(z) = -\varphi(\beta,\gamma;z) \quad \text{and} \quad \tilde{y}_{1}^{(\infty,0)}(z) = e^{i \pi (\beta-\gamma)} e^{z} \varphi \left(\gamma-\beta, \gamma ; e^{- i \pi} z \right), \quad \quad z \in \widetilde{\Sigma}_{0}.\]
Let $\gamma_{\infty,0}$ be a curve as described at the beginning of this subsection. Proposition \ref{proposition:kp} shows how to represent the solutions of Kummer equation (\ref{eq:kummer}) around $z=\infty$ using a Mellin-Barnes integral. Due to the analyticity of this integral, as shown in the first part of the proof of Proposition \ref{proposition:kp}, Proposition \ref{proposition:kpp} provides the formula for the analytic continuation of these solutions to $z=0$. That is to say,
\begin{align} \gamma_{\infty,0} \left[\tilde{y}_{2}^{(\infty,0)}\right](z) = -\frac{\Gamma(\gamma-1)}{\Gamma(\beta)} \tilde{y}_{1}^{(0)}(z) - \frac{\Gamma(1-\gamma)}{\Gamma(\beta+1-\gamma)} \tilde{y}_{2}^{(0)}(z). \nonumber \end{align}
By manipulating the parameters and variable as follows: $\beta \mapsto \gamma-\beta$, $\gamma \mapsto \gamma$, $z \mapsto e^{i \pi}z$, we also deduce,
\begin{align} \gamma_{\infty,0}\left[\tilde{y}_{1}^{(\infty,0)}\right](z) &= e^{i \pi (\beta-\gamma)} \frac{\Gamma(\gamma-1)}{\Gamma(\gamma-\beta)} e^{-i \pi (1-\gamma)} z^{1-\gamma}e^{z} \ _{1}F_{1} \left(\begin{array}{c} 1-\beta \\ 2-\gamma \end{array} ; -z \right) \nonumber \\
& \quad \quad + e^{i \pi (\beta-\gamma)} \frac{\Gamma(1-\gamma)}{\Gamma(1-\beta)} e^{z} \ _{1}F_{1} \left(\begin{array}{c}\gamma-\beta \\ \gamma \end{array} ; -z \right). \nonumber \end{align}
After applying Kummer transformation, 
\begin{align} e^{z} \ _{1}F_{1} \left(\begin{array}{c} a \\ c \end{array} ; -z \right) \equiv \ _{1}F_{1} \left(\begin{array}{c} c-a \\ c \end{array} ;z \right), \nonumber \end{align}
we deduce the connection matrix as given in (\ref{eq:kummerc0}), namely,
\begin{align} \left( \gamma_{\infty,0} \left[\tilde{y}_{1}^{(\infty,0)}\right](z) , \ \gamma_{\infty,0} \left[\tilde{y}_{2}^{(\infty,0)}\right](z) \right) = \left(\tilde{y}_{1}^{(0)}(z) , \ \tilde{y}_{2}^{(0)}(z) \right) \ \widetilde{C}^{0 \infty}, \nonumber \end{align}
where,
\begin{align} \widetilde{C}^{0 \infty} = \left(\begin{array}{cc} e^{i \pi (\beta-1)} \frac{\Gamma(\gamma-1)}{\Gamma(\gamma-\beta)} & -\frac{\Gamma(\gamma-1)}{\Gamma(\beta)} \\ e^{i \pi (\beta-\gamma)} \frac{\Gamma(1-\gamma)}{\Gamma(1-\beta)} & -\frac{\Gamma(1-\gamma)}{\Gamma(\beta+1-\gamma)} \end{array} \right). \nonumber \end{align}

We now turn our attention to proving the formulae (\ref{eq:kummers-1}) for Stokes matrices. By Definition \ref{definition:kumst} of the Stokes matrices $\widetilde{S}_{k}$ and by the asymptotic behaviour (\ref{eq:chgasy}) of the fundamental solutions $\widetilde{Y}^{(\infty,k)}(z)$, we have, 
\[\left(\begin{array}{cc} z^{\beta-\gamma} e^{z} & 0 \\ 0 & z^{1-\beta} \end{array} \right) \widetilde{S}_{k} \left(\begin{array}{cc} z^{\gamma-\beta} e^{-z} & 0 \\ 0 & z^{\beta-1}\end{array}\right) \sim I, \quad \text{as } z \rightarrow \infty, \ \text{arg}(z) - k \pi \in \left(\frac{\pi}{2} , \frac{3\pi}{2}\right).\]
From this relation we easiy deduce that $\widetilde{S}_{-1}$ is lower triangular and $\widetilde{S}_{0}$ is upper triangular, both with unit diagonals. Denote by $\tilde{s}_{-1}$ and $\tilde{s}_{0}$ the $(2,1)$ and $(1,2)$ elements of the matrices $\widetilde{S}_{-1}$ and $\widetilde{S}_{0}$ respectively. With the knowledge of the connection matrix $\widetilde{C}^{0 \infty}$, we use the cyclic relation (\ref{eq:cycrel}) as follows,
\begin{align} &\widetilde{C}^{\infty 0} e^{2 \pi i \widetilde{\Theta}_{0}} \widetilde{C}^{0 \infty} = \left(\widetilde{S}_{-1}\right)^{-1} e^{-2 \pi i \widetilde{\Theta}_{\infty}} \left(\widetilde{S}_{0}\right)^{-1}, \nonumber \\
&\quad \Leftrightarrow \left(\begin{array}{cc} e^{2 \pi i (\beta-\gamma)} & \frac{-2 \pi i e^{- i \pi \gamma}}{\Gamma(\beta)\Gamma(\beta+1-\gamma)} \\ \frac{-2 \pi i e^{2 \pi i (\beta-\gamma)}}{\Gamma(1-\beta)\Gamma(\gamma-\beta)} & 1 - e^{2 \pi i (\beta-\gamma)} + e^{2 \pi i (1-\gamma)} \end{array} \right) \nonumber \\
&\quad \hspace{80pt} = \left(\begin{array}{cc} 1 & 0 \\ -\tilde{s}_{-1} & 0 \end{array} \right) \left(\begin{array}{cc} e^{2 \pi i (\beta-\gamma)} & 0 \\ 0 & e^{2 \pi i (1-\beta)} \end{array} \right) \left(\begin{array}{cc} 1 & -\tilde{s}_{0} \\ 0 & 1 \end{array}\right), \nonumber \\
&\quad \Leftrightarrow \left\{ \begin{matrix*}[l] \tilde{s}_{-1} = \frac{2 \pi i}{\Gamma(1-\beta)\Gamma(\gamma-\beta)} , \\ \tilde{s}_{0} = \frac{2 \pi i}{\Gamma(\beta)\Gamma(\beta+1-\gamma)}e^{i \pi(\gamma-2\beta)}, \end{matrix*} \right. \end{align}
which are indeed the Stokes multipliers found in the formulae (\ref{eq:kummers-1}) for the Stokes matrices. \end{proofyyy}

\begin{remark} If we had chosen to normalise the monodromy data of Kummer equation with respect to the fundamental solution $\widetilde{Y}^{(\infty,-1)}(z)$ then the signs of the exponents in $\widetilde{C}^{0 \infty}$ would be inverted. Furthermore, the monodromy matrix around infinity would change as $\widetilde{M}_{\infty} \mapsto \widetilde{S}_{0}^{-1} \widetilde{M}_{\infty}\widetilde{S}_{0}$. \end{remark}

\subsubsection{Gevrey Asymptotics and a result of Ramis and Martinet} 
We close this subsection about Kummer confluent hypergeometric differential equation by examining Gevrey asymptotics and stating a result of Ramis and Martinet \cite{ramis}. This also gives us the opportunity to show a contempory approach to the theory of Stokes phenomenon, which we have learned from \cite{balser, vanderput}. The contents of this additional subsection will not be necessary for our main theorems in Section \ref{sec:hgconf}, we include it for the curiosity of the reader. \\

We recall some definitions and facts regarding asymptotic theory. In the following, keep in mind that the role of the letter $k$ will mirror the concept of a linear differential equation having a pole of Poincar\'{e} rank $k$, so that for Kummer equation we are specifically concerned with $k=1$. Denote by $\mathbb{C}[[z^{-1}]]$ the field of formal series in $z^{-1}$. 
\begin{definition} \label{definition:gevreyasy} Let $f$ be a function analytic in a sector $\widetilde{\Sigma}$. We say that $f$ has the series $\widehat{f}=\sum_{n= 0}^{\infty}f_{n}z^{-n} \in \mathbb{C}[[z^{-1}]]$ as its Gevrey asymptotic expansion of order $k^{-1}$ as $z \rightarrow \infty$, $z \in \widetilde{\Sigma}$, denoted $f \simeq_{\frac{1}{k}} \widehat{f}$, if for every closed subsector $\sigma$ of $\widetilde{\Sigma}$, there exists a constant $K>0$ such that, for all $N \in \mathbb{N}$ and $z \in \sigma$,
\begin{eqnarray} \left| z^{N} \left(f(z) - \sum_{n =0}^{N-1}f_{n}z^{-n} \right) \right| \leq K^{N} \Gamma \left(1+\frac{N}{k}\right). \label{eq:gevrey} \end{eqnarray}
We denote by $\mathcal{A}_{\frac{1}{k}}(\widetilde{\Sigma})$ the set of analytic functions on $\widetilde{\Sigma}$ which have a Gevrey asymptotic expansion of order $k^{-1}$. \end{definition}	
Gevrey asymptotics is a stronger defintion than the usual one of Poincar\'{e} because it specifies how the right hand side of the inequality (\ref{eq:gevrey}) depends on $N$. In Poincar\'{e}'s definition of an asymptotic series the precise dependence on $N$ is not relevant. If we denote by $\mathcal{A}(\widetilde{\Sigma})$ the set of analytic functions on a sector $\widetilde{\Sigma}$ which admit an asymptotic expansion then we have, \begin{align} \mathcal{A}(\widetilde{\Sigma}) \supset \mathcal{A}_{1}(\widetilde{\Sigma}) \supset \mathcal{A}_{\frac{1}{2}}(\widetilde{\Sigma}) \supset \mathcal{A}_{\frac{1}{3}}(\widetilde{\Sigma}) \supset \ldots, \label{eq:recall} \end{align}
since the asymptotic expansion (\ref{eq:gammaasy}) of the Gamma function implies:
\[\frac{\Gamma \left(1+\frac{N}{k+1} \right)}{\Gamma \left(1+\frac{N}{k}\right)} \rightarrow 0 \text{ as } N \rightarrow \infty.\]

We note that, if $f \in \mathcal{A}_{\frac{1}{k}}(\widetilde{\Sigma})$, with $f \simeq_{\frac{1}{k}} \sum_{n =0}^{\infty}f_{n}z^{-n}$, then these coefficients satisfy $|f_{n}| < K^{n}\Gamma \left(1+\frac{n}{k}\right)$, for some positive constant $K$ and $n \geq 1$. To see this, we add the following inequalities: 
\begin{align} \left| f(z) - \sum_{n=0}^{N-1} f_{n}z^{-n} \right| &\leq |z|^{-N}K^{N} \Gamma \left( 1 + \frac{N}{k} \right), \nonumber \\
\left| f(z) - \sum_{n=0}^{N} f_{n}z^{-n}\right| &\leq |z|^{-N-1} K^{N+1}\Gamma\left(1+\frac{N+1}{k}\right), \nonumber \end{align}
to obtain the following inequality for $f_{N}$, 
\[\left| f_{N} \right| \leq K^{N} \Gamma \left(1+\frac{N}{k} \right) + |z|^{-1} K^{N+1} \Gamma \left(1+\frac{N+1}{k } \right),\]
from which we immediately find the claimed property by taking the limit $z \rightarrow \infty$. This motivates the following definition.
\begin{definition} We call a series $\widehat{f} = \sum_{n =0}^{\infty} f_{n}z^{-n} \in \mathbb{C}[[z]]$ a Gevrey series of order $k^{-1}$ if there exists a positive constant $K$ such that, $|f_{n}| < K^{n}\Gamma \left( 1 + \frac{n}{k}\right)$ for all $n \geq 1$. We denote by $\mathbb{C}[[z]]_{\frac{1}{k}}$ the set of all Gevrey series of order $k^{-1}$. \end{definition}
Consider the map $J : \mathcal{A}_{\frac{1}{k}}(\widetilde{\Sigma}) \rightarrow \mathbb{C}[[z]]_{\frac{1}{k}}$ which maps an analytic function $f$ on the sector $\widetilde{\Sigma}$ to its Gevrey asymptotic expansion of order $k^{-1}$. We recall the following result, see for instance \cite{balser,vanderput}. 
\begin{theorem} \label{theorem:bv} Assume $k>\frac{1}{2}$. The set $\mathcal{A}_{\frac{1}{k}}(\widetilde{\Sigma})$ is a differential algebra and the map $J$ is a homomorphism. Moreover, if the sector $\widetilde{\Sigma}$ has an opening less than $\frac{\pi}{k}$, then $J$ is surjective, otherwise, if $\widetilde{\Sigma}$ has an opening greater than $\frac{\pi}{k}$, then $J$ is injective. \end{theorem}
This remarkable theorem draws the connection between Gevrey asymptotics and Stokes phenomenon. Given a formal Gevrey series of order $k^{-1}$, this theorem shows that there is a unique analytic function on a sector of opening greater than $\frac{\pi}{k}$ which has that series as its Gevrey asymptotic expansion of order $k^{-1}$. Observe that this is exactly parallel to the theory of Stokes phenomenon: given a differential equation with a pole of Poincar\'{e} rank $k$ and a formal fundamental series solution at that point, there are unique analytic fundamental solutions on a sectors of openings greater than $\frac{\pi}{k}$ with the prescribed formal series as their asymptotic expansions. \\

Let $\varphi(\beta,\gamma;z)$ be defined as in (\ref{eq:mb}). Ramis and Martinet prove the following result.
\begin{theorem} \label{theorem3:ramis} The function $z^{a} \varphi(a,c;z)$ has $_{2}F_{0} \left(a,a+1-c;-z^{-1}\right)$ as its Gevrey asymptotic expansion of order one as $z \rightarrow \infty$ with $|\text{arg}(z)| < \frac{3\pi}{2}$. Similarly, $(-z)^{c-a} \varphi(c-a,c;-z)$ has $_{2}F_{0}\left(c-a,1-a;z^{-1}\right)$ as its Gevrey asymptotic expansion of order one with $|\text{arg}(-z)|<\frac{3\pi}{2}$. \end{theorem}
We have seen in the first part of the proof of Proposition \ref{proposition:kp} that $\varphi(a,c;z)$ and $\varphi(c-a,c;-z)$ are analytic in the sectors $\widetilde{\Sigma}_{-1}$ and $\widetilde{\Sigma}_{0}$. In particular, since these sectors have openings greater than $\pi$, Theorem \ref{theorem:bv} states that the map $J : \mathcal{A}_{1} \left( \widetilde{\Sigma}_{+} \right) \rightarrow \mathbb{C}[[z]]_{1}$ is injective. In other words, there are unique analytic functions on these sectors which have the formal series solutions,
\begin{align} z^{-a} \ _{2}F_{0} \left(a,a+1-c;-z^{-1}\right) \quad \text{and} \quad (-z)^{a-c} e^{z} \ _{2}F_{0}\left(c-a,1-a;z^{-1}\right), \label{eq:fs} \end{align}
as their Gevrey asymptotic expansions of order $1$. Since we have seen that Gevrey asymptotics imply asymptotics in the usual sense, recall (\ref{eq:recall}), this implies that such analytic functions on these sectors are in fact solutions to Kummer equation (\ref{eq:kummer}), by the uniqueness statement in Theorem \ref{theorem:kummerst}. Since the formal series solutions (\ref{eq:fs}) are clearly linearly independent, Ramis and Martinet's Theorem shows that the functions,
\[\varphi(a,c;z) \quad \text{and} \quad e^{z} \varphi(c-a,c;-z),\]
constitute a fundamental set of solutions of Kummer equation. Compared with our proof of this fact, stated as Proposition \ref{proposition:kp}, it is satisfying to deduce this from a different perspective.

\end{document}